\documentclass[11pt, reqno]{amsart}
\usepackage{amsmath, amsthm, amscd, amsfonts, amssymb, graphicx, color, mathtools}
\usepackage{mathrsfs}
\usepackage[bookmarksnumbered, colorlinks, plainpages]{hyperref}
\usepackage{enumerate}
\usepackage{setspace}
\usepackage{multicol}
\usepackage[margin=1in]{geometry}
\usepackage{comment}
\usepackage{dsfont}
\usepackage{booktabs}
\usepackage{caption}
\usepackage{subcaption}
\usepackage{pgfplots}

\usepackage[normalem]{ulem}
\newcommand\tsout{\bgroup\markoverwith{\textcolor{red}{\rule[0.5ex]{2pt}{1.4pt}}}\ULon}
\newcommand{\stkout}[1]{\ifmmode\text{\tsout{\ensuremath{#1}}}\else\tsout{#1}\fi}
\allowdisplaybreaks

\usepackage[authoryear,round,semicolon]{natbib}

\theoremstyle{definition}
\newtheorem{theorem}{Theorem}[section]
\newtheorem{lemma}[theorem]{Lemma}
\newtheorem{proposition}[theorem]{Proposition}
\newtheorem{corollary}[theorem]{Corollary}
\newtheorem{definition}[theorem]{Definition}

\newtheorem{remark}[theorem]{Remark}
\numberwithin{equation}{section}

\newcommand{\dif}{\mathrm{d}}
\newcommand{\bff}{\boldsymbol}
\newcommand{\bb}{\mathbb}

\newcommand{\dt}{\mathrm{d}t}

\newcommand{\ds}{\mathrm{d}s}

\newcommand{\dW}{\mathrm{d}W}

\newcommand{\norm}[2]{\left\|{#1}\right\|_{#2}}
\newcommand{\inpro}[2]{\left\langle#1,#2\right\rangle}

\newcommand{\abs}[1]{\left|{#1}\right|}
\newcommand{\one}{\mathds{1}}
\newcommand{\seminorm}[1]{\left\lvert\hspace{-1 pt}\left\lvert\hspace{-1 pt}\left\lvert#1\right\lvert\hspace{-1 pt}\right\lvert\hspace{-1 pt}\right\lvert}

\begin{document}
\setcounter{page}{1}

\title[Strong convergence of FEM for the stochastic Landau--Lifshitz--Bloch equation]{Strong convergence of finite element schemes for the stochastic Landau--Lifshitz--Bloch equation}

\author[Agus L. Soenjaya]{Agus L. Soenjaya}
\address{School of Mathematics and Statistics, The University of New South Wales, Sydney 2052, Australia}
\email{\textcolor[rgb]{0.00,0.00,0.84}{a.soenjaya@unsw.edu.au}}

\date{28 July 2026}

\begin{abstract}
The dynamics of magnetisation in a bounded ferromagnet in $\mathbb{R}^d$ ($d=1,2$) at high temperatures can be described by the stochastic Landau--Lifshitz--Bloch (sLLB) equation, which is a vector-valued quasilinear stochastic partial differential equation. In this paper, assuming adequate regularity of the initial data, we establish strong convergence in $L^2(\Omega)$ of several semi-implicit and implicit fully discrete finite element schemes for the sLLB equation, together with explicit convergence rates. The analysis relies on localised error estimates and new exponential moment bounds for the exact solution. As a by-product, these moment bounds yield mean-square exponential stability of solutions and uniqueness of the invariant measure in one spatial dimension under a small-noise assumption. We also sharpen existing convergence-in-probability results for the numerical schemes. Numerical experiments are presented to illustrate and support the theoretical findings.
\end{abstract}
\maketitle


\section{Introduction}

The study of micromagnetics at elevated temperatures has gained interest in recent years due to its wealth of applications. A modern implementation of this theory leads to the development of HAMR (heat-assisted magnetic recording), a state-of-the-art hard drive with unprecedented capacity and reading/writing speed~\citep{HsuVic22, Mcd12, ZhuLi13}. At high temperatures, noise-induced transitions between equilibrium states and the effects of thermal fluctuations in a ferromagnet become important, thus it is necessary to consider a stochastic model for micromagnetics to take these into account~\citep{GarChu04, Eva_etal12, MenLom20}.

The dynamics of magnetisation vector field $\bff{u}: [0,T]\times \mathscr{D}\to \bb{R}^3$ in a ferromagnet (a bounded domain) $\mathscr{D}\subset \bb{R}^d$ for temperatures above the Curie temperature, assuming the effective field is dominated by the exchange field, is commonly described by the stochastic Landau--Lifshitz--Bloch (sLLB) equation~\citep{BrzGolLe20, Eva_etal12, Gar97}:
\begin{align}\label{equ:sllb}
	\dif \bff{u}
	= 
	\left(\kappa_1 \Delta \bff{u}+ \gamma \bff{u}\times \Delta \bff{u} - \kappa_2(1+\mu \abs{\bff{u}}^2) \bff{u}\right) \dt + \left(\kappa_1 \bff{g}+ \gamma \bff{u}\times \bff{g}\right) \circ \dW,
\end{align}
with an initial condition $\bff{u}(0)=\bff{u}_0\in \bb{H}^1(\mathscr{D})$ and a homogeneous Neumann boundary condition $\partial_{\bff{\nu}} \bff{u}=\bff{0}$ on $\partial\mathscr{D}$. The physical constants $\kappa_1,\gamma,\kappa_2$, and $\mu$ are all positive, and the vector field $\bff{g}\in \bb{H}^s(\mathscr{D})$ is given for some $s\geq 2$ to be specified. The coefficient $\bff{g}$ encodes the intensity and spatial structure of the thermal fluctuations driving the magnetisation dynamics, with its magnitude determined by the system temperature through the fluctuation-dissipation relation~\citep{ChuNie20}. Here, $W$ is a real-valued Wiener process on a filtered probability space $(\Omega, \mathcal{F}, \bb{F}, \bb{P})$ with respect to the usual filtration and $\circ$ denotes the Stratonovich differential. Countable number of Wiener processes following the framework in~\citep{BrzGolLe20, GolJiaLe24} can also be considered, but we consider here only a single Brownian motion without loss of generality, just for simplicity of presentation. Note that \eqref{equ:sllb} is a vector-valued quasilinear stochastic PDE.

Existence of a martingale weak solution for~\eqref{equ:sllb} is shown in~\citep{JiaJuWan19}, while the existence of a unique pathwise strong solution for $d\leq 2$ is shown in~\citep{BrzGolLe20}.
A fully-discrete numerical scheme has been proposed in~\citep{GolJiaLe24} to solve the sLLB equation for $d\leq 2$. For $d=1$, a $\mathcal{C}^0$-conforming finite element scheme is proposed to discretise~\eqref{equ:sllb} directly, while for $d=2$, a $\mathcal{C}^1$-conforming finite element scheme is proposed to discretise a regularised version of~\eqref{equ:sllb}, namely:
\begin{align}\label{equ:reg sllb}
	\dif \bff{u}^\varepsilon
	= 
	\left(-\varepsilon\Delta^2 \bff{u}^\varepsilon + \kappa_1 \Delta \bff{u}^\varepsilon+ \gamma \bff{u}^\varepsilon \times \Delta \bff{u}^\varepsilon - \kappa_2(1+\mu \abs{\bff{u}^\varepsilon}^2) \bff{u}^\varepsilon\right) \dt + \left(\kappa_1 \bff{g}+ \gamma \bff{u}^\varepsilon\times \bff{g}\right) \circ \dW,
\end{align}
with $\varepsilon\in (0,1)$, equipped with an initial condition $\bff{u}^\varepsilon(0)=\bff{u}_0^\varepsilon\in \bb{H}^3(\mathscr{D})$ and homogeneous Neumann boundary conditions $\partial_{\bff{\nu}} \bff{u}^\varepsilon= \varepsilon \partial_{\bff{\nu}} \Delta\bff{u}^\varepsilon=\bff{0}$ on $\partial \mathscr{D}$. Such regularisation is necessary since higher order regularity results for~\eqref{equ:sllb} are not known for $d=2$. In all cases, convergence in probability with certain rates are shown for the proposed numerical methods. The strong solution of~\eqref{equ:reg sllb} is then shown to converge to that of~\eqref{equ:sllb} as $\varepsilon\to 0^+$. 

It is worth noting that \eqref{equ:reg sllb} is a special case of the stochastic Landau--Lifshitz--Baryakhtar (sLLBar) equation, another stochastic micromagnetic model at elevated temperatures analysed in~\citep{GolSoeTra24b}. We also mention that when $\bff{g}=\bff{0}$, problem~\eqref{equ:sllb} reduces to the deterministic LLB equation~\citep{Le16} for which several numerical schemes have been developed~\citep{LeSoeTra24, Soe24, Soe25}; see also~\citep{Soe25b} for results in the regime below the Curie temperature. When $\bff{g}=\bff{0}$, problem~\eqref{equ:reg sllb} reduces to a special case of the deterministic LLBar equation~\citep{SoeTra23, Soe24, WanDvo15}.

We elaborate available numerical results for \eqref{equ:sllb} and \eqref{equ:reg sllb} in more detail. Let $h$ and $k$ denote the mesh size and time step, respectively. For $d=1$, it has been shown in \citep{GolJiaLe24} that a semi-implicit $\mathcal{C}^0$-conforming finite element method applied directly to \eqref{equ:sllb} converges strongly in the $\bb{L}^2$-norm to the exact solution, up to a discrete stopping time, at rate $O\big(h^{\frac12}+k^{\frac14 -\delta}\big)$ for any $\delta>0$, assuming $\bff{u}_0,\bff{g}\in \bb{H}^2$ and $h\leq Ck$. This further implies convergence in probability with the same rate. For $d=2$, a semi-implicit $\mathcal{C}^1$-conforming method for the regularised system \eqref{equ:reg sllb} achieves a rate of $O\big(h^{\frac12}+ k^{\frac18 -\delta}\big)$, assuming $\bff{u}_0,\bff{g}\in \bb{H}^3$.
This similarly leads to a convergence in probability with the same rate. Related work~\citep{GolSoeTra24c} has also established convergence of a mixed finite element method for the sLLBar equation using a truncation approach to handle the nonlinearities.

However, under the analytic and stochastic frameworks of~\citep{GolJiaLe24}, the convergence rates in probability of the considered finite element schemes do not appear to be optimal, and strong convergence in $L^2(\Omega)$ is only established up to a discrete stopping time.
In this paper, we build on these results and extend them in several directions. First, we prove sharper convergence rates for the proposed numerical methods: the $\mathcal{C}^0$-conforming scheme attains an $\bb{L}^2$ rate of convergence of $O\big(h+k^{\frac12-\delta}\big)$, while the $\mathcal{C}^1$-conforming scheme achieves the same rate $O\big(h+k^{\frac12-\delta}\big)$, both in probability and strongly over a large sample space. Second, for the one-dimensional case ($d=1$), we propose a new implicit finite element scheme and prove an optimal convergence rate in probability with respect to a stronger norm, namely $O\big(h+k^{\frac12-\delta}\big)$ measured in the $\bb{H}^1$ norm. Finally, we establish strong convergence in $L^2(\Omega)$ with explicit, though suboptimal, rates for all the aforementioned schemes.

We note that problem \eqref{equ:sllb} is of quasilinear type, while \eqref{equ:reg sllb} is of semilinear type, both containing non-globally Lipschitz and non-monotone nonlinearities. Therefore, general results on numerical schemes for monotone stochastic PDEs from~\citep{GyoMil09, LiuQia21}, for instance, do not apply and any error analysis must proceed by ad-hoc arguments due to the nature of the nonlinearities.
Differently from~\citep{GolJiaLe24}, we make use of the localisation technique in~\citep{CarPro12, BesMil19} to derive error estimates over a large sample space. For the semi-implicit schemes in particular, we ensure that the sample space depends only on the exact solution (and not on the finite element solution) to obtain a better global rate of convergence in $L^2(\Omega)$, making use of a newly obtained exponential moment bound for the exact solution. Given the form of the nonlinearities, the estimates are nontrivial and have to be done more carefully. 
In the case $d=1$, the newly proposed implicit scheme attains an optimal rate of convergence in probability. This is achieved by proving a stronger stability result, which in turn yields a sharper error estimate in a stronger norm. Numerical experiments are presented to corroborate our theoretical results.

In summary, the main results of this paper are:
\begin{enumerate}[(i)]
    \item exponential moment bounds for $d\leq 2$ (Theorem~\ref{the:exp moment}), which imply mean-square exponential stability and uniqueness of the invariant measure for $d=1$ under sufficiently small noise;
    \item sharper convergence rates for the schemes proposed in~\citep{GolJiaLe24}, over a large sample space (Corollaries~\ref{cor:error subset} and \ref{cor:error subset 2d}) and in probability (Corollaries~\ref{cor:1d prob error} and~\ref{cor:2d prob rate});
    \item strong convergence in $L^2(\Omega)$ for these schemes with rates (Theorems~\ref{the:rate 1d} and \ref{the:rate 2d}), whose proofs rely on exponential moment bounds for the exact solution;
    \item a new implicit finite element scheme for $d=1$, where an optimal rate of convergence in probability (Corollaries~\ref{cor:error subset implicit} and \ref{cor:1d prob error implicit}) and a strong convergence is shown (Theorem~\ref{the:rate 1d implicit}).
\end{enumerate}
This paper is organised as follows:
\begin{enumerate}[(i)]
    \item Section~\ref{sec:notations} defines various notations used throughout the paper.
    \item Section~\ref{sec:exist unique} discusses analytical preliminaries and establishes exponential moment bounds and exponential stability of the solution.
    \item Sections~\ref{sec:fem 1d} and~\ref{sec:fem 2d} analyse the semi-implicit finite element schemes proposed in~\citep{GolJiaLe24} for solving \eqref{equ:sllb} and its regularised version \eqref{equ:reg sllb}, respectively.
    \item Section~\ref{sec:implicit} proposes and analyses an implicit scheme to solve \eqref{equ:sllb}.
    \item Section~\ref{sec:num exp} presents several numerical experiments to verify the theoretical results.
\end{enumerate}

\section{Notations}\label{sec:notations}

We begin by defining some notations used in this paper. Let $\mathscr{D}$ be an interval (if $d=1$) or a convex polygonal domain (if $d=2$). The function space $\bb{L}^p := \bb{L}^p(\mathscr{D}; \bb{R}^3)$ denotes the usual space of $p$-th integrable functions defined on~$\mathscr{D}$ and taking values in $\bb{R}^3$, and $\bb{W}^{k,p} := \bb{W}^{k,p}(\mathscr{D}; \bb{R}^3)$ denotes the usual Sobolev space of 
functions on $\mathscr{D} \subset \bb{R}^d$ taking values in $\bb{R}^3$. We
write $\bb{H}^k := \bb{W}^{k,2}$.
The partial derivative
$\partial/\partial x_i$ will be written by $\partial_i$ for short. The partial derivative of~$f$ with respect to time $t$ will be denoted by $\partial_t$. The operator $\Delta$ (or $\partial_{xx}$ if $d=1$) denotes the Neumann Laplacian acting on $\bb{R}^3$-valued functions with domain
\[
	\text{D}(\Delta):= \left\{ \bff{v}\in \bb{H}^2 : \frac{\partial\bff{v}}{\partial \bff{n}}=0 \text{ on } \partial\mathscr{D} \right\}.
\]

If $X$ is a Banach space, $L^p(0,T; X)$ and $W^{k,p}(0,T;X)$ denote respectively the usual Lebesgue and Sobolev spaces of {strongly measurable} functions on $(0,T)$ taking values in $X$. Similarly, the space $L^p(\Omega; X)$ denotes the space of $X$-valued {strongly measurable} random variables with finite $p$-th moment, where $(\Omega,\mathcal{F},\bb{P})$ is a probability space. The space $\mathcal{C}^0([0,T];X)$ and $\mathcal{C}^\alpha([0,T];X)$ denote respectively the space of continuous functions and the space of $\alpha$-H\"older continuous functions on $[0,T]$ taking values in $X$. For brevity, we will denote the spaces $L^p(0,T;X)$, $\mathcal{C}^0([0,T];X)$, and $\mathcal{C}^\alpha([0,T];X)$ by $L^p_T(X)$, $\mathcal{C}^0_T(X)$, and $\mathcal{C}^\alpha_T(X)$, respectively.

We denote the scalar product in a Hilbert space $H$ by $\langle \cdot, \cdot\rangle_H$ and its corresponding norm by $\|\cdot\|_H$. The expectation of a random variable $Y$ will be denoted by $\bb{E}[Y]$. We will not distinguish between the scalar product of $\bb{L}^2$ vector-valued functions taking values in $\bb{R}^3$ and the scalar product of $\bb{L}^2$ matrix-valued functions taking values in $\bb{R}^{3\times 3}$, and denote them by $\langle\cdot,\cdot\rangle$.

In various estimates, the constant $C$ in the estimate denotes a
generic constant which takes different values at different occurrences. If
the dependence of $C$ on some variables, e.g.~$R$ and $S$, is highlighted, we will write $C(R,S)$.

\begin{remark}
In this work, we restrict ourselves to one-dimensional intervals and two-dimensional convex polygonal domains, where triangulations with straight-edged elements suffice, standard elliptic regularities hold, and the relevant Sobolev embeddings are valid. For domains with smooth boundaries, one could approximate the boundary by polygons (introducing a geometric error) or use isoparametric curved elements, at the cost of more involved analysis (see, e.g.,~\citep{BarEll88}). While these approaches are relevant and interesting, we do not pursue them here for simplicity of exposition.
\end{remark}

\section{Analytical preliminaries: well-posedness, stability, and invariant measures}\label{sec:exist unique}

First, we collect known results on existence, uniqueness, and regularity of solution to \eqref{equ:sllb} and \eqref{equ:reg sllb}. We then show a new exponential moment bound in Theorem~\ref{the:exp moment}, which is an essential result needed to derive strong convergence of the finite element schemes. As corollaries, we obtain results on exponential stability and uniqueness of the invariant measure.

Rewriting \eqref{equ:sllb} and \eqref{equ:reg sllb} in It\^o form, we introduce notions of solutions that are weak in both the probabilistic and PDE senses, following \citep{BrzGolLe20, GolJiaLe24, JiaJuWan19}. We begin by recalling the definition of a solution to \eqref{equ:sllb} in the case $d=1$.
For notational simplicity, henceforth all numerical coefficients in~\eqref{equ:sllb} and \eqref{equ:reg sllb} are set equal to 1, except for $\varepsilon$.

\begin{definition}
	Let $d=1$ and $T>0$. Suppose that $\bff{u}_0\in \bb{H}^1$. A martingale weak solution of~\eqref{equ:sllb} is $(\Omega,\mathcal{F},\bb{F},\bb{P},W,\bff{u})$, where $(\Omega,\mathcal{F},\bb{F},\bb{P})$ is a filtered probability space with filtration $\bb{F}= \{\mathcal{F}_t\}_{t\in [0,T]}$ satisfying the usual conditions, $W$ is an $\bb{F}$-adapted real-valued Wiener process, and  $\bb{P}$-a.s. $\bff{u}\in L^\infty_T(\bb{H}^1)$ is a progressively measurable process satisfying for every $t\in [0,T]$, $\bb{P}$-a.s.,
	\begin{align}\label{equ:weakform sllb}
		\inpro{\bff{u}(t)}{\bff{\phi}} &=
		\inpro{\bff{u}_0}{\bff{\phi}}
		-
		\int_0^t \inpro{\partial_x \bff{u}(s)}{\partial_x \bff{\phi}} \ds 
		-
		\int_0^t \inpro{\bff{u}(s) \times \partial_x \bff{u}(s)}{\partial_x \bff{\phi}} \ds
		\nonumber\\
		&\quad
		-
		\int_0^t \inpro{\bff{u}(s)}{\bff{\phi}} \ds 
		-
		\int_0^t \inpro{\abs{\bff{u}(s)}^2 \bff{u}(s)}{\bff{\phi}} \ds 
		+
		\frac12 \int_0^t \inpro{(\bff{u}(s)\times \bff{g}) \times \bff{g}}{\bff{\phi}} \ds 
		\nonumber\\
		&\quad
		+
		\int_0^t \inpro{\bff{g}+\bff{u}(s) \times \bff{g}}{\bff{\phi}} \dW_s,
		\quad \forall \bff{\phi}\in \bb{H}^1.
	\end{align}
\end{definition}

Next, we turn to the notion of solution for~\eqref{equ:reg sllb} in the case $d=2$.

\begin{definition}
	Let $d=2$ and $T>0$. Suppose that $\bff{u}_0^\varepsilon\in \bb{H}^2$. A martingale weak solution of~\eqref{equ:reg sllb} is $(\Omega,\mathcal{F},\bb{F},\bb{P},W,\bff{u}^\varepsilon)$, where $(\Omega,\mathcal{F},\bb{F},\bb{P})$ is a filtered probability space with filtration $\bb{F}= \{\mathcal{F}_t\}_{t\in [0,T]}$ satisfying the usual conditions, $W$ is an $\bb{F}$-adapted real-valued Wiener process, and  $\bb{P}$-a.s. $\bff{u}^\varepsilon \in L^\infty_T(\bb{H}^2)$ is a progressively measurable process satisfying for every $t\in [0,T]$, $\bb{P}$-a.s.,
	\begin{align}\label{equ:weakform reg sllb}
		\inpro{\bff{u}^\varepsilon(t)}{\bff{\phi}} &=
		\inpro{\bff{u}_0^\varepsilon}{\bff{\phi}}
		-
		\varepsilon\int_0^t \inpro{\Delta \bff{u}^\varepsilon(s)}{\Delta \bff{\phi}} \ds
		-
		\int_0^t \inpro{\nabla \bff{u}^\varepsilon(s)}{\nabla \bff{\phi}} \ds 
		-
		\int_0^t \inpro{\bff{u}^\varepsilon(s) \times \nabla \bff{u}^\varepsilon(s)}{\nabla \bff{\phi}} \ds
		\nonumber\\
		&\quad
		-
		\int_0^t \inpro{\bff{u}^\varepsilon(s)}{\bff{\phi}} \ds 
		-
		\int_0^t \inpro{\abs{\bff{u}^\varepsilon(s)}^2 \bff{u}^\varepsilon(s)}{\bff{\phi}} \ds 
		+
		\frac12 \int_0^t \inpro{(\bff{u}^\varepsilon(s)\times \bff{g}) \times \bff{g}}{\bff{\phi}} \ds 
		\nonumber\\
		&\quad
		+
		\int_0^t \inpro{\bff{g}+\bff{u}^\varepsilon(s) \times \bff{g}}{\bff{\phi}} \dW_s,
		\quad \forall \bff{\phi}\in \bb{H}^2.
	\end{align}
\end{definition}

The following results on the existence, uniqueness, and regularity of pathwise strong solution to \eqref{equ:sllb} and \eqref{equ:reg sllb} are essentially obtained in~\citep{BrzGolLe20, GolJiaLe24, GolSoeTra24b}.

\begin{theorem}\label{the:regularity 1d}
Let $d=1$. Suppose that $\bff{u}_0\in \bb{H}^s$, where $s\in \{2,3,4\}$. There exists a unique pathwise solution $\bff{u}\in L^p\big(\Omega; L^\infty_T(\bb{H}^s) \cap L^2_T(\bb{H}^{s+1})\big)$ to \eqref{equ:sllb}. Furthermore, for this pathwise solution, there exists a constant $C:=C(p,T, \norm{\bff{g}}{\bb{H}^s})$ such that
\begin{align}\label{equ:regularity sllb}
	\bb{E}\left[ \norm{\bff{u}}{L^\infty_T(\bb{H}^s)}^{2p} + \norm{\bff{u}}{L^2_T(\bb{H}^{s+1})}^{2p} \right] + \bb{E}\left[\norm{\bff{u}}{\mathcal{C}^\alpha_T(\bb{H}^{s-1})}^p\right]
	\leq C,
\end{align}
for $p\in [1,\infty)$ and $\alpha\in (0,\frac12)$.
\end{theorem}

\begin{theorem}
Let $d=2$. Suppose that $\bff{u}_0^\varepsilon\in \bb{H}^3$. There exists a unique pathwise solution $\bff{u}^\varepsilon\in L^p\big(\Omega; L^\infty_T(\bb{H}^3) \cap L^2_T(\bb{H}^5)\big)$ to \eqref{equ:reg sllb}. Furthermore, for this pathwise solution, there exists a constant $C:=C(p,T, \norm{\bff{g}}{\bb{H}^3})$ such that
\begin{align}\label{equ:regularity reg sllb}
	\bb{E}\left[ \norm{{\bff{u}^\varepsilon}}{L^\infty_T(\bb{H}^3)}^{2p} + \norm{{\bff{u}^\varepsilon}}{L^2_T(\bb{H}^5)}^{2p} \right] + \bb{E}\left[\norm{{\bff{u}^\varepsilon}}{\mathcal{C}^\alpha_T(\bb{H}^2)}^p\right]
	\leq C,
\end{align}
for $p\in [1,\infty)$ and $\alpha\in (0,\frac12)$.
\end{theorem}

We now derive some exponential moment bounds for the solution of \eqref{equ:sllb} and \eqref{equ:reg sllb}. As far as we know, these bounds are new for the stochastic Landau--Lifshitz--Bloch equation and may be of independent interest.

{
\begin{theorem}\label{the:exp moment}
Let $d\leq 2$, and suppose that $\bff{u}_0\in \bb{H}^1$ and $\bff{g}\in \bb{W}^{1,4}$. Then the solution $\bff{u}$ of \eqref{equ:sllb} satisfies
\begin{align}\label{equ:E exp L2}
    &\bb{E}\left[ \exp \left( \sup_{t\in [0,T]} \alpha \norm{\bff{u}(t)}{\bb{L}^2}^2 + 2\alpha \int_0^T \norm{\nabla \bff{u}(s)}{\bb{L}^2}^2 \ds  + 2\alpha \int_0^T \norm{\bff{u}(s)}{\bb{L}^4}^4 \ds \right) \right]
	   \nonumber\\
	   &\quad\leq
	   \exp\left(\alpha_0 \norm{\bff{u}_0}{\bb{L}^2}^2 + \alpha_0 T \norm{\bff{g}}{\bb{L}^2}^2 + 3 \right),
    \end{align}
for any $\alpha\in (0,\alpha_0]$, where $\alpha_0=\big(4\norm{\bff{g}}{\bb{L}^2}^2\big)^{-1}$.

In addition, with $G_0:= \norm{\nabla\bff{g}}{\bb{L}^2}^2 + \norm{\bff{g}}{\bb{W}^{1,4}}^4$, the solution $\bff{u}$ to \eqref{equ:sllb} further satisfies
\begin{align}\label{equ:E exp H1}
	&\bb{E}\left[ \exp \left( \sup_{t\in [0,T]} \beta \norm{\nabla \bff{u}(t)}{\bb{L}^2}^2 + \beta \int_0^T \left(\norm{\nabla \bff{u}(s)}{\bb{L}^2}^2 + \norm{\Delta \bff{u}(s)}{\bb{L}^2}^2 \right) \ds  + \frac{\beta}{2} \int_0^T \norm{|\bff{u}(s)| |\nabla \bff{u}(s)|}{\bb{L}^2}^2 \ds \right) \right]
	\nonumber\\
	&\quad\leq
	\sqrt{3} \exp\left(\beta \norm{\nabla \bff{u}_0}{\bb{L}^2}^2 + 4\beta G_0 T + \frac12 \Big(\alpha_0 \norm{\bff{u}_0}{\bb{L}^2}^2 + \alpha_0 T \norm{\bff{g}}{\bb{L}^2}^2 + 3\Big) \right) 
\end{align}
for any $\beta\in (0,\beta_0]$, where $\beta_0:= \min \left\{\big(16 \norm{\bff{g}}{\bb{H}^1}^2 \big)^{-1}, \alpha_0/\big(2\norm{\bff{g}}{\bb{L}^\infty}^2 \big), 2 \alpha_0 \right\}$. 

The bounds \eqref{equ:E exp L2} and \eqref{equ:E exp H1} remain valid for the solution $\bff{u}^\varepsilon$ of the regularised problem~\eqref{equ:reg sllb}.
\end{theorem}

\begin{proof}
Recalling that all coefficients have been set equal to 1, we may write equation~\eqref{equ:sllb} in It\^o form as
\begin{equation}\label{equ:du Fu}
    \mathrm{d}\bff{u}= F(\bff{u}) \,\dt + G(\bff{u}) \,\dW,
\end{equation}
where
\begin{align*}
F(\bff{u}) &:= \Delta \bff{u}+ \bff{u}\times \Delta \bff{u} -  (1+|\bff{u}|^2) \bff{u} + \frac12 (\bff{u}\times \bff{g})\times \bff{g},
\\
G(\bff{u}) &:= \bff{g}+ \bff{u}\times \bff{g}.
\end{align*}

First, we will prove \eqref{equ:E exp L2}. Applying It\^o's lemma~\citep{Par79} to the map $\bff{v}\mapsto \norm{\bff{v}}{\bb{L}^2}^2$ gives
\begin{align*}
    \mathrm{d} \norm{\bff{u}(t)}{\bb{L}^2}^2
    =
    \left(2\inpro{\bff{u}}{F(\bff{u})} + \norm{G(\bff{u})}{\bb{L}^2}^2 \right)\dt + 2\inpro{\bff{u}}{G(\bff{u})} \dW_t.
\end{align*}
Noting that $\bff{a}\cdot (\bff{a}\times \bff{b})=0$ for any $\bff{a},\bff{b}\in \bb{R}^3$, we have
\begin{align*}
    \inpro{\bff{u}}{F(\bff{u})}
    &=
    -\norm{\nabla \bff{u}}{\bb{L}^2}^2 - \norm{\bff{u}}{\bb{L}^2}^2 - \norm{\bff{u}}{\bb{L}^4}^4 - \frac12 \norm{\bff{u}\times \bff{g}}{\bb{L}^2}^2,
    \\
    \norm{G(\bff{u})}{\bb{L}^2}^2
    &=
    \norm{\bff{g}}{\bb{L}^2}^2 + \norm{\bff{u}\times \bff{g}}{\bb{L}^2}^2.
\end{align*}
Therefore, we have
	\begin{align}\label{equ:ito u L2}
		\norm{\bff{u}(t)}{\bb{L}^2}^2
		&+
		2 \int_0^t \norm{\nabla \bff{u}(s)}{\bb{L}^2}^2 \ds 
        +
		2 \int_0^t \norm{\bff{u}(s)}{\bb{L}^2}^2 \ds 
		+
		2 \int_0^t \norm{\bff{u}(s)}{\bb{L}^4}^4 \ds 
		\nonumber\\
        &\quad
        =
		\norm{\bff{u}_0}{\bb{L}^2}^2
		+
		t\norm{\bff{g}}{\bb{L}^2}^2
		+
		2\int_0^t \inpro{\bff{g}}{\bff{u}(s)} \dW_s.
	\end{align}
Now, let $\alpha>0$ and define $R_t:= 2\alpha \int_0^t \inpro{\bff{g}}{\bff{u}(s)} \dW_s$. Multiplying \eqref{equ:ito u L2} by $\alpha$ and rearranging the terms give
\begin{align}\label{equ:alpha u L2}
    &\alpha \norm{\bff{u}(t)}{\bb{L}^2}^2
    +
    2\alpha \int_0^t \norm{\nabla \bff{u}(s)}{\bb{L}^2}^2 \ds 
    +
    \alpha \int_0^t \norm{\bff{u}(s)}{\bb{L}^2}^2 \ds 
    +
    2 \alpha \int_0^t \norm{\bff{u}(s)}{\bb{L}^4}^4 \ds 
	\nonumber\\
    &
    =
	\alpha\norm{\bff{u}_0}{\bb{L}^2}^2
	+
	\alpha t\norm{\bff{g}}{\bb{L}^2}^2
	+
    R_t
    -
	\alpha \int_0^t \norm{\bff{u}(s)}{\bb{L}^2}^2 \ds.
\end{align}
Note that $R_t$ is a martingale with quadratic variation $\langle R \rangle_t$, where
	\begin{align*}
		\langle R \rangle_t \leq 4\alpha^2 \norm{\bff{g}}{\bb{L}^2}^2 \int_0^t \norm{\bff{u}(s)}{\bb{L}^2}^2 \ds.
	\end{align*}
Hence, for any $\alpha>0$, we have
	\begin{align*}
		R_t - \alpha \int_0^t \norm{\bff{u}(s)}{\bb{L}^2}^2 \ds 
		\leq 
		R_t - \frac{\alpha_0}{\alpha} \langle R \rangle_t, \quad \text{where }\, \alpha_0:=\big(4\norm{\bff{g}}{\bb{L}^2}^2\big)^{-1}.
	\end{align*}
Now, let
	\begin{equation*}
		X_T:= \exp\left[ \sup_{t\in [0,T]} \left( R_t- \alpha \int_0^t \norm{\bff{u}(s)}{\bb{L}^2}^2 \ds \right) \right].
	\end{equation*}
Noting this, by a standard exponential martingale argument, we obtain for any $K>0$,
	\begin{align}\label{equ:P exp L2}
		\bb{P} \left[X_T\geq e^K\right] 
        &=
        \bb{P}\left[\sup_{t\in [0,T]} \left(R_t - \alpha \int_0^t \norm{\bff{u}(s)}{\bb{L}^2}^2 \ds \right) \geq K\right]
		\nonumber\\
		&\leq
		\bb{P} \left[\sup_{t\in [0,T]} \left(R_t- \frac{\alpha_0}{\alpha} \langle R \rangle_t\right) \geq K \right]
		\nonumber\\
		&\leq
		\bb{P} \left[\sup_{t\in [0,T]} \exp\left( \frac{2\alpha_0}{\alpha} R_t - \frac12 \left\langle \frac{2\alpha_0}{\alpha} R \right\rangle_t \right) \geq \exp \left(\frac{2\alpha_0 K}{\alpha}\right)\right]
		\nonumber\\
		&\leq
		\exp \left(-\frac{2\alpha_0 K}{\alpha}\right) \bb{E} \left[ \exp\left( \frac{2\alpha_0}{\alpha} R_T - \frac12 \left\langle \frac{2\alpha_0}{\alpha} R \right\rangle_T \right)\right]
		\leq
		\exp \left(-\frac{2\alpha_0 K}{\alpha}\right),
	\end{align}
	where in the last step we used the expected value of the stochastic exponential of $R$. 
	By \eqref{equ:P exp L2}, for any $\alpha\in (0,\alpha_0]$, we have $\bb{P}(X_T\geq e^K)\leq \exp\left(-2\alpha_0 K/\alpha\right) \leq e^{-2K}$ for any $K>0$. Using this inequality for any $\eta:= e^K>1$ (so $K>0$), we infer that $\bb{E}[X_T] \leq 2+ \int_1^\infty \bb{P}(X_T\geq \eta) \,\mathrm{d}\eta \leq 3$. Therefore, taking supremum over $[0,T]$, applying the exponential, and taking the expected value on \eqref{equ:alpha u L2}, we obtain \eqref{equ:E exp L2}.

    Next, we aim to prove \eqref{equ:E exp H1}. Using the notation in \eqref{equ:du Fu}, applying It\^o's lemma to the map $\bff{v}\mapsto \norm{\nabla \bff{v}}{\bb{L}^2}^2$, we obtain
    \begin{align}\label{equ:ito nabla u}
        \mathrm{d} \norm{\nabla \bff{u}(t)}{\bb{L}^2}^2
        =
        \left(-2\inpro{\Delta \bff{u}}{F(\bff{u})} + \norm{\nabla G(\bff{u})}{\bb{L}^2}^2 \right)\dt - 2\inpro{\Delta\bff{u}}{G(\bff{u})} \dW_t.
\end{align}
Note that
\begin{align*}
    \inpro{\Delta \bff{u}}{F(\bff{u})}
    &=
    \norm{\Delta \bff{u}}{\bb{L}^2}^2 + \norm{\nabla \bff{u}}{\bb{L}^2}^2 + \norm{\abs{\bff{u}} \abs{\nabla \bff{u}}}{\bb{L}^2}^2 + 2 \norm{\bff{u}\cdot \nabla \bff{u}}{\bb{L}^2}^2
    \\
    &\quad
    -
    \frac{1}{2} \inpro{\nabla \bff{u}}{\nabla \big((\bff{u}\times \bff{g})\times \bff{g}\big)}
    \\
    &=
    \norm{\Delta \bff{u}}{\bb{L}^2}^2 
    + 
    \norm{\nabla \bff{u}}{\bb{L}^2}^2 
    + 
    \norm{\abs{\bff{u}} \abs{\nabla \bff{u}}}{\bb{L}^2}^2 
    + 2\norm{\bff{u}\cdot \nabla \bff{u}}{\bb{L}^2}^2
    \\
    &\quad
    + \frac{1}{2} \norm{\nabla \bff{u}\times \bff{g}}{\bb{L}^2}^2
    -
    \frac{1}{2} \inpro{\nabla \bff{u}}{(\bff{u}\times \nabla \bff{g})\times \bff{g}}
    -
    \frac{1}{2} \inpro{\nabla \bff{u}}{(\bff{u}\times \bff{g})\times \nabla \bff{g}}.
\end{align*}
Similarly, noting that $\bff{a}\cdot (\bff{a}\times \bff{b})=0$ for any $\bff{a},\bff{b}\in \bb{R}^3$, we have 
\begin{align*}
    \norm{\nabla G(\bff{u})}{\bb{L}^2}^2
    &=
    \norm{\nabla \bff{g}}{\bb{L}^2}^2 + \norm{\nabla(\bff{u}\times \bff{g})}{\bb{L}^2}^2 + 2\inpro{\nabla \bff{g}}{\nabla\bff{u}\times \bff{g}},
    \\
    \inpro{\Delta \bff{u}}{G(\bff{u})}
    &=
    \inpro{\Delta \bff{u}}{\bff{g}}
    -
    \inpro{\nabla \bff{u}}{\bff{u}\times \nabla \bff{g}}.
\end{align*}
Substituting these expressions back into \eqref{equ:ito nabla u} and applying H\"older's and Young's inequalities, we infer that
\begin{align}\label{equ:nabla u int}
    &\norm{\nabla \bff{u}(t)}{\bb{L}^2}^2
    +
    2\int_0^t \norm{\Delta \bff{u}(s)}{\bb{L}^2}^2 \ds 
    +
    2\int_0^t \norm{\nabla \bff{u}(s)}{\bb{L}^2}^2 \ds
    \nonumber\\
	&\quad
	+
	2\int_0^t \norm{|\bff{u}(s)| |\nabla \bff{u}(s)|}{\bb{L}^2}^2 \ds 
	+
	4\int_0^t \norm{\bff{u}(s) \cdot \nabla \bff{u}(s)}{\bb{L}^2}^2 \ds 
    +
    \int_0^t \norm{\nabla \bff{u}(s)\times \bff{g}}{\bb{L}^2}^2 \ds
	\nonumber\\
	&=
    \norm{\nabla \bff{u}_0}{\bb{L}^2}^2
    +
    \int_0^t \inpro{\nabla \bff{u}(s)}{(\bff{u}(s)\times \nabla \bff{g})\times \bff{g}+ (\bff{u}(s)\times \bff{g})\times \nabla \bff{g}} \ds 
    \nonumber\\
    &\quad
    +
    t \norm{\nabla \bff{g}}{\bb{L}^2}^2
    +
    \int_0^t \norm{\nabla(\bff{u}(s)\times \bff{g})}{\bb{L}^2}^2 \ds 
    \nonumber\\
    &\quad
    +
    2\int_0^t \inpro{\nabla \bff{g}}{\nabla \bff{u}(s) \times \bff{g}} \ds 
    \nonumber\\
    &\quad
    +
    \int_0^t \big(-2\inpro{\Delta \bff{u}(s)}{\bff{g}} + 2\inpro{\nabla \bff{u}(s)}{\bff{u}(s)\times \nabla \bff{g}}\big)\, \dW_s
    \nonumber\\
    &\leq
    \norm{\nabla \bff{u}_0}{\bb{L}^2}^2
    +
    2\int_0^t \norm{\abs{\bff{u}(s)} \abs{\nabla \bff{u}(s)}}{\bb{L}^2} \norm{\bff{g}}{\bb{L}^4} \norm{\nabla \bff{g}}{\bb{L}^4} \ds
    \nonumber\\
    &\quad
    +
    t \norm{\nabla \bff{g}}{\bb{L}^2}^2
    +
    2 \int_0^t \left(\norm{\bff{g}}{\bb{L}^\infty}^2 \norm{\nabla \bff{u}(s)}{\bb{L}^2}^2 
    + \norm{\nabla \bff{g}}{\bb{L}^4}^2 \norm{\bff{u}(s)}{\bb{L}^4}^2 \right)
    \ds 
    \nonumber\\
    &\quad
    +
    2\int_0^t \norm{\nabla \bff{g}}{\bb{L}^4} \norm{\nabla \bff{u}(s)}{\bb{L}^2} \norm{\bff{g}}{\bb{L}^4} \ds
    \nonumber\\
    &\quad
    +
    \int_0^t \big(-2\inpro{\Delta \bff{u}(s)}{\bff{g}} + 2\inpro{\nabla \bff{u}(s)}{\bff{u}(s)\times \nabla \bff{g}}\big)\, \dW_s
    \nonumber\\
    &\leq
    \norm{\nabla \bff{u}_0}{\bb{L}^2}^2
    +
    t \left(\norm{\bff{g}}{\bb{L}^4}^4 + \norm{\nabla \bff{g}}{\bb{L}^4}^4\right)
    +
    \frac12 \int_0^t \norm{\abs{\bff{u}(s)} \abs{\nabla \bff{u}(s)}}{\bb{L}^2}^2 \ds 
    \nonumber\\
    &\quad
    +
    t \norm{\nabla \bff{g}}{\bb{L}^2}^2
    +
    2\norm{\bff{g}}{\bb{L}^\infty}^2 \int_0^t \norm{\nabla \bff{u}(s)}{\bb{L}^2}^2 \ds 
    +
    \frac12 \int_0^t \norm{\bff{u}(s)}{\bb{L}^4}^4 \ds 
    +
    2 t\norm{\nabla \bff{g}}{\bb{L}^4}^4  
    \nonumber\\
    &\quad
    +
    \int_0^t \norm{\nabla \bff{u}(s)}{\bb{L}^2}^2 \ds
    +
    \frac12 t \left(\norm{\bff{g}}{\bb{L}^4}^4 + \norm{\nabla \bff{g}}{\bb{L}^4}^4 \right)
    \nonumber\\
    &\quad
    +
    \int_0^t \big(2\inpro{\nabla \bff{u}(s)}{\bff{u}(s)\times \nabla \bff{g}} -2\inpro{\Delta \bff{u}(s)}{\bff{g}} \big)\, \dW_s.
\end{align}
Multiplying \eqref{equ:nabla u int} by $\beta>0$ and rearranging some terms, we obtain
\begin{align}\label{equ:alpha u H1 new}
    &\beta \norm{\nabla \bff{u}(t)}{\bb{L}^2}^2
    +
    \beta \int_0^t \norm{\Delta \bff{u}(s)}{\bb{L}^2}^2 \ds 
    +
    \beta \int_0^t \norm{\nabla \bff{u}(s)}{\bb{L}^2}^2 \ds
	+
	\frac{\beta}{2} \int_0^t \norm{|\bff{u}(s)| |\nabla \bff{u}(s)|}{\bb{L}^2}^2 \ds 
	\nonumber\\
	&\leq
    \beta \norm{\nabla \bff{u}_0}{\bb{L}^2}^2
    +
    4\beta t\norm{\bff{g}}{\bb{W}^{1,4}}^4
    +
    \beta t\norm{\nabla \bff{g}}{\bb{L}^2}^2
    +
    2 \beta\norm{\bff{g}}{\bb{L}^\infty}^2 \int_0^t \norm{\nabla \bff{u}(s)}{\bb{L}^2}^2 \ds 
    +
    \frac{\beta}{2} \int_0^t \norm{\bff{u}(s)}{\bb{L}^4}^4 \ds 
    \nonumber\\
    &\quad
    +
    M_t 
    -
    \beta \int_0^t \left(\norm{\Delta \bff{u}(s)}{\bb{L}^2}^2 + 
    \norm{|\bff{u}(s)| |\nabla \bff{u}(s)|}{\bb{L}^2}^2 \right) \ds
    \nonumber\\
    &=:V_t+M_t- \beta \int_0^t \seminorm{\bff{u}(s)}^2 \ds,
\end{align}
where $\seminorm{\bff{v}}^2:= \norm{\Delta \bff{v}}{\bb{L}^2}^2 + \norm{|\bff{v}||\nabla \bff{v}|}{\bb{L}^2}^2$ and
	\begin{align*}
    V_t &:= \beta \norm{\nabla \bff{u}_0}{\bb{L}^2}^2
    +
    4\beta t\norm{\bff{g}}{\bb{W}^{1,4}}^4
    +
    \beta t\norm{\nabla \bff{g}}{\bb{L}^2}^2
    +
    2\beta\norm{\bff{g}}{\bb{L}^\infty}^2 \int_0^t \norm{\nabla \bff{u}(s)}{\bb{L}^2}^2 \ds 
    +
    \frac{\beta}{2} \int_0^t \norm{\bff{u}(s)}{\bb{L}^4}^4 \ds,
    \\
		M_t &:= 2\beta \int_0^t \big( 
        \inpro{\nabla \bff{u}(s)}{\bff{u}(s) \times \nabla \bff{g}} - \inpro{\Delta \bff{u}(s)}{\bff{g}} \big)\, \dW_s.
	\end{align*}
    Let
	\begin{equation*}
		Y_T:= \exp\left[ \sup_{t\in [0,T]} \left( M_t- \beta \int_0^t \seminorm{\bff{u}(s)}^2 \ds \right) \right].
	\end{equation*}
Taking supremum of \eqref{equ:alpha u H1 new} over $[0,T]$, then applying the exponential and the expected value, we obtain by the Cauchy--Schwarz inequality,
    \begin{align}\label{equ:E exp}
        &\bb{E}\left[ \exp \left( \sup_{t\in [0,T]} \beta \norm{\nabla \bff{u}(t)}{\bb{L}^2}^2 + \beta \int_0^T \left(\norm{\nabla \bff{u}(s)}{\bb{L}^2}^2 + \norm{\Delta \bff{u}(s)}{\bb{L}^2}^2 \right) \ds  + \frac{\beta}{2} \int_0^T \norm{|\bff{u}(s)| |\nabla \bff{u}(s)|}{\bb{L}^2}^2 \ds \right) \right]
        \nonumber\\
        &\leq
        \bb{E} \left[ \exp\left( V_T \right)\cdot Y_T \right]
        \nonumber\\
        &\leq
        \left(\bb{E} \left[\exp(2V_T)\right]\right)^{\frac12}  \left(\bb{E}\left[Y_T^2\right]\right)^{\frac12}.
    \end{align}
    We will estimate each term on the last line.
    To this end, note that $M_t$ is a martingale with quadratic variation $\langle M \rangle_t$ such that
	\begin{align*}
		\langle M \rangle_t \leq 8\beta^2 \norm{\bff{g}}{\bb{H}^1}^2 \int_0^t \seminorm{\bff{u}(s)}^2 \ds.
	\end{align*}
    Therefore, we have for any $\beta>0$,
    \begin{align*}
		M_t - \beta \int_0^t \seminorm{\bff{u}(s)}^2 \ds 
		\leq 
		M_t - \frac{\beta_1}{\beta} \langle M \rangle_t, \quad \text{where } \beta_1:=\big(8\norm{\bff{g}}{\bb{H}^1}^2 \big)^{-1}.
	\end{align*}
	In view of the above, by repeating the argument used in \eqref{equ:P exp L2}, we obtain, for any $K>0$,
	\begin{align*}
		\bb{P}\left[ Y_T\geq e^K \right]
        &=\bb{P}\left[\sup_{t\in [0,T]} \left(M_t - \beta \int_0^t \seminorm{\bff{u}(s)}^2 \ds \right) \geq K\right]
		\nonumber\\
		&\leq
		\bb{P} \left[\sup_{t\in [0,T]} \exp\left( \frac{2\beta_1}{\beta} M_t - \frac12 \left\langle \frac{2\beta_1}{\beta} M \right\rangle_t \right) \geq \exp \left(\frac{2\beta_1 K}{\beta}\right)\right]
		\nonumber\\
		&\leq
		\exp \left(-\frac{2\beta_1 K}{\beta}\right) \bb{E} \left[ \exp\left( \frac{2\beta_1}{\beta} M_T - \frac12 \left\langle \frac{2\beta_1}{\beta} M \right\rangle_T \right)\right]
		\leq
		\exp \left(-\frac{2\beta_1 K}{\beta}\right),
	\end{align*}
	where in the last step we used the expected value of the stochastic exponential of $M$. 
	Hence, for any $\beta\in (0, \beta_1/2]$, we have
    \[
    \bb{P}(Y_T^2 \geq e^K)= \bb{P}(Y_T\geq e^{K/2}) \leq \exp\left(\frac{-\beta_1 K}{\beta}\right) \leq e^{-2K}, \quad  \forall K>0.
    \]
    Using this inequality for any $\eta:= e^K>1$ (so $K>0$), we infer that
    \begin{align}\label{equ:E Yt2}
    \bb{E}[Y_T^2] \leq 2+ \int_1^\infty \bb{P}(Y_T^2\geq \eta) \,\mathrm{d}\eta \leq 3, \quad \forall \beta\in (0, \beta_1/2].
    \end{align}
    For the term $\bb{E}\left[\exp(2V_T)\right]$ in \eqref{equ:E exp}, we use \eqref{equ:E exp L2} to infer that for any $\beta\leq \min \left\{\alpha_0/\big(2\norm{\bff{g}}{\bb{L}^\infty}^2 \big), 2 \alpha_0 \right\}$,
    \begin{align}\label{equ:E exp 2VT}
        \bb{E}\left[\exp(2V_T)\right]
        &\leq
        \exp\left(2\beta \norm{\nabla \bff{u}_0}{\bb{L}^2}^2
        +
        8\beta t\norm{\bff{g}}{\bb{W}^{1,4}}^4
        +
        2\beta t\norm{\nabla \bff{g}}{\bb{L}^2}^2\right) 
        \nonumber\\
        &\quad
        \cdot
        \bb{E}\left[ \exp\left(
        4\beta\norm{\bff{g}}{\bb{L}^\infty}^2 \int_0^t \norm{\nabla \bff{u}(s)}{\bb{L}^2}^2 \ds 
        +
        \beta \int_0^t \norm{\bff{u}(s)}{\bb{L}^4}^4 \ds\right) \right]
        \nonumber\\
        &\leq
        \exp\left(2\beta \norm{\nabla \bff{u}_0}{\bb{L}^2}^2 + 8\beta G_0 T\right) \cdot
        \exp\left(\alpha_0 \norm{\bff{u}_0}{\bb{L}^2}^2 + \alpha_0 T \norm{\bff{g}}{\bb{L}^2}^2 + 3 \right).
    \end{align}
    Substitution of \eqref{equ:E Yt2} and \eqref{equ:E exp 2VT} into \eqref{equ:E exp} yields \eqref{equ:E exp H1}, valid for any $\beta$ for which the bounds \eqref{equ:E Yt2} and \eqref{equ:E exp 2VT} hold. This completes the proof of the theorem.
\end{proof}

Consequently, we obtain the following mean-square exponential stability result in the $\bb{L}^2$ norm for $d=1$, provided the noise coefficient is sufficiently small.

\begin{proposition}\label{pro:eps zero}
Let $d=1$. Suppose that $\bff{u}$ and $\bff{v}$ are two solutions to \eqref{equ:sllb} with initial data $\bff{u}_0,\bff{v}_0\in \bb{H}^1$, respectively, driven by the same Wiener process. There exists a positive constant $\epsilon_0$ such that, if $\norm{\bff{g}}{\bb{W}^{1,4}}\leq \epsilon_0$, then
\begin{align*}
    \bb{E}\left[\norm{\bff{u}(t)-\bff{v}(t)}{\bb{L}^2}^2 \right] \leq Ce^{-\lambda t} \norm{\bff{u}_0-\bff{v}_0}{\bb{L}^2}^2
\end{align*}
for some constants $C,\lambda>0$ independent of $t$.
\end{proposition}

\begin{proof}
Suppose that $\bff{u}$ and $\bff{v}$ are pathwise solutions to \eqref{equ:sllb} with initial data $\bff{u}_0\in \bb{H}^1$ and $\bff{v}_0\in \bb{H}^1$, respectively. For any $t>0$, let $\bff{w}(t):= \bff{u}(t)-\bff{v}(t)$ and $\bff{w}_0:= \bff{u}_0-\bff{v}_0$. Note that we have
\begin{align*}
    \mathrm{d}\bff{w}= \left[\bff{w}\times \Delta \bff{u}+ \bff{v}\times \Delta \bff{w}+ \Delta \bff{w}-\bff{w}- \big( (\bff{u}+\bff{v})\cdot \bff{w}\big) \bff{u}- \abs{\bff{v}}^2 \bff{w} + \frac12 (\bff{w}\times \bff{g})\times \bff{g}\right] \dt + (\bff{w}\times \bff{g})\, \dW.
\end{align*}
Applying It\^o's lemma to the map $\bff{w}\mapsto \norm{\bff{w}}{\bb{L}^2}^2$, we obtain
\begin{align*}
    \mathrm{d} \norm{\bff{w}(t)}{\bb{L}^2}^2
    =
    \left(2\inpro{\nabla \bff{v}\times \bff{w}}{\nabla \bff{w}}
    - 2\norm{\nabla \bff{w}}{\bb{L}^2}^2
    - 2\norm{\bff{w}}{\bb{L}^2}^2
    - 2\inpro{\abs{\bff{u}}^2 \bff{u}-\abs{\bff{v}}^2 \bff{v}}{\bff{w}} \right) \dt.
\end{align*}
Here, the It\^o correction term cancels exactly due to the skew-symmetry of the noise operator, while the stochastic integral vanishes since $\bff{w}\cdot (\bff{w}\times\bff{g})=0$ pointwise. By Young's inequality and the monotonicity property of the map $\bff{v}\mapsto \abs{\bff{v}}^2 \bff{v}$, we have
\begin{align*}
    \norm{\bff{w}(t)}{\bb{L}^2}^2
    +
    2\int_0^t \norm{\bff{w}(s)}{\bb{L}^2}^2 \ds
    +
    2\int_0^t \norm{\nabla \bff{w}(s)}{\bb{L}^2}^2 \ds
    &\leq
    \norm{\bff{w}_0}{\bb{L}^2}^2
    +
    \int_0^t \norm{\nabla \bff{w}(s)}{\bb{L}^2}^2 \ds 
    \\
    &\quad
    +
    \int_0^t \norm{\nabla \bff{v}(s)}{\bb{L}^\infty}^2 \norm{\bff{w}(s)}{\bb{L}^2}^2 \ds.
\end{align*}
By the Gronwall lemma and the embedding $\bb{H}^1\hookrightarrow \bb{L}^\infty$ in one spatial dimension, we infer that
\begin{align}\label{equ:w L2}
    \norm{\bff{w}(t)}{\bb{L}^2}^2
    \leq
    \norm{\bff{w}_0}{\bb{L}^2}^2 \,
    \exp\left(-2t+ \int_0^t \norm{\nabla \bff{v}(s)}{\bb{H}^1}^2 \ds \right).
\end{align}
We will apply inequality~\eqref{equ:E exp H1}, noting that in this case $\alpha_0 \norm{\bff{g}}{\bb{L}^2}^2=\frac14$.
Taking expectations in \eqref{equ:w L2}, and assuming that the noise coefficient $\bff{g}$ is sufficiently small such that $\beta_0\geq 1$, we obtain
\begin{align*}
    \bb{E}\left[\norm{\bff{w}(t)}{\bb{L}^2}^2\right]
    &\leq
    \norm{\bff{w}_0}{\bb{L}^2}^2 e^{-2t} \, \bb{E}\left[\exp\left(\int_0^t \norm{\nabla \bff{v}(s)}{\bb{H}^1}^2 \ds \right) \right]
    \\
    &\leq
    C \norm{\bff{w}_0}{\bb{L}^2}^2 \, \exp\left(-2t+4 G_0 t+ \frac18 t\right),
\end{align*}
where the constant $C$ is independent of $t$.
Consequently, if in addition $G_0< 15/32$, then there exists $\lambda>0$ such that
\begin{align*}
    \bb{E}\left[\norm{\bff{w}(t)}{\bb{L}^2}^2\right]
    &\leq
    C \norm{\bff{w}_0}{\bb{L}^2}^2 e^{-\lambda t}.
\end{align*}
We note that a smallness condition on $\norm{\bff{g}}{\bb{W}^{1,4}}$ is sufficient to guarantee both $\beta_0\geq 1$ and $G_0< 15/32$. This completes the proof of the proposition.
\end{proof}

As a corollary of Proposition~\ref{pro:eps zero}, we obtain uniqueness of the invariant measure for $d=1$ under a sufficiently small noise coefficient; existence was already shown in~\citep{BrzGolLe20}.

\begin{corollary}\label{cor:inv meas}
Let $d=1$ and assume that $\norm{\bff{g}}{\bb{W}^{1,4}}\leq \epsilon_0$, where $\epsilon_0>0$ is as in Proposition~\ref{pro:eps zero}. Then the Markov semigroup associated with \eqref{equ:sllb} admits a unique invariant measure.
\end{corollary}

}


\section{A semi-implicit finite element approximation of the sLLB equation for $d=1$}\label{sec:fem 1d}

In this section, we consider a finite element approximation of the sLLB equation~\eqref{equ:sllb} for $d=1$. Let $\mathcal{T}_h$ be a regular partition of $\mathscr{D}\subset \bb{R}$ into sub-intervals with maximal width $h$. Let $T>0$ be fixed and $k$ be the time-step size. Define $\bb{V}_h\subset \bb{W}^{1,\infty}$ to be the Lagrange finite element space 
\begin{equation}\label{equ:Vh}
	\bb{V}_h := \{\bff{\phi}\in C(\overline{\mathscr{D}}; \bb{R}^3): \bff{\phi}|_K \in \bb{P}_1(K;\bb{R}^3), \; \forall K \in \mathcal{T}_h\},
\end{equation}
where $\bb{P}_1(K; \bb{R}^3)$ denotes the space of linear polynomials on $K$ taking values in $\bb{R}^3$.

We define the orthogonal projection operator $\Pi_h: \bb{L}^2 \to \bb{V}_h$ such that
\begin{align}\label{equ:orth proj}
	\inpro{\Pi_h \bff{v}-\bff{v}}{\bff{\chi}}=0,
	\quad
	\forall \bff{\chi}\in \bb{V}_h.
\end{align}
The operator $\Pi_h$ has the following approximation property: there exists $C>0$ such that
\begin{align}\label{equ:proj approx}
	\norm{\bff{v}- \Pi_h\bff{v}}{\bb{L}^2}
	+
	h \norm{\nabla( \bff{v}-\Pi_h\bff{v})}{\bb{L}^2}
	\leq
	Ch^2 \norm{\bff{v}}{\bb{H}^2}, \quad \forall \bff{v}\in \bb{H}^2.
\end{align}

A semi-implicit finite element scheme proposed in~\citep{GolJiaLe24} for the sLLB equation when $d=1$ is as follows. Let $\bff{u}_h^n$ be the approximation in $\bb{V}_h$ of $\bff{u}(t)$ at time $t=t_n:=nk\in [0,T]$, where $n=0,1,2,\ldots, N$ and $N=\lfloor T/k \rfloor$. We start with $\bff{u}_h^0= \Pi_h \bff{u}(0) \in \bb{V}_h$. Given $\bff{u}_h^{n-1} \in \bb{V}_h$, we find $\mathcal{F}_{t_n}$-adapted and $\bb{V}_h$-valued random variable $\bff{u}_h^n$ satisfying $\bb{P}$-a.s.,
\begin{align}\label{equ:euler}
		\inpro{\bff{u}_h^n-\bff{u}_h^{n-1}}{\bff{\phi}_h}
		&=
		-k \inpro{\partial_x \bff{u}_h^n}{\partial_x \bff{\phi}_h}
		-
		k \inpro{\bff{u}_h^{n-1}\times \partial_x \bff{u}_h^n}{\partial_x \bff{\phi}_h}
		-
		k \inpro{\bff{u}_h^n}{\bff{\phi}_h}
		-
		k \inpro{\abs{\bff{u}_h^{n-1}}^2 \bff{u}_h^n}{\bff{\phi}_h}
		\nonumber\\
		&\quad
		+
		\frac{k}{2} \inpro{\left(\bff{u}_h^{n-1}\times \bff{g}\right) \times \bff{g}}{\bff{\phi}_h}
		+
		\inpro{\bff{g}+\bff{u}_h^{n-1}\times \bff{g}}{\bff{\phi}_h} \overline{\Delta} W^n,
\end{align}
for all $\bff{\phi}_h \in \bb{V}_h$. Here, $\overline{\Delta}W^n:= W(t_n)-W(t_{n-1}) \sim \mathcal{N}(0,k)$.

The scheme~\eqref{equ:euler} is linear and well-posed by the Lax--Milgram lemma. Moreover, it satisfies the following stability property.

\begin{lemma}\label{lem:stab L2}
There exists a positive constant $C$ independent of $n$, $h$, and $k$, such that for any $p\in [1,\infty)$,
\begin{align*}
	&\bb{E} \left[ \max_{l\leq n} \norm{\bff{u}_h^l}{\bb{L}^2}^{2p} \right]
	+
	\bb{E}\left[ \left(\sum_{j=1}^n \norm{\bff{u}_h^j-\bff{u}_h^{j-1}}{\bb{L}^2}^2\right)^p \right]
	+
	\bb{E}\left[ \left(k \sum_{j=1}^n \norm{\bff{u}_h^j}{\bb{H}^1}^2 \right)^p \right] 
	\leq
	C.
\end{align*}
\end{lemma}

\begin{proof}
This is obtained by taking $\bff{\phi}_h= \bff{u}_h^n$ in~\eqref{equ:euler} and applying Jensen's inequality, noting that $(\bff{a}\times \bff{b})\cdot \bff{b}=0$ for any $\bff{a},\bff{b}\in \bb{R}^3$; see~\citep[Lemma~4.2]{GolJiaLe24}.
\end{proof}

To facilitate the proof of the error estimate, we decompose the error of the numerical method at time $t_n$, $n=0,1,\ldots,N$, as:
\begin{align}\label{equ:split u}
    \bff{u}(t_n)- \bff{u}_h^n
    &= 
    \left(\bff{u}(t_n)-\Pi_h \bff{u}(t_n)\right)
    +
    \left(\Pi_h \bff{u}(t_n)- \bff{u}_h^n\right)
    =:
    \bff{\rho}^n+ \bff{\theta}^n,
\end{align}
As such by \eqref{equ:orth proj},
\begin{equation}\label{equ:proj zero}
\inpro{\bff{\rho}^n}{\bff{\phi}_h}=0, \quad \forall \bff{\phi}_h\in \bb{V}_h.
\end{equation}
Furthermore, define a sequence of subsets of $\Omega$ which depend on $\kappa$ and $m$:
\begin{align}\label{equ:Omega k m}
    \Omega_{\kappa,m}:= \left\{\omega\in \Omega: \max_{t\leq t_{m} \wedge T} \norm{\bff{u}(t)}{\bb{H}^1}^2 
   	\leq \kappa \right\},
\end{align}
where $\kappa>1$ is to be specified. Note that for any $\kappa>1$ and $m\in \bb{N}$, we have $\Omega_{\kappa,m} \supset \Omega_{\kappa,m+1}$. Furthermore, as in~\citep{CarPro12}, for any time-discrete random variable $\bff{v}^n$,
\begin{align}\label{equ:1 vn vn1}
    \bb{E}\left[\max_{m\leq n} \sum_{\ell=1}^m \one_{\Omega_{\kappa,\ell-1}} \inpro{\bff{v}^\ell-\bff{v}^{\ell-1}}{\bff{v}^\ell}\right]
    &\geq
    \frac12 \bb{E}\left[\max_{m\leq n} \left( \one_{\Omega_{\kappa,m-1}} \norm{\bff{v}^m}{\bb{L}^2}^2 \right) \right] 
    \nonumber\\
    &\quad
	+
    \frac12  \sum_{\ell=1}^n \bb{E} \left[\one_{\Omega_{\kappa,\ell-1}} \norm{\bff{v}^\ell-\bff{v}^{\ell-1}}{\bb{L}^2}^2 \right].
\end{align}

We are now ready to prove an auxiliary error estimate. 

\begin{proposition}\label{pro:E theta n L2}
Let $d=1$ and $\bff{g}\in \bb{H}^2$. Suppose that $\bff{u}$ is the pathwise solution to \eqref{equ:sllb} with initial data $\bff{u}_0\in \bb{H}^2$, and let $\{\bff{u}_h^n\}_n$ be a sequence of random variables solving~\eqref{equ:euler} with $\bff{u}_h^0=\Pi_h \bff{u}_0$. Let $\Omega_{\kappa,m}$ and $\bff{\theta}^n$ be as defined in~\eqref{equ:Omega k m} and~\eqref{equ:split u}, respectively. Then for $n\in \{1,2,\ldots, N\}$, we have
\begin{align}\label{equ:C tilde}
    \bb{E}\left[\max_{m\leq n} \left( \one_{\Omega_{\kappa,m-1}} \norm{\bff{\theta}^m}{\bb{L}^2}^2 \right) \right]  
    +
    k  \sum_{\ell=1}^n \bb{E} \left[ \one_{\Omega_{\kappa,\ell-1}} \norm{\partial_x \bff{\theta}^\ell}{\bb{L}^2}^2 \right]
    &\leq
    \widetilde{C} e^{\widetilde{C} \kappa^2} \left(h^2+k^{2\alpha}\right),
\end{align}
for any $\alpha\in (0,\frac12)$. The constant $\widetilde{C}$ depends on $T$, but is independent of $n$, $h$, $k$, and $\kappa$.
\end{proposition}

\begin{proof}
First, we subtract \eqref{equ:weakform sllb} at time $t_{\ell-1}$ from the same equation at time $t_\ell$. Subtracting \eqref{equ:euler} from the resulting equation and noting~\eqref{equ:split u} and \eqref{equ:proj zero}, we have for any $\bff{\phi}_h\in \bb{V}_h$,
\begin{align}
\label{equ:theta n theta n1}
    &\inpro{\bff{\theta}^\ell-\bff{\theta}^{\ell-1}}{\bff{\phi}_h}
    +
    \int_{t_{\ell-1}}^{t_\ell} \inpro{\partial_x \bff{u}(s)-\partial_x \bff{u}_h^\ell}{\partial_x \bff{\phi}_h} \ds
    +
    \int_{t_{\ell-1}}^{t_\ell} \inpro{\bff{u}(s)-\bff{u}_h^\ell}{\bff{\phi}_h} \ds
    \nonumber\\
    &=
    -
    \int_{t_{\ell-1}}^{t_\ell} \inpro{\big(\bff{u}(s)-\bff{u}_h^{\ell-1} \big) \times \partial_x \bff{u}(s)}{\partial_x \bff{\phi}_h} \ds
    -
    \int_{t_{\ell-1}}^{t_\ell} \inpro{\bff{u}_h^{\ell-1} \times \big(\partial_x \bff{u}(s)- \partial_x \bff{u}_h^\ell\big)}{\partial_x \bff{\phi}_h} \ds 
    \nonumber\\
    &\quad
    -
    \int_{t_{\ell-1}}^{t_\ell} \inpro{\big(\abs{\bff{u}(s)}^2 - |\bff{u}_h^{\ell-1}|^2\big) \bff{u}(s)}{\bff{\phi}_h} \ds 
    -
    \int_{t_{\ell-1}}^{t_\ell} \inpro{|\bff{u}_h^{\ell-1}|^2 \big(\bff{u}(s)- \bff{u}_h^\ell \big)}{\bff{\phi}_h} \ds
    \nonumber\\
    &\quad
    +
    \frac12 \int_{t_{\ell-1}}^{t_\ell} \inpro{\big((\bff{u}(s)-\bff{u}_h^{\ell-1})\times \bff{g}\big)\times \bff{g}}{\bff{\phi}_h} \ds
    +
    \int_{t_{\ell-1}}^{t_\ell} \inpro{\big(\bff{u}(s)-\bff{u}_h^{\ell-1}\big)\times \bff{g}}{\bff{\phi}_h} \dW_s.
\end{align}
We now put $\bff{\phi}_h=\bff{\theta}^\ell$ in~\eqref{equ:theta n theta n1} and multiply the resulting equations by $\one_{\Omega_{\kappa,\ell-1}}$, where the set $\Omega_{\kappa,m}$ was defined in~\eqref{equ:Omega k m}. We then sum the resulting expression over $\ell\in \{1,2,\ldots,m\}$, take the maximum over $m\leq n$, and apply the expectation value. Note that using~\eqref{equ:split u}, we can write
\begin{align}\label{equ:u2 u exp}
	&\big(\abs{\bff{u}(s)}^2 - |\bff{u}_h^{\ell-1}|^2\big) \bff{u}(s)
	+
	|\bff{u}_h^{\ell-1}|^2 \big({\bff{u}(s)-\bff{u}_h^\ell} \big)
	\nonumber\\
	&=
	\big(\bff{u}(s)+\bff{u}_h^{\ell-1}\big) \cdot \big(\bff{u}(s)-\bff{u}(t_{\ell-1})+ \bff{\rho}^{\ell-1} + \bff{\theta}^{\ell-1}\big) \bff{u}(s)
	+
	|\bff{u}_h^{\ell-1}|^2 \big(\bff{u}(s)-\bff{u}(t_\ell)+ \bff{\theta}^\ell + \bff{\rho}^\ell\big).
\end{align}
Applying~\eqref{equ:split u}, \eqref{equ:proj zero}, \eqref{equ:1 vn vn1}, and~\eqref{equ:u2 u exp}, and rearranging the terms, we obtain
\begin{align}\label{equ:12 theta n}
    &\frac12 \bb{E}\left[\max_{m\leq n} \left( \one_{\Omega_{\kappa,m-1}} \norm{\bff{\theta}^m}{\bb{L}^2}^2 \right) \right] 
    +
    \frac12  \sum_{\ell=1}^n \bb{E} \left[\one_{\Omega_{\kappa,\ell-1}} \norm{\bff{\theta}^\ell-\bff{\theta}^{\ell-1}}{\bb{L}^2}^2 \right]
    \nonumber\\
    &\quad
    +
    k
    \bb{E}\left[\sum_{\ell=1}^n \one_{\Omega_{\kappa,\ell-1}} \norm{\partial_x \bff{\theta}^\ell}{\bb{L}^2}^2 \right]
    +
    k
    \bb{E}\left[\sum_{\ell=1}^n \one_{\Omega_{\kappa,\ell-1}} \norm{\bff{\theta}^\ell}{\bb{L}^2}^2 \right]
    +
    k
    \bb{E}\left[\sum_{\ell=1}^n \one_{\Omega_{\kappa,\ell-1}} \norm{\abs{\bff{u}_h^{\ell-1}} \abs{\bff{\theta}^\ell}}{\bb{L}^2}^2 \right]
    \nonumber\\
    &\leq
    \frac12 \bb{E} \left[\norm{\bff{\theta}^0}{\bb{L}^2}^2 \right]
    -
    \bb{E}\left[\max_{m\leq n} \sum_{\ell=1}^m \one_{\Omega_{\kappa,\ell-1}} \int_{t_{\ell-1}}^{t_\ell} \inpro{\partial_x \bff{u}(s)- \partial_x \bff{u}(t_\ell) + \partial_x \bff{\rho}^\ell}{\partial_x \bff{\theta}^\ell} \ds\right]
    \nonumber\\
    &\quad
    -
    \bb{E}\left[\max_{m\leq n} \sum_{\ell=1}^m \one_{\Omega_{\kappa,\ell-1}} \int_{t_{\ell-1}}^{t_\ell} \inpro{\bff{u}(s)- \bff{u}(t_\ell)}{\bff{\theta}^\ell} \ds\right]
    \nonumber\\
    &\quad
    -
    \bb{E}\left[\max_{m\leq n} \sum_{\ell=1}^m \one_{\Omega_{\kappa,\ell-1}} \int_{t_{\ell-1}}^{t_\ell} \inpro{\big(\bff{u}(s)-\bff{u}(t_{\ell-1})+ \bff{\rho}^{\ell-1}+ \bff{\theta}^{\ell-1}\big) \times \partial_x \bff{u}(s)}{\partial_x \bff{\theta}^\ell} \ds\right]
    \nonumber\\
    &\quad
    -
    \bb{E}\left[\max_{m\leq n} \sum_{\ell=1}^m \one_{\Omega_{\kappa,\ell-1}} \int_{t_{\ell-1}}^{t_\ell} \inpro{\bff{u}_h^{\ell-1}\times \big(\partial_x \bff{u}(s)-\partial_x \bff{u}(t_\ell)+ \partial_x \bff{\rho}^\ell\big)}{\partial_x \bff{\theta}^\ell} \ds \right]
    \nonumber\\
    &\quad
    -
    \bb{E}\left[\max_{m\leq n} \sum_{\ell=1}^m \one_{\Omega_{\kappa,\ell-1}} \int_{t_{\ell-1}}^{t_\ell} \inpro{\big(\bff{u}(s)+\bff{u}_h^{\ell-1}\big) \cdot \big(\bff{u}(s)-\bff{u}(t_{\ell-1})+ \bff{\rho}^{\ell-1} + \bff{\theta}^{\ell-1}\big) \bff{u}(s)}{\bff{\theta}^\ell} \ds \right]
    \nonumber\\
    &\quad
    -
   	\bb{E}\left[\max_{m\leq n} \sum_{\ell=1}^m \one_{\Omega_{\kappa,\ell-1}} \int_{t_{\ell-1}}^{t_\ell} \inpro{|\bff{u}_h^{\ell-1}|^2 \bff{\rho}^\ell}{\bff{\theta}^\ell} \ds \right]
    \nonumber\\
    &\quad
    -
   \bb{E}\left[\max_{m\leq n} \sum_{\ell=1}^m \one_{\Omega_{\kappa,\ell-1}} \int_{t_{\ell-1}}^{t_\ell} \inpro{|\bff{u}_h^{\ell-1}|^2 \big(\bff{u}(s)-\bff{u}(t_\ell)\big)}{\bff{\theta}^\ell} \ds \right]
    \nonumber\\
    &\quad
    +
    \frac12 \bb{E}\left[\max_{m\leq n} \sum_{\ell=1}^m \one_{\Omega_{\kappa,\ell-1}} \int_{t_{\ell-1}}^{t_\ell} \inpro{\big( (\bff{u}(s)-\bff{u}(t_{\ell-1})+ \bff{\rho}^{\ell-1}+ \bff{\theta}^{\ell-1})\times \bff{g}\big)\times \bff{g}}{\bff{\theta}^\ell} \ds \right]
    \nonumber\\
    &\quad
    +
    \bb{E}\left[\max_{m\leq n} \sum_{\ell=1}^m \one_{\Omega_{\kappa,\ell-1}} \int_{t_{\ell-1}}^{t_\ell} \inpro{(\bff{u}(s)-\bff{u}(t_{\ell-1})+ \bff{\rho}^{\ell-1}+ \bff{\theta}^{\ell-1})\times \bff{g}}{\bff{\theta}^\ell} \dW_s \right]
    \nonumber\\
    &=: \frac12 \bb{E} \left[\norm{\bff{\theta}^0}{\bb{L}^2}^2 \right]+ I_1+I_2+\cdots+ I_9.
\end{align}
We will estimate each term on the last line. We recall that $\bff{u}$ has regularity given by~\eqref{equ:regularity sllb}, which will be used in the estimates without further mention. Let $\delta>0$ be a constant to be determined later. For the term $I_1$, by H\"older continuity in time of $\bff{u}$, Young's inequality, and \eqref{equ:proj approx} we have
\begin{align*}
	\abs{I_1} 
	\leq 
	Ch^2+ Ck^{2\alpha} + \delta k \bb{E} \left[\sum_{\ell=1}^n \one_{\Omega_{\kappa,\ell-1}}  \norm{\bff{\theta}^\ell}{\bb{H}^1}^2 \right].
\end{align*}
Similarly, for the term $I_2$, we have
\begin{align*}
    \abs{I_2}
    \leq
    Ck^{2\alpha} + \delta k \bb{E} \left[\sum_{\ell=1}^n \one_{\Omega_{\kappa,\ell-1}}  \norm{\bff{\theta}^\ell}{\bb{L}^2}^2 \right].
\end{align*}
For the term $I_3$, by H\"older's inequality,
\begin{align}\label{equ:I3 holder}
	\abs{I_3}
	&\leq
	\bb{E}\left[ \sum_{\ell=1}^n \one_{\Omega_{\kappa,\ell-1}} \int_{t_{\ell-1}}^{t_\ell} \norm{\bff{u}(s)-\bff{u}(t_{\ell-1})}{\bb{L}^\infty} \norm{\partial_x \bff{u}(s)}{\bb{L}^2} \norm{\partial_x \bff{\theta}^\ell}{\bb{L}^2} \ds \right]
    \nonumber\\
    &\quad
    +
    \bb{E}\left[ \sum_{\ell=1}^n \one_{\Omega_{\kappa,\ell-1}} \int_{t_{\ell-1}}^{t_\ell} \norm{\bff{\rho}^{\ell-1}}{\bb{L}^\infty} \norm{\partial_x \bff{u}(s)}{\bb{L}^2} \norm{\partial_x \bff{\theta}^\ell}{\bb{L}^2} \ds \right]
    \nonumber\\
    &\quad
    +
    \bb{E}\left[ \sum_{\ell=1}^n \one_{\Omega_{\kappa,\ell-1}} \int_{t_{\ell-1}}^{t_\ell} \norm{\bff{\theta}^{\ell-1}}{\bb{L}^\infty}  \norm{\partial_x \bff{u}(s)}{\bb{L}^2} \norm{\partial_x \bff{\theta}^\ell}{\bb{L}^2} \ds \right]
    \nonumber\\
    &=: 
    I_{3a}+I_{3b}+I_{3c}.
\end{align}
The terms $I_{3a}$ and $I_{3b}$ can be estimated by using H\"older continuity in time of $\bff{u}$, the Sobolev embedding $\bb{H}^1\hookrightarrow \bb{L}^\infty$ when $d=1$, Young's inequality, and \eqref{equ:proj approx}:
\begin{align*}
    I_{3a}
    &\leq
    Ck^{2\alpha} \bb{E}\left[\norm{\bff{u}}{\mathcal{C}^\alpha_T(\bb{H}^1)}^4 + \norm{\partial_x \bff{u}}{L^\infty_T(\bb{L}^2)}^4 \right]
    +
	\delta k \bb{E}\left[ \sum_{\ell=1}^n \one_{\Omega_{\kappa,\ell-1}} \norm{\partial_x \bff{\theta}^\ell}{\bb{L}^2}^2 \right],
    \\
    I_{3b}
    &\leq
    Ch^2 \bb{E}\left[\norm{\bff{u}}{L^\infty_T(\bb{H}^2)}^4\right]
	+
	\delta k \bb{E}\left[ \sum_{\ell=1}^n \one_{\Omega_{\kappa,\ell-1}} \norm{\partial_x \bff{\theta}^\ell}{\bb{L}^2}^2 \right].
\end{align*}
With the aim of estimating $I_{3c}$ and using \eqref{equ:Omega k m}, we write $\partial_x \bff{u}(s)= \big(\partial_x \bff{u}(s)-\partial_x \bff{u}(t_{\ell-1})\big) + \partial_x \bff{u}(t_{\ell-1})$. By the triangle inequality, employing H\"older continuity in time for $\bff{u}$, we have for $s\in [t_{\ell-1},t_\ell]$,
\begin{align*}
    \norm{\partial_x \bff{u}(s)}{\bb{L}^2}
    \leq
   k^\alpha \norm{\bff{u}}{\mathcal{C}^\alpha_T(\bb{H}^1)} + \norm{\partial_x \bff{u}(t_{\ell-1})}{\bb{L}^2}.
\end{align*}
By H\"older's and Agmon's inequalities, we then obtain
\begin{align*}
    \abs{I_{3c}}
    &\leq
    Ck^{1+\alpha} \bb{E}\left[ \sum_{\ell=1}^n \one_{\Omega_{\kappa,\ell-1}} 
    \norm{\bff{\theta}^{\ell-1}}{\bb{L}^\infty}
    \norm{\bff{u}}{\mathcal{C}^\alpha_T(\bb{H}^1)} \norm{\partial_x \bff{\theta}^\ell}{\bb{L}^2} \right]
        \\
	&\quad
    +
    Ck \bb{E}\left[ \sum_{\ell=1}^n \one_{\Omega_{\kappa,\ell-1}} 
    \norm{\bff{\theta}^{\ell-1}}{\bb{L}^2}^{\frac12} \norm{\bff{\theta}^{\ell-1}}{\bb{H}^1}^{\frac12} 
    \norm{\partial_x \bff{u}(t_{\ell-1})}{\bb{L}^2} \norm{\partial_x \bff{\theta}^\ell}{\bb{L}^2} \right]
    \\
    &=: I_{3c}'+ I_{3c}''.
\end{align*}
Now, we need to estimate the last two terms, $I_{3c}'$ and $I_{3c}''$. For $I_{3c}'$, by Young's inequality and the embedding $\bb{H}^1\hookrightarrow \bb{L}^\infty$,
\begin{align*}
    \abs{I_{3c}'}
    &\leq
    C\bb{E}\left[ k^{2\alpha} \norm{\bff{u}}{\mathcal{C}^\alpha_T(\bb{H}^1)}^2 \sum_{\ell=1}^n k \norm{\bff{\theta}^{\ell-1}}{\bb{L}^\infty}^2 \right]
    +
	\delta k \bb{E}\left[ \sum_{\ell=1}^n \one_{\Omega_{\kappa,\ell-1}} \norm{\partial_x \bff{\theta}^\ell}{\bb{L}^2}^2 \right]
    \\
    &\leq
    C k^{2\alpha} \bb{E}\left[ \norm{\bff{u}}{\mathcal{C}^\alpha_T(\bb{H}^1)}^4 + \left(\sum_{\ell=1}^n k \norm{\bff{\theta}^{\ell-1}}{\bb{H}^1}^2 \right)^2 \right]
    +
    \delta k \bb{E}\left[ \sum_{\ell=1}^n \one_{\Omega_{\kappa,\ell-1}} \norm{\partial_x \bff{\theta}^\ell}{\bb{L}^2}^2 \right]
    \\
    &\leq
    Ck^{2\alpha} +
    \delta k \bb{E}\left[ \sum_{\ell=1}^n \one_{\Omega_{\kappa,\ell-1}} \norm{\partial_x \bff{\theta}^\ell}{\bb{L}^2}^2 \right],
\end{align*}
where in the last step we used the fact that $\bff{\theta}^{\ell-1}= \Pi_h \bff{u}(t_{\ell-1})- \bff{u}_h^{\ell-1}$, and applied the triangle inequality, the $\bb{H}^1$-stability of $\Pi_h$, as well as the stability estimate in Lemma~\ref{lem:stab L2}. For the term $I_{3c}''$, we use \eqref{equ:Omega k m} and Young's inequality to obtain
\begin{align*}
    \abs{I_{3c}''}
    &\leq
    C\kappa^2 k \bb{E}\left[ \sum_{\ell=1}^n \one_{\Omega_{\kappa,\ell-1}} \norm{\bff{\theta}^{\ell-1}}{\bb{L}^2}^2 \right] 
	+
	\delta k \bb{E}\left[ \sum_{\ell=1}^n \one_{\Omega_{\kappa,\ell-1}} \norm{\bff{\theta}^{\ell-1}}{\bb{H}^1}^2 \right]
	+
	\delta k \bb{E}\left[ \sum_{\ell=1}^n \one_{\Omega_{\kappa,\ell-1}} \norm{\partial_x \bff{\theta}^\ell}{\bb{L}^2}^2 \right].
\end{align*}
To summarise, from \eqref{equ:I3 holder} we have
\begin{align*}
    \abs{I_3}
    &\leq
    Ck^{2\alpha}+Ch^2 + {C\kappa^2 k \bb{E}\left[ \sum_{\ell=1}^n \one_{\Omega_{\kappa,\ell-1}} \norm{\bff{\theta}^{\ell-1}}{\bb{L}^2}^2 \right] }
    \\
    &\quad
    +
	\delta k \bb{E}\left[ \sum_{\ell=1}^n \one_{\Omega_{\kappa,\ell-1}} \norm{\partial_x \bff{\theta}^\ell}{\bb{L}^2}^2 \right]
    +
	\delta k \bb{E}\left[ \sum_{\ell=1}^n \one_{\Omega_{\kappa,\ell-1}} \norm{\bff{\theta}^{\ell-1}}{\bb{H}^1}^2 \right].
\end{align*}
Next, for $I_4$, applying Young's inequality and noting Lemma~\ref{lem:stab L2}, we have
\begin{align*}
	\abs{I_4}
	&\leq 
	\bb{E}\left[ \sum_{\ell=1}^n \one_{\Omega_{\kappa,\ell-1}} \int_{t_{\ell-1}}^{t_\ell} \norm{\bff{u}_h^{\ell-1}}{\bb{L}^\infty} \left( \norm{\partial_x \bff{u}(s)-\partial_x \bff{u}(t_{\ell})}{\bb{L}^2} + \norm{\partial_x \bff{\rho}^{\ell}}{\bb{L}^2} \right) \norm{\partial_x \bff{\theta}^\ell}{\bb{L}^2} \ds \right]
    \\
    &\leq
    C \bb{E}\left[\sum_{\ell=1}^n \left(k^\alpha \norm{\bff{u}}{\mathcal{C}^\alpha_T(\bb{H}^1)}\right) \left(k^{\frac12} \norm{\bff{u}_h^{\ell-1}}{\bb{L}^\infty}\right) \left(\one_{\Omega_{\kappa,\ell-1}} k^{\frac12} \norm{\partial_x \bff{\theta}^\ell}{\bb{L}^2} \right) \right] 
    \\
    &\quad
    +
    C\bb{E}\left[\sum_{\ell=1}^n \left(h \norm{\bff{u}}{L^\infty_T(\bb{H}^2)} \right) \left(k^{\frac12} \norm{\bff{u}_h^{\ell-1}}{\bb{L}^\infty}\right) \left(\one_{\Omega_{\kappa,\ell-1}} k^{\frac12} \norm{\partial_x \bff{\theta}^\ell}{\bb{L}^2} \right) \right]
    \\
    &\leq
    C \bb{E}\left[k^{2\alpha} \norm{\bff{u}}{\mathcal{C}^\alpha_T(\bb{H}^1)}^2 \sum_{\ell=1}^n k \norm{\bff{u}_h^{\ell-1}}{\bb{L}^\infty}^2 \right] 
    +
    \frac12 \delta k \bb{E}\left[ \sum_{\ell=1}^n \one_{\Omega_{\kappa,\ell-1}} \norm{\partial_x \bff{\theta}^\ell}{\bb{L}^2}^2 \right]
    \\
    &\quad
    +
    C\bb{E}\left[h^2 \norm{\bff{u}}{L^\infty_T(\bb{H}^2)}^2 \sum_{\ell=1}^n k \norm{\bff{u}_h^{\ell-1}}{\bb{L}^\infty}^2 \right]
    +
    \frac12 \delta k \bb{E}\left[ \sum_{\ell=1}^n \one_{\Omega_{\kappa,\ell-1}} \norm{\partial_x \bff{\theta}^\ell}{\bb{L}^2}^2 \right]
	\\
	&\leq
	Ck^{2\alpha} \bb{E}\left[ \norm{\bff{u}}{\mathcal{C}^\alpha_T(\bb{H}^1)}^4 + \left(k \sum_{\ell=1}^n \norm{\bff{u}_h^{\ell-1}}{\bb{H}^1}^2 \right)^2 \right]
    +
    Ch^2 \bb{E}\left[ \norm{\bff{u}}{L^\infty_T(\bb{H}^2)}^4 + \left(k \sum_{\ell=1}^n \norm{\bff{u}_h^{\ell-1}}{\bb{H}^1}^2 \right)^2 \right]
    \\
    &\quad
	+
	\delta k \bb{E}\left[ \sum_{\ell=1}^n \one_{\Omega_{\kappa,\ell-1}} \norm{\partial_x \bff{\theta}^\ell}{\bb{L}^2}^2 \right]
	\\
	&\leq
	Ck^{2\alpha}+ Ch^2+ \delta k \bb{E}\left[ \sum_{\ell=1}^n \one_{\Omega_{\kappa,\ell-1}} \norm{\partial_x \bff{\theta}^\ell}{\bb{L}^2}^2 \right].
\end{align*}
We further split the term $I_5$ as follows:
\begin{align}\label{equ:term I5}
	\abs{I_5}
	&\leq
	\bb{E}\left[ \sum_{\ell=1}^n \one_{\Omega_{\kappa,\ell-1}} \int_{t_{\ell-1}}^{t_\ell} \norm{\bff{u}(s)+\bff{u}_h^{\ell-1}}{\bb{L}^2} \norm{\bff{u}(s)- \bff{u}(t_{\ell-1})}{\bb{L}^2} \norm{\bff{u}(s)}{\bb{L}^\infty} \norm{\bff{\theta}^\ell}{\bb{L}^\infty} \ds \right]
	\nonumber\\
	&\quad
	+
	\bb{E}\left[ \sum_{\ell=1}^n \one_{\Omega_{\kappa,\ell-1}} \int_{t_{\ell-1}}^{t_\ell} \norm{\bff{u}(s)+\bff{u}_h^{\ell-1}}{\bb{L}^2} \norm{\bff{\rho}^{\ell-1}}{\bb{L}^2} \norm{\bff{u}(s)}{\bb{L}^\infty} \norm{\bff{\theta}^\ell}{\bb{L}^\infty} \ds \right]
	\nonumber\\
	&\quad
	+
	\bb{E}\left[ \sum_{\ell=1}^n \one_{\Omega_{\kappa,\ell-1}} \int_{t_{\ell-1}}^{t_\ell} \norm{\bff{u}(s)}{\bb{L}^\infty}^2 \norm{\bff{\theta}^{\ell-1}}{\bb{L}^2}  \norm{\bff{\theta}^\ell}{\bb{L}^2} \ds \right]
	\nonumber\\
	&\quad
	+
	\bb{E}\left[ \sum_{\ell=1}^n \one_{\Omega_{\kappa,\ell-1}} \int_{t_{\ell-1}}^{t_\ell} \norm{\abs{\bff{u}_h^{\ell-1}} \abs{\bff{\theta}^\ell}}{\bb{L}^2} \norm{\bff{\theta}^{\ell-1}}{\bb{L}^2} \norm{\bff{u}(s)}{\bb{L}^\infty} \ds \right]
	\nonumber\\
	&=:
	I_{5a}+I_{5b}+I_{5c}+I_{5d}.
\end{align}
To estimate these terms, we use Young's inequality, H\"older continuity in time of $\bff{u}$, the Sobolev embedding $\bb{H}^1\hookrightarrow\bb{L}^\infty$, Lemma~\ref{lem:stab L2}, and~\eqref{equ:regularity sllb}. Firstly, for the term $I_{5a}$, we obtain
\begin{align*}
	\abs{I_{5a}}
	&\leq
	\bb{E}\left[ \sum_{\ell=1}^n \one_{\Omega_{\kappa,\ell-1}} \int_{t_{\ell-1}}^{t_\ell} \norm{\bff{u}(s)+\bff{u}_h^{\ell-1}}{\bb{L}^2} \norm{\bff{u}(s)-\bff{u}(t_{\ell-1})}{\bb{L}^2} \norm{\bff{u}(s)}{\bb{L}^\infty} \norm{\bff{\theta}^\ell}{\bb{L}^\infty} \ds \right]
	\\
	&\leq
	Ck^{2\alpha} \bb{E}\left[1+\norm{\bff{u}}{L^\infty_T(\bb{H}^1)}^6 + \left(\max_{j\leq n} \norm{\bff{u}_h^j}{\bb{L}^2}^6\right) + \norm{\bff{u}}{\mathcal{C}^\alpha_T(\bb{L}^2)}^6 \right]
	+
	\delta k \bb{E}\left[ \sum_{\ell=1}^n \one_{\Omega_{\kappa,\ell-1}} \norm{\bff{\theta}^\ell}{\bb{H}^1}^2 \right]
	\\
	&\leq
	Ck^{2\alpha} +
	\delta k \bb{E}\left[ \sum_{\ell=1}^n \one_{\Omega_{\kappa,\ell-1}} \norm{\bff{\theta}^\ell}{\bb{H}^1}^2 \right].
\end{align*}
For the term $I_{5b}$, similarly we obtain
\begin{align*}
	\abs{I_{5b}}
	&\leq
	\bb{E}\left[ \sum_{\ell=1}^n \one_{\Omega_{\kappa,\ell-1}} \int_{t_{\ell-1}}^{t_\ell} \norm{\bff{u}(s)+\bff{u}_h^{\ell-1}}{\bb{L}^2} \norm{\bff{\rho}^{\ell-1}}{\bb{L}^2} \norm{\bff{u}(s)}{\bb{L}^\infty} \norm{\bff{\theta}^\ell}{\bb{L}^\infty} \ds \right]
	\\
	&\leq
	Ch^4 \bb{E}\left[1+\norm{\bff{u}}{L^\infty_T(\bb{H}^2)}^6+ \left(\max_{j\leq n} \norm{\bff{u}_h^j}{\bb{L}^2}^6\right) \right]
	+
	\delta k \bb{E}\left[ \sum_{\ell=1}^n \one_{\Omega_{\kappa,\ell-1}} \norm{\bff{\theta}^\ell}{\bb{H}^1}^2 \right]
	\\
	&\leq
	Ch^4+ \delta k \bb{E}\left[ \sum_{\ell=1}^n \one_{\Omega_{\kappa,\ell-1}} \norm{\bff{\theta}^\ell}{\bb{H}^1}^2 \right].
\end{align*}
Next, for the term $I_{5c}$, we proceed as in the case of $I_{3c}$ in \eqref{equ:I3 holder}, estimating for $s\in [t_{\ell-1},t_\ell]$ as follows:
\begin{align*}
    \norm{\bff{u}(s)}{\bb{L}^\infty}^2 \norm{\bff{\theta}^{\ell-1}}{\bb{L}^2} \norm{\bff{\theta}^\ell}{\bb{L}^2}
    &\leq
    2\left(k^\alpha \norm{\bff{u}}{\mathcal{C}^\alpha_T(\bb{H}^1)}^2 + \norm{\bff{u}(t_{\ell-1})}{\bb{H}^1}^2\right) \norm{\bff{\theta}^{\ell-1}}{\bb{L}^2} \norm{\bff{\theta}^\ell}{\bb{L}^2},
\end{align*}
so that by \eqref{equ:Omega k m}, H\"older's and Young's inequalities,
\begin{align}\label{equ:I5c hold}
    \abs{I_{5c}}
    &\leq
    C \bb{E}\left[ \sum_{\ell=1}^n \one_{\Omega_{\kappa,\ell-1}} 
    \left(k^{\alpha+\frac12} \norm{\bff{u}}{\mathcal{C}^\alpha_T(\bb{H}^1)}^2 
    \norm{\bff{\theta}^{\ell-1}}{\bb{L}^2} \right)
    \left(k^{\frac12} \norm{\bff{\theta}^\ell}{\bb{L}^2} \right) \right]
    \nonumber\\
    &\quad
    +
    C k \bb{E}\left[ \sum_{\ell=1}^n \one_{\Omega_{\kappa,\ell-1}} 
    \norm{\bff{u}(t_{\ell-1})}{\bb{H}^1}^2 
    \norm{\bff{\theta}^{\ell-1}}{\bb{L}^2}
    \norm{\bff{\theta}^\ell}{\bb{L}^2} \right]
    \nonumber\\
    &\leq
    C\bb{E} \left[k^{2\alpha} \norm{\bff{u}}{\mathcal{C}^\alpha_T(\bb{H}^1)}^4  \sum_{\ell=1}^n k \norm{\bff{\theta}^{\ell-1}}{\bb{L}^2}^2 \right]
    +
    C\kappa^2 k \bb{E}\left[ \sum_{\ell=1}^n \one_{\Omega_{\kappa,\ell-1}} \norm{\bff{\theta}^{\ell-1}}{\bb{L}^2}^2 \right] 
    \nonumber\\
    &\quad
    +
    \delta k \bb{E}\left[ \sum_{\ell=1}^n \one_{\Omega_{\kappa,\ell-1}} \norm{\bff{\theta}^\ell}{\bb{L}^2}^2 \right]
    \nonumber\\
    &\leq
    C k^{2\alpha} \bb{E}\left[ \norm{\bff{u}}{\mathcal{C}^\alpha_T(\bb{H}^1)}^8 + \norm{\bff{u}}{L^\infty_T(\bb{L}^2)}^4 + \max_{j\leq n} \norm{\bff{u}_h^j}{\bb{L}^2}^4 \right]
    +
    C\kappa^2 k \bb{E}\left[ \sum_{\ell=1}^n \one_{\Omega_{\kappa,\ell-1}} \norm{\bff{\theta}^{\ell-1}}{\bb{L}^2}^2 \right] 
    \nonumber\\
    &\quad
    +
    \delta k \bb{E}\left[ \sum_{\ell=1}^n \one_{\Omega_{\kappa,\ell-1}} \norm{\bff{\theta}^\ell}{\bb{L}^2}^2 \right]
    \nonumber\\
    &\leq
    Ck^{2\alpha}+  C\kappa^2 k \bb{E}\left[ \sum_{\ell=1}^n \one_{\Omega_{\kappa,\ell-1}} \norm{\bff{\theta}^{\ell-1}}{\bb{L}^2}^2 \right] 
    +
    \delta k \bb{E}\left[ \sum_{\ell=1}^n \one_{\Omega_{\kappa,\ell-1}} \norm{\bff{\theta}^\ell}{\bb{L}^2}^2 \right],
\end{align}
where in the last step we used the $\bb{H}^1$-stability of $\Pi_h$ and Lemma~\ref{lem:stab L2}.
For the term $I_{5d}$, we similarly infer that
\begin{align*}
	\abs{I_{5d}}
	&\leq 
    Ck^{2\alpha}
    +
	C\kappa k \bb{E}\left[ \sum_{\ell=1}^n \one_{\Omega_{\kappa,\ell-1}} \norm{\bff{\theta}^{\ell-1}}{\bb{L}^2}^2 \right] 
	+
	\delta k \bb{E}\left[ \sum_{\ell=1}^n \one_{\Omega_{\kappa,\ell-1}} \norm{\abs{\bff{u}_h^{\ell-1}} \abs{\bff{\theta}^\ell}}{\bb{L}^2}^2 \right].
\end{align*}
For the term $I_6$, by Young's inequality and \eqref{equ:proj approx} we have
\begin{align*}
	\abs{I_6}
	&\leq
	\bb{E}\left[ \sum_{\ell=1}^n \one_{\Omega_{\kappa,\ell-1}} \int_{t_{\ell-1}}^{t_\ell} \norm{\bff{u}_h^{\ell-1}}{\bb{L}^\infty} \norm{\bff{\rho}^{\ell}}{\bb{L}^2} \norm{\abs{\bff{u}_h^{\ell-1}} \abs{\bff{\theta}^\ell}}{\bb{L}^2} \ds \right]
	\\
	&\leq
	Ch^4 \bb{E}\left[k \sum_{\ell=1}^n \one_{\Omega_{\kappa,\ell-1}} \norm{\bff{u}_h^{\ell-1}}{\bb{L}^\infty}^2 \norm{\bff{u}}{L^\infty_T(\bb{H}^2)}^2 \right] 
	+
	\delta k \bb{E}\left[ \sum_{\ell=1}^n \one_{\Omega_{\kappa,\ell-1}} \norm{\abs{\bff{u}_h^{\ell-1}} \abs{\bff{\theta}^\ell}}{\bb{L}^2}^2 \right]
	\\
	&\leq
	Ch^4 \bb{E} \left[ \left(k \sum_{\ell=1}^n \one_{\Omega_{\kappa,\ell-1}} \norm{\bff{u}_h^{\ell-1}}{\bb{H}^1}^2 \right)^2 + \norm{\bff{u}}{L^\infty_T(\bb{H}^2)}^4 \right] 
	+
	\delta k \bb{E}\left[ \sum_{\ell=1}^n \one_{\Omega_{\kappa,\ell-1}} \norm{\abs{\bff{u}_h^{\ell-1}} \abs{\bff{\theta}^\ell}}{\bb{L}^2}^2 \right]
	\\
	&\leq
	Ch^4 + \delta k \bb{E}\left[ \sum_{\ell=1}^n \one_{\Omega_{\kappa,\ell-1}} \norm{\abs{\bff{u}_h^{\ell-1}} \abs{\bff{\theta}^\ell}}{\bb{L}^2}^2 \right].
\end{align*}
For the term $I_7$, by a similar argument we have
\begin{align*}
    \abs{I_7}
    &\leq
    \bb{E}\left[ \sum_{\ell=1}^n \one_{\Omega_{\kappa,\ell-1}} \int_{t_{\ell-1}}^{t_\ell} \norm{\bff{u}_h^{\ell-1}}{\bb{L}^2} \norm{\bff{u}(s)-\bff{u}(t_\ell)}{\bb{L}^\infty} \norm{\abs{\bff{u}_h^{\ell-1}} \abs{\bff{\theta}^\ell}}{\bb{L}^2} \ds \right]
    \\
    &\leq
    Ck^{2\alpha} \bb{E}\left[\norm{\bff{u}}{\mathcal{C}^\alpha_T(\bb{H}^1)}^2 \left(\max_{j\leq n-1} \norm{\bff{u}_h^j}{\bb{L}^2}^2 \right) \right] 
    +
    \delta k \bb{E}\left[ \sum_{\ell=1}^n \one_{\Omega_{\kappa,\ell-1}} \norm{\abs{\bff{u}_h^{\ell-1}} \abs{\bff{\theta}^\ell}}{\bb{L}^2}^2 \right]
    \\
    &\leq
    Ck^{2\alpha} \bb{E}\left[\norm{\bff{u}}{\mathcal{C}^\alpha_T(\bb{H}^1)}^4 + \max_{j\leq n-1} \norm{\bff{u}_h^j}{\bb{L}^2}^4  \right] 
    +
    \delta k \bb{E}\left[ \sum_{\ell=1}^n \one_{\Omega_{\kappa,\ell-1}} \norm{\abs{\bff{u}_h^{\ell-1}} \abs{\bff{\theta}^\ell}}{\bb{L}^2}^2 \right]
    \\
    &\leq
    Ck^{2\alpha} +
    \delta k \bb{E}\left[ \sum_{\ell=1}^n \one_{\Omega_{\kappa,\ell-1}} \norm{\abs{\bff{u}_h^{\ell-1}} \abs{\bff{\theta}^\ell}}{\bb{L}^2}^2 \right].
\end{align*}
Next, for the term $I_8$, noting that $\bff{g}\in \bb{L}^\infty$, we use Young's inequality and \eqref{equ:proj approx} to obtain
\begin{align*}
	\abs{I_8}
	&\leq
	Ch^4+ Ck^{2\alpha} + Ck\bb{E}\left[ \sum_{\ell=1}^n \one_{\Omega_{\kappa,\ell-1}} \norm{\bff{\theta}^{\ell-1}}{\bb{L}^2}^2 \right] 
	+
	\delta k \bb{E}\left[ \sum_{\ell=1}^n \one_{\Omega_{\kappa,\ell-1}} \norm{\bff{\theta}^\ell}{\bb{L}^2}^2 \right].
\end{align*}
Finally, for the term $I_9$, we split the stochastic integral as
\begin{align}\label{equ:I8 stoch}
	I_9
	&=
	\bb{E}\left[\max_{m\leq n} \sum_{\ell=1}^m \one_{\Omega_{\kappa,\ell-1}} \int_{t_{\ell-1}}^{t_\ell} \inpro{(\bff{u}(s)-\bff{u}(t_{\ell-1})+ \bff{\rho}^{\ell-1}+ \bff{\theta}^{\ell-1})\times \bff{g}}{\bff{\theta}^{\ell-1}} \dW_s \right]
	\nonumber\\
	&\quad
	+
	\bb{E}\left[\max_{m\leq n} \sum_{\ell=1}^m \one_{\Omega_{\kappa,\ell-1}} \int_{t_{\ell-1}}^{t_\ell} \inpro{(\bff{u}(s)-\bff{u}(t_{\ell-1})+ \bff{\rho}^{\ell-1}+ \bff{\theta}^{\ell-1})\times \bff{g}}{\bff{\theta}^\ell- \bff{\theta}^{\ell-1}} \dW_s \right]
	\nonumber\\
	&=: I_{9a}+ I_{9b}.
\end{align}
For the term $I_{9a}$, we first note that $(\bff{a}\times \bff{b})\cdot \bff{a}=0$ for any $\bff{a},\bff{b}\in \bb{R}^3$. Using the H\"older continuity of $\bff{u}$, the Burkholder--Davis--Gundy and the Young inequalities, we obtain
\begin{align*}
	I_{9a}
	&\leq
	C \bb{E}\left[\left( \sum_{\ell=1}^n \one_{\Omega_{\kappa,\ell-1}} \int_{t_{\ell-1}}^{t_\ell} \left(\norm{\bff{u}(s)-\bff{u}(t_{\ell-1})}{\bb{L}^2}^2 + \norm{\bff{\rho}^{\ell-1}}{\bb{L}^2}^2 \right) \norm{\bff{\theta}^{\ell-1}}{\bb{L}^2}^2 \ds \right)^{\frac12} \right]
	\\
	&\leq
	C \bb{E} \left[ \left(\max_{m\leq n} \one_{\Omega_{\kappa,m}} \norm{\bff{\theta}^{m}}{\bb{L}^2}\right) \left(\sum_{\ell=1}^n \one_{\Omega_{\kappa,\ell-1}} \int_{t_{\ell-1}}^{t_\ell} \norm{\bff{u}(s)-\bff{u}(t_{\ell-1})}{\bb{L}^2}^2 + \norm{\bff{\rho}^{\ell-1}}{\bb{L}^2}^2 \ds \right)^{\frac12} \right]
	\\
	&\leq
	\delta \bb{E} \left[\max_{m\leq n} \one_{\Omega_{\kappa,m-1}} \norm{\bff{\theta}^m}{\bb{L}^2}^2 \right]
	+
	Ch^4
	+
	Ck^{2\alpha}.
\end{align*}
For the term $I_{9b}$, we use Young's inequality, It\^o's isometry, and the H\"older continuity of $\bff{u}$ to obtain
\begin{align*}
	I_{9b}
	&\leq
	C\bb{E} \left[\sum_{\ell=1}^n \norm{\one_{\Omega_{\kappa,\ell-1}} \int_{t_{\ell-1}}^{t_\ell} \left((\bff{u}(s)-\bff{u}(t_{\ell-1})+ \bff{\rho}^{\ell-1}+ \bff{\theta}^{\ell-1})\times \bff{g}\right) \dW_s}{\bb{L}^2}^2 \right] 
	\\
	&\quad
	+
	\delta \bb{E}\left[\sum_{\ell=1}^n \one_{\Omega_{\kappa,\ell-1}} \norm{\bff{\theta}^\ell -\bff{\theta}^{\ell-1}}{\bb{L}^2}^2 \right]
	\\
	&=
	C\bb{E} \left[\sum_{\ell=1}^n \one_{\Omega_{\kappa,\ell-1}} \int_{t_{\ell-1}}^{t_\ell} \norm{(\bff{u}(s)-\bff{u}(t_{\ell-1})+ \bff{\rho}^{\ell-1}+ \bff{\theta}^{\ell-1})\times \bff{g}}{\bb{L}^2}^2 \ds \right] 
	\\
	&\quad
	+
	\delta \bb{E}\left[\sum_{\ell=1}^n \one_{\Omega_{\kappa,\ell-1}} \norm{\bff{\theta}^\ell -\bff{\theta}^{\ell-1}}{\bb{L}^2}^2 \right]
	\\
	&\leq
	Ch^4
	+
	Ck^{2\alpha}
	+
	C k \bb{E}\left[\sum_{\ell=1}^n \one_{\Omega_{\kappa,\ell-1}} \norm{\bff{\theta}^{\ell-1}}{\bb{L}^2}^2 \right]
	+
	\delta \bb{E}\left[\sum_{\ell=1}^n \one_{\Omega_{\kappa,\ell-1}} \norm{\bff{\theta}^\ell -\bff{\theta}^{\ell-1}}{\bb{L}^2}^2 \right].
\end{align*}

Finally, we substitute all the above estimates into~\eqref{equ:12 theta n}, set $\delta=1/16$, and rearrange the terms to obtain
\begin{align}\label{equ:E theta L2 last}
	&\bb{E}\left[\max_{m\leq n} \left( \one_{\Omega_{\kappa,m-1}} \norm{\bff{\theta}^m}{\bb{L}^2}^2 \right) \right] 
	+
	\sum_{\ell=1}^n \bb{E} \left[\one_{\Omega_{\kappa,\ell-1}} \norm{\bff{\theta}^\ell-\bff{\theta}^{\ell-1}}{\bb{L}^2}^2 \right]
	\nonumber\\
	&\quad
	+
	k
	\bb{E}\left[\sum_{\ell=1}^n \one_{\Omega_{\kappa,\ell-1}} \norm{\partial_x \bff{\theta}^\ell}{\bb{L}^2}^2 \right]
	+
	k
	\bb{E}\left[\sum_{\ell=1}^n \one_{\Omega_{\kappa,\ell-1}} \norm{\bff{\theta}^\ell}{\bb{L}^2}^2 \right]
	+
	k
	\bb{E}\left[\sum_{\ell=1}^n \one_{\Omega_{\kappa,\ell-1}} \norm{\abs{\bff{u}_h^{\ell-1}} \abs{\bff{\theta}^\ell}}{\bb{L}^2}^2 \right]
	\nonumber\\
	&\leq
	\bb{E}\left[\norm{\bff{\theta}^0}{\bb{L}^2}^2 \right] 
	+
	Ch^2 + Ck^{2\alpha}
	+
	C\kappa^2 k \sum_{\ell=1}^{n-1} \bb{E} \left[\max_{m\leq \ell} \one_{\Omega_{\kappa,m-1}} \norm{\bff{\theta}^m}{\bb{L}^2}^2 \right].
\end{align}
By choosing $\bff{u}_h^0$ such that $\bb{E}\left[\norm{\bff{\theta}^0}{\bb{H}^1}^2 \right] \leq Ch^2$, say $\bff{u}_h^0=\Pi_h\bff{u}_0$, we infer the required result by the discrete Gronwall lemma.
\end{proof}

\begin{corollary}\label{cor:error subset}
Under the same hypotheses as Proposition~\ref{pro:E theta n L2}, for $n\in \{1,2,\ldots,N\}$ we have
\begin{align}\label{equ:E error 1}
	\bb{E}\left[\max_{m\leq n} \left( \one_{\Omega_{\kappa,m-1}} \norm{\bff{u}(t_m)-\bff{u}_h^m}{\bb{L}^2}^2 \right) \right]  
	&+
	k  \sum_{\ell=1}^n \bb{E} \left[ \one_{\Omega_{\kappa,\ell-1}} \norm{\partial_x \bff{u}(t_m)-\partial_x \bff{u}_h^m}{\bb{L}^2}^2 \right]
	\nonumber\\
	&\quad \leq
	\widetilde{C} e^{\widetilde{C} \kappa^2} \left(h^2+k^{2\alpha}\right),
\end{align}
for any $\alpha\in (0,\frac12)$. The constant $\widetilde{C}$ depends on $T$, but is independent $n$, $h$, $k$, and $\kappa$. 

In particular, for the set
\begin{align*}
    \Omega_{h,k}:= \left\{\omega\in \Omega: \max_{t\in [0,T]} \norm{\bff{u}(t)}{\bb{H}^1}^2 \leq \sqrt{\log\big(\log \left[(h^2+k^{2\alpha})^{-1}\right]\big)} \right\}
\end{align*}
which satisfies $\bb{P}\left[\Omega_{h,k}\right] \geq 1- \left(\log\big(\log\left[(h^2+k^{2\alpha})^{-1}\right]\big)\right)^{-\frac12}$, 
we have for every $\delta>0$,
\begin{align*}
    \bb{E}\left[ \one_{\Omega_{h,k}} \left( \max_{m\leq n}  \norm{\bff{u}(t_m)-\bff{u}_h^m}{\bb{L}^2}^2 \right) + k \sum_{\ell=1}^n \norm{\partial_x \bff{u}(t_\ell)-\partial_x \bff{u}_h^\ell}{\bb{L}^2}^2 \right]  
	\leq
    C \left(h^2+k^{2\alpha}\right)^{1-\delta}.
\end{align*}
\end{corollary}

\begin{proof}
These follow immediately from \eqref{equ:proj approx}, Proposition~\ref{pro:E theta n L2}, and the triangle inequality.
\end{proof}

In the remaining of this section, we need the following quantities:
\begin{align}\label{equ:An exp}
	A_n&:= \max_{m\leq n} \left(\one_{\Omega_{\kappa,m}} \norm{\bff{u}(t_m)-\bff{u}_h^m}{\bb{L}^2}^2 \right)
	+
	k  \sum_{\ell=1}^n \one_{\Omega_{\kappa,\ell}} \norm{\partial_x \bff{u}(t_\ell)- \partial_x \bff{u}_h^\ell}{\bb{L}^2}^2,
	\\
	\label{equ:An comp}
	\widetilde{A_n} &:= \max_{m\leq n} \left(\one_{\Omega_{\kappa,m}^\complement} \norm{\bff{u}(t_m)-\bff{u}_h^m}{\bb{H}^1}^2 \right)
	+
	k  \sum_{\ell=1}^n \one_{\Omega_{\kappa,\ell}^\complement} \norm{\partial_x \bff{u}(t_\ell)- \partial_x \bff{u}_h^\ell}{\bb{L}^2}^2.
\end{align}
The following corollary yields a sharper rate of convergence in probability compared to~\citep{GolJiaLe24}.

\begin{corollary}\label{cor:1d prob error}
Under the same hypotheses as Proposition~\ref{pro:E theta n L2}, for $n\in \{1,2,\ldots,N\}$ and for every $\delta\in (0,1)$, $\alpha\in (0,\frac12)$, and $\gamma>0$, we have
\begin{align}\label{equ:error prob}
	\lim_{h,k\to 0^+} \bb{P} \left[\max_{m\leq n} \norm{\bff{u}(t_m)-\bff{u}_h^m}{\bb{L}^2}^2 + k \sum_{\ell=1}^n  \norm{\partial_x \bff{u}(t_\ell)-\partial_x \bff{u}_h^\ell}{\bb{L}^2}^2 \geq \gamma(h^2+k^{2\alpha})^{1-\delta} \right] = 0.
\end{align}
\end{corollary}

\begin{proof}
By Chebyshev's inequality and \eqref{equ:E error 1} with $\kappa^2= O\big(\log\big(\log\left[(h^2+k^{2\alpha})^{-1}\right]\big)\big)$, for any $\delta\in (0,1)$ and $\gamma>0$ we have
\begin{align*}
	&\bb{P} \left[\max_{m\leq n} \norm{\bff{u}(t_m)-\bff{u}_h^m}{\bb{L}^2}^2 + k \sum_{\ell=1}^n  \norm{\partial_x \bff{u}(t_\ell)-\partial_x \bff{u}_h^\ell}{\bb{L}^2}^2 \geq \gamma(h^2+k^{2\alpha})^{1-\delta} \right]
	\\
	&\leq
	\gamma^{-1} (h^2+k^{2\alpha})^{\delta-1}\, \bb{E}\left[A_n\right] +  \bb{P} \left[\Omega_{\kappa,n-1}^\complement\right]
	\\
	&\leq
	C\gamma^{-1} (h^2+k^{2\alpha})^{\delta-1} (h^2+k^{2\alpha})^{1-\frac{1}{2}\delta}
	+
	C \left[\log\big(\log\left[(h^2+k^{2\alpha})^{-1}\right]\big)\right]^{-\frac12},
\end{align*}
which tends to $0$ as $h,k\to 0^+$, as required.
\end{proof}

Finally, we prove the following theorem on the strong convergence of the numerical scheme.

\begin{theorem}\label{the:rate 1d}
Let $d=1$ and $\bff{g}\in \bb{H}^2$. Suppose that $\bff{u}$ is the pathwise solution to \eqref{equ:sllb} with initial data $\bff{u}_0\in \bb{H}^2$, and let $\{\bff{u}_h^n\}_n$ be a sequence of random variables solving~\eqref{equ:euler} with $\bff{u}_h^0=\Pi_h \bff{u}_0$. Then the semi-implicit scheme \eqref{equ:euler} converges in $L^2\big(\Omega; \ell^\infty(0,T;\bb{L}^2(\mathscr{D}))\big) \cap L^2\big(\Omega; \ell^2(0,T;\bb{H}^1(\mathscr{D}))\big)$. More precisely, for sufficiently small $h$ and $k$ we have
\begin{align}\label{equ:1d rate}
	 \bb{E}\left[\max_{m\leq n} \norm{\bff{u}(t_m)-\bff{u}_h^m}{\bb{L}^2}^2 + k \sum_{\ell=1}^n  \norm{\partial_x \bff{u}(t_\ell)-\partial_x \bff{u}_h^\ell}{\bb{L}^2}^2 \right]
	&\leq
	C \exp \left(-\gamma_0 \sqrt{\log[(h^2+k^{2\alpha})^{-1}]} \right)
\end{align}
for any $\alpha\in (0,\frac12)$. The constant $C$ depends on $T$, while the constant $\gamma_0$ depends on $\beta_0$ (defined in Theorem~\ref{the:exp moment}).
\end{theorem}

\begin{proof}
Let $\Omega_{\kappa,n}$ be defined in \eqref{equ:Omega k m}. Note that by H\"older's inequality with exponents $2^{q-1}$ and $p=2^{q-1}/(2^{q-1}-1)$, where $q>1$, we have
\begin{align}\label{equ:max compl u L2}
	\bb{E} \left[\max_{m\leq n} \one_{\Omega_{\kappa,m-1}^\complement} \norm{\bff{u}(t_m)- \bff{u}_h^m}{\bb{L}^2}^2 \right]
	\leq
	C\left(\bb{P} \left[\Omega_{\kappa,n-1}^\complement\right]\right)^{\frac{1}{p}} 
	\left(\bb{E}\left[\max_{t\in [0,T]} \norm{\bff{u}(t)}{\bb{L}^2}^{2^q} + \max_{m\leq n} \norm{\bff{u}_h^m}{\bb{L}^2}^{2^q}\right] \right)^{\frac{1}{2^{q-1}}}
\end{align}
and
\begin{align}\label{equ:max compl sum u L2}
	\bb{E} \left[k \sum_{\ell=1}^n \one_{\Omega_{\kappa,\ell-1}^\complement} \norm{\bff{u}(t_\ell)-\bff{u}_h^\ell}{\bb{H}^1}^2 \right]
	\leq
	C\left(\bb{P} \left[\Omega_{\kappa,n-1}^\complement\right]\right)^{\frac{1}{p}} 
	\left(\bb{E}\left[\left(k \sum_{\ell=1}^n \norm{\bff{u}(t_\ell)}{\bb{H}^1}^2 + \norm{\bff{u}_h^\ell}{\bb{H}^1}^2\right)^{2^{q-1}}\right] \right)^{\frac{1}{2^{q-1}}}.
\end{align}
We need a good control on $\bb{P} \left(\Omega_{\kappa,n-1}^\complement\right)$, while at the same time balancing the error committed in \eqref{equ:E error 1}, \eqref{equ:max compl u L2}, and \eqref{equ:max compl sum u L2}. By the exponential Markov inequality, with
$\beta_0$ as defined in Theorem~\ref{the:exp moment}, we obtain
\begin{align}\label{equ:P Omega}
	\bb{P}\bigl(\Omega_{\kappa,n-1}^\complement\bigr)
	&\le
	e^{-\beta_0 \kappa}\,
	\bb{E}\left[\exp\left(\beta_0 \sup_{t\in[0,T]}
	\norm{\bff{u}(t)}{\bb{H}^1}^2\right)\right]
	\le
	C e^{-\beta_0 \kappa},
\end{align}
where the second inequality follows from Theorem~\ref{the:exp moment}. The
constant $C$ depends on $\beta_0$ and $T$, but is independent of $\kappa$ and $\beta$. We now choose $\kappa>0$ such that
\begin{align}\label{equ:equal kappa}
	\exp(\widetilde{C} \kappa^2) \left(h^2+k^{2\alpha}\right) =  \exp \left(-\frac{\beta_0 \kappa}{p}\right),
\end{align}
where $\widetilde{C}$ is the constant in \eqref{equ:C tilde}, i.e. $\kappa$ is the positive solution to the quadratic equation
\[
\widetilde{C} \kappa^2 + \frac{\beta_0}{p}\kappa + \log(h^2+k^{2\alpha})=0.
\]
For any $h,k>0$ sufficiently small such that $h^2+k^{2\alpha}<1$, this equation has a positive solution
\begin{align}\label{equ:kappa choice}
	\kappa= \frac{1}{2\widetilde{C}} \left(- \frac{\beta_0}{p}+ \sqrt{\left(\frac{\beta_0}{p}\right)^2+ 4\widetilde{C} \log[(h^2+k^{2\alpha})^{-1}]}  \right).
\end{align}
Note that $\kappa\to \infty$ as $h,k\to 0^+$. With this choice of $\kappa$, continuing from \eqref{equ:P Omega} we have for any $p>1$ and sufficiently small $h,k>0$,
\begin{align}\label{equ:P bad}
	\left[\bb{P} \left(\Omega_{\kappa,n}^\complement\right)\right]^{\frac{1}{p}} 
	\leq
	C \exp \left(-\frac{\beta_0 \kappa}{p}\right)
	\leq
	C \exp \left(-\frac{\beta_0}{p} \sqrt{\log[(h^2+k^{2\alpha})^{-1}]} \right).
\end{align}
With $\kappa$, $A_n$, and $\widetilde{A}_n$ defined in~\eqref{equ:kappa choice}, \eqref{equ:An exp}, and \eqref{equ:An comp}, respectively, we note \eqref{equ:max compl u L2} and \eqref{equ:max compl sum u L2}, and apply \eqref{equ:equal kappa}, \eqref{equ:P bad} and \eqref{equ:E error 1} to obtain for any $p>1$,
\begin{align*}
	\bb{E}\left[\max_{m\leq n} \norm{\bff{u}(t_m)-\bff{u}_h^m}{\bb{L}^2}^2 +
	k  \sum_{\ell=1}^n \norm{\partial_x \bff{u}(t_\ell)- \partial_x \bff{u}_h^\ell}{\bb{L}^2}^2 \right]
	&=
	\bb{E}\left[A_n\right] + \bb{E} \left[\widetilde{A_n}\right]
	\\
	&\leq
	C \exp \left(-\frac{\beta_0}{p} \sqrt{\log[(h^2+k^{2\alpha})^{-1}]} \right).
\end{align*}
This completes the proof of the theorem.
\end{proof}

\begin{remark}\label{rem:fast rate}
The order of convergence in probability in \eqref{equ:error prob} appears to be optimal for the $\bb{H}^1$-norm, though sub-optimal in the $\bb{L}^2$-norm, given the regularity of the exact solution described in~\eqref{equ:regularity sllb}. In contrast, the rate of strong convergence over the full probability space in \eqref{equ:1d rate} is likely to be suboptimal. Nevertheless, the right-hand side of \eqref{equ:1d rate} tends to zero as $h,k\to 0^+$. For sufficiently small $h,k>0$, we note that this theoretical rate is slower than $C(h^2+k^{2\alpha})^{\frac{1}{r}}$ for any $r\geq 1$, yet faster than $C_r \abs{\log(h^2+k^{2\alpha})}^{-r}$.
\end{remark}

\section{A semi-implicit finite element approximation of the regularised sLLB equation for $d=2$}\label{sec:fem 2d}

In this section, we consider a finite element approximation of the regularised sLLB equation~\eqref{equ:reg sllb} for $d=2$. Let $\mathcal{T}_h$ be a regular triangulation of $\mathscr{D}\subset \bb{R}^2$ into triangles with maximal mesh-size $h$. Let $T>0$ be fixed and $k$ be the time-step size. We define $\bb{W}_h\subset \bb{H}^2$ to be the $\mathcal{C}^1$-conforming finite space based on the reduced Hsieh--Clough--Tocher elements or the singular Zienkiewicz elements~\citep{Cia78} satisfying the homogeneous Neumann condition on edges belonging to $\partial \mathscr{D}$.

Define a bilinear form $\mathfrak{a}: \bb{H}^2\times \bb{H}^2 \to \bb{R}$ by 
\begin{equation}\label{equ:bilinear}
\mathfrak{a}(\bff{v},\bff{w}):= \inpro{\bff{v}}{\bff{w}}+ \inpro{\nabla \bff{v}}{\nabla \bff{w}} + \varepsilon\inpro{\Delta \bff{v}}{\Delta \bff{w}}.
\end{equation}
With respect to this bilinear form, we define the elliptic projection operator $Q_h:\bb{H}^2\to \bb{W}_h$ by
\begin{align}\label{equ:Rh a proj}
	\mathfrak{a}(Q_h \bff{v}- \bff{v}, \bff{\chi})=0, \quad \forall \bff{\chi}\in \bb{W}_h.
\end{align}
We also define the orthogonal projection operator $P_h:\bb{L}^2\to \bb{W}_h$ such that
\begin{align}\label{equ:Ph H2 L2 proj}
	\inpro{P_h \bff{v}-\bff{v}}{\bff{\chi}}=0, \quad \forall \bff{\chi}\in \bb{W}_h.
\end{align}
A straightforward application of the C\'ea lemma as in~\citep[Chapter~6.1]{Cia78}, followed by the Aubin--Nitsche trick~\citep[Chapter~5.8]{BreSco08} and interpolation~\citep[Chapter~14.3]{BreSco08}, yields the following approximation property for the elliptic projector~$Q_h$: there exists $C>0$ such that
\begin{align}\label{equ:Rh a approx}
	\norm{\bff{v}-Q_h\bff{v}}{\bb{L}^2} +
	h \norm{\nabla(\bff{v}-Q_h\bff{v})}{\bb{L}^2} + 
	h^2 \norm{\Delta (\bff{v}-Q_h \bff{v})}{\bb{L}^2}
	\leq
	Ch^3 \norm{\bff{v}}{\bb{H}^3}, \quad \forall \bff{v}\in \bb{H}^3.
\end{align}
Similarly, there exists $C>0$ such that
\begin{align}\label{equ:Ph a approx}
	\norm{\bff{v}-P_h\bff{v}}{\bb{L}^2} +
	h \norm{\nabla(\bff{v}-P_h\bff{v})}{\bb{L}^2} + 
	h^2 \norm{\Delta (\bff{v}-P_h \bff{v})}{\bb{L}^2}
	\leq
	Ch^3 \norm{\bff{v}}{\bb{H}^3}, \quad \forall \bff{v}\in \bb{H}^3.
\end{align}

We now consider a semi-implicit finite element scheme to solve the regularised sLLB equation \eqref{equ:reg sllb} for $d=2$ proposed in~\citep{GolJiaLe24}. 
Let $\bff{u}_h^{\varepsilon,n}$ be the approximation in $\bb{W}_h$ of $\bff{u}^\varepsilon(t)$ at time $t=t_n$, where $n=0,1,2,\ldots, N$ and $N=\lfloor T/k \rfloor$. For simplicity, we will write $\bff{u}_h^n$ in place of $\bff{u}_h^{\varepsilon,n}$.
We start with $\bff{u}_h^0= Q_h \bff{u}^\varepsilon(0)\in \bb{W}_h$. Given $\bff{u}_h^{n-1} \in \bb{W}_h$, we find $\mathcal{F}_{t_n}$-adapted and $\bb{W}_h$-valued random variable $\bff{u}_h^n$ satisfying $\bb{P}$-a.s.,
\begin{align}\label{equ:2d euler}
	\inpro{\bff{u}_h^n-\bff{u}_h^{n-1}}{\bff{\phi}_h}
	&=
	-
	\varepsilon k\inpro{\Delta \bff{u}_h^n}{\Delta \bff{\phi}_h}
	-
	k \inpro{\nabla \bff{u}_h^n}{\nabla \bff{\phi}_h}
	-
	k \inpro{\bff{u}_h^{n-1}\times \nabla \bff{u}_h^n}{\nabla \bff{\phi}_h}
	\nonumber\\
	&\quad
	-
	k \inpro{\bff{u}_h^n}{\bff{\phi}_h}
	-
	k \inpro{\abs{\bff{u}_h^{n-1}}^2 \bff{u}_h^n}{\bff{\phi}_h}
	\nonumber\\
	&\quad
	+
	\frac{k}{2} \inpro{\left(\bff{u}_h^{n-1}\times \bff{g}\right) \times \bff{g}}{\bff{\phi}_h}
	+
	\inpro{\bff{g}+\bff{u}_h^{n-1}\times \bff{g}}{\bff{\phi}_h} \overline{\Delta} W^n,
\end{align}
for all $\bff{\phi}_h \in \bb{W}_h$. Here, $\overline{\Delta}W^n:= W(t_n)-W(t_{n-1}) \sim \mathcal{N}(0,k)$.

The above scheme employs a $\mathcal{C}^1$-conforming finite element discretisation. Although such methods allow for a direct discretisation of the fourth-order differential operator without introducing auxiliary variables, they are typically more costly to implement than non-conforming or mixed formulations. This scheme is therefore primarily theoretical and the analysis complements our companion paper~\citep{GolSoeTra24c}, where we study a mixed FEM formulation for a physically relevant regularisation of \eqref{equ:sllb}.
Scheme~\eqref{equ:2d euler} is linear, well-posed by the Lax--Milgram lemma, and enjoys the following stability property.

\begin{lemma}\label{lem:stab H2}
	There exists a positive constant $C$ independent of $n$, $h$, and $k$, such that for any $p\in [1,\infty)$,
	\begin{align*}
		&\bb{E} \left[ \max_{l\leq n} \norm{\bff{u}_h^l}{\bb{L}^2}^{2p} \right]
		+
		\bb{E}\left[ \left(\sum_{j=1}^n \norm{\bff{u}_h^j-\bff{u}_h^{j-1}}{\bb{L}^2}^2\right)^p \right]
		+
		\bb{E}\left[ \left(k \sum_{j=1}^n \norm{\bff{u}_h^j}{\bb{H}^2}^2 \right)^p \right] 
		\leq
		C.
	\end{align*}
\end{lemma}

\begin{proof}
This is obtained by taking $\bff{\phi}_h= \bff{u}_h^n$ in \eqref{equ:2d euler}, see~\citep[Lemma~3.2]{GolJiaLe24}.
\end{proof}

To facilitate the proof of the error estimate, we decompose the error of the numerical method at time $t_n$ in two different ways:
\begin{align}\label{equ:split u Rh}
	\bff{u}^\varepsilon(t_n)- \bff{u}_h^n
	&= 
	\left(\bff{u}^\varepsilon(t_n)-Q_h \bff{u}^\varepsilon(t_n)\right)
	+
	\left(Q_h \bff{u}^\varepsilon(t_n)- \bff{u}_h^n\right)
	=:
	\overline{\bff{\eta}}^n+ \overline{\bff{\xi}}^n,
	\\
	\label{equ:split u Ph}
	&= 
	\left(\bff{u}^\varepsilon(t_n)-P_h \bff{u}^\varepsilon(t_n)\right)
	+
	\left(P_h \bff{u}^\varepsilon(t_n)- \bff{u}_h^n\right)
	=:
	\bff{\eta}^n+ \bff{\xi}^n,
\end{align}
As such by \eqref{equ:Rh a proj} and \eqref{equ:Ph H2 L2 proj}, with $\mathfrak{a}$ defined in \eqref{equ:bilinear}, we have
\begin{align}
	\label{equ:proj Rh zero}
	\mathfrak{a}(\overline{\bff{\eta}}^n, \bff{\phi}_h)&=0, \quad \forall \bff{\phi}_h\in \bb{W}_h,
	\\
	\label{equ:proj Ph zero}
	\inpro{\bff{\eta}^n}{\bff{\phi}_h}&=0, \quad \forall \bff{\phi}_h\in \bb{W}_h.
\end{align}
Furthermore, define a sequence of subsets of $\Omega$ which depend on $\kappa$ and $m$:
\begin{align}\label{equ:Omega H2 k m}
	\Omega_{\kappa,m}:= \left\{\omega\in \Omega: \max_{t\leq t_{m} \wedge T} \norm{\bff{u}^\varepsilon(t)}{\bb{H}^1}^2
	\leq \kappa \right\},
\end{align}
where $\kappa>1$ is to be specified. We prove an auxiliary error estimate next. In the following estimate, we do not trace the dependence of the constant on $\varepsilon$ explicitly.

\begin{proposition}\label{pro:2d xi m}
Let $d=2$ and $\bff{g}\in \bb{H}^3$. Suppose that $\bff{u}^\varepsilon$ is the pathwise solution to \eqref{equ:reg sllb} with initial data $\bff{u}_0^\varepsilon \in \bb{H}^3$, and let $\{\bff{u}_h^n\}_n$ be a sequence of random variables solving \eqref{equ:2d euler} with $\bff{u}_h^0= Q_h \bff{u}_0^\varepsilon$. Let $\Omega_{\kappa,m}$ and $\bff{\xi}^n$ be as defined in \eqref{equ:Omega H2 k m} and \eqref{equ:split u Ph}, respectively. Then for $n\in \{1,2,\ldots, N\}$, we have
\begin{align*}
	\bb{E}\left[\max_{m\leq n} \left( \one_{\Omega_{\kappa,m-1}} \norm{\bff{\xi}^m}{\bb{L}^2}^2 \right) \right]  
	+
	k  \sum_{\ell=1}^n \bb{E} \left[ \one_{\Omega_{\kappa,\ell-1}} \norm{\bff{\xi}^\ell}{\bb{H}^2}^2 \right]
	&\leq
	\widetilde{C} e^{\widetilde{C} \kappa^2} \left(h^2+k^{2\alpha}\right),
\end{align*}
for any $\alpha\in (0,\frac12)$.
The constant $\widetilde{C}$ depends on $\varepsilon$ and $T$, but is independent of $n$, $h$, $k$, and $\kappa$.
\end{proposition}

\begin{proof}
First, we subtract \eqref{equ:weakform reg sllb} at time $t_{\ell-1}$ from the same equation at time $t_\ell$. Subtracting \eqref{equ:2d euler} from the resulting equation and noting the definition of $\mathfrak{a}$ in~\eqref{equ:bilinear}, we have for any $\bff{\phi}_h\in \bb{W}_h$,
\begin{align}
	\label{equ:uhl u reg}
	&\inpro{\big(\bff{u}^\varepsilon(t_\ell)-\bff{u}_h^\ell\big) - \big(\bff{u}^\varepsilon(t_{\ell-1}) - \bff{u}_h^{\ell-1} \big)}{\bff{\phi}_h}
	+
	\int_{t_{\ell-1}}^{t_\ell} \mathfrak{a}\big(\bff{u}^\varepsilon(s)-\bff{u}_h^\ell, \bff{\phi}_h\big) \,\ds
	\nonumber\\
	&=
	-
	\int_{t_{\ell-1}}^{t_\ell} \inpro{\big(\bff{u}^\varepsilon(s)-\bff{u}_h^{\ell-1} \big) \times \nabla \bff{u}^\varepsilon(s)}{\nabla \bff{\phi}_h} \ds
	-
	\int_{t_{\ell-1}}^{t_\ell} \inpro{\bff{u}_h^{\ell-1} \times \big(\nabla \bff{u}^\varepsilon(s)- \nabla \bff{u}_h^\ell\big)}{\nabla \bff{\phi}_h} \ds 
	\nonumber\\
	&\quad
	-
	\int_{t_{\ell-1}}^{t_\ell} \inpro{\big(\abs{\bff{u}^\varepsilon(s)}^2 - |\bff{u}_h^{\ell-1}|^2\big) \bff{u}^\varepsilon(s)}{\bff{\phi}_h} \ds 
	-
	\int_{t_{\ell-1}}^{t_\ell} \inpro{|\bff{u}_h^{\ell-1}|^2 \big(\bff{u}^\varepsilon(s)- \bff{u}_h^\ell \big)}{\bff{\phi}_h} \ds
	\nonumber\\
	&\quad
	+
	\frac12 \int_{t_{\ell-1}}^{t_\ell} \inpro{\big((\bff{u}^\varepsilon(s)-\bff{u}_h^{\ell-1})\times \bff{g}\big)\times \bff{g}}{\bff{\phi}_h} \ds
	+
	\int_{t_{\ell-1}}^{t_\ell} \inpro{\big(\bff{u}^\varepsilon(s)-\bff{u}_h^{\ell-1}\big)\times \bff{g}}{\bff{\phi}_h} \dW_s.
\end{align}
By using~\eqref{equ:split u Rh} and~\eqref{equ:split u Ph}, noting \eqref{equ:proj Rh zero} and \eqref{equ:proj Ph zero}, we can write the left-hand side of~\eqref{equ:uhl u reg} as
\begin{align*}
	&\inpro{\big(\bff{u}^\varepsilon(t_\ell)-\bff{u}_h^\ell\big) - \big(\bff{u}^\varepsilon(t_{\ell-1}) - \bff{u}_h^{\ell-1} \big)}{\bff{\phi}_h}
	+
	\int_{t_{\ell-1}}^{t_\ell} \mathfrak{a} \big(\bff{u}^\varepsilon(s)-\bff{u}_h^\ell, \bff{\phi}_h\big) \,\ds
	\\
	&=
	\inpro{\bff{\xi}^\ell- \bff{\xi}^{\ell-1}}{\bff{\phi}_h}
	+
	\int_{t_{\ell-1}}^{t_\ell} \mathfrak{a}\big(\bff{u}^\varepsilon(s)- \bff{u}^\varepsilon(t_\ell) + \overline{\bff{\xi}}^\ell, \bff{\phi}_h\big) \,\ds.
\end{align*}
Noting this, we put $\bff{\phi}_h=\bff{\xi}^\ell$ in~\eqref{equ:uhl u reg} and multiply the resulting equations by $\one_{\Omega_{\kappa,\ell-1}}$, where the set $\Omega_{\kappa,m}$ was defined in~\eqref{equ:Omega H2 k m}. We then sum the resulting expression over $\ell\in \{1,2,\ldots,m\}$, take the maximum over $m\leq n$, and apply the expectation value. Note that by \eqref{equ:split u Rh} and \eqref{equ:split u Ph},
\begin{align*}
	\mathfrak{a}\big(\overline{\bff{\xi}}^\ell, \bff{\xi}^\ell\big)
	=
	\mathfrak{a}\big( \bff{\xi}^\ell + \bff{\eta}^\ell - \overline{\bff{\eta}}^\ell, \bff{\xi}^\ell\big)
	=
	\mathfrak{a}\big( \bff{\xi}^\ell, \bff{\xi}^\ell \big) + \mathfrak{a}\big( \bff{\eta}^\ell, \bff{\xi}^\ell\big).
\end{align*}
Using~\eqref{equ:split u Ph}, \eqref{equ:proj Ph zero}, \eqref{equ:1 vn vn1}, and~\eqref{equ:u2 u exp} as in the proof of Proposition~\ref{pro:E theta n L2}, and rearranging the terms, we obtain
\begin{align}\label{equ:reg 12 theta n}
	&\frac12 \bb{E}\left[\max_{m\leq n} \left( \one_{\Omega_{\kappa,m-1}} \norm{\bff{\xi}^m}{\bb{L}^2}^2 \right) \right] 
	+
	\frac12  \sum_{\ell=1}^n \bb{E} \left[\one_{\Omega_{\kappa,\ell-1}} \norm{\bff{\xi}^\ell-\bff{\xi}^{\ell-1}}{\bb{L}^2}^2 \right]
	\nonumber\\
	&\quad
	+
	k
	\bb{E}\left[\sum_{\ell=1}^n \one_{\Omega_{\kappa,\ell-1}} \norm{\bff{\xi}^\ell}{\bb{H}^1}^2 \right]
	+
	\varepsilon k
	\bb{E}\left[\sum_{\ell=1}^n \one_{\Omega_{\kappa,\ell-1}} \norm{\Delta \bff{\xi}^\ell}{\bb{L}^2}^2 \right]
	+
	k
	\bb{E}\left[\sum_{\ell=1}^n \one_{\Omega_{\kappa,\ell-1}} \norm{\abs{\bff{u}_h^{\ell-1}} \abs{\bff{\xi}^\ell}}{\bb{L}^2}^2 \right]
	\nonumber\\
	&\leq
	\frac12 \bb{E} \left[\norm{\bff{\xi}^0}{\bb{L}^2}^2 \right]
	-
	\bb{E}\left[\max_{m\leq n} \sum_{\ell=1}^m \one_{\Omega_{\kappa,\ell-1}} \int_{t_{\ell-1}}^{t_\ell} \mathfrak{a}\big(\bff{u}^\varepsilon(s)- \bff{u}^\varepsilon(t_\ell) + \bff{\eta}^\ell, \bff{\xi}^\ell \big) \, \ds\right]
	\nonumber\\
	&\quad
-
\bb{E}\left[\max_{m\leq n} \sum_{\ell=1}^m \one_{\Omega_{\kappa,\ell-1}} \int_{t_{\ell-1}}^{t_\ell} \inpro{\big(\bff{u}^\varepsilon(s)-\bff{u}^\varepsilon(t_{\ell-1})+ \bff{\eta}^{\ell-1}+ \bff{\xi}^{\ell-1}\big) \times \nabla \bff{u}^\varepsilon(s)}{\nabla \bff{\xi}^\ell} \ds\right]
\nonumber\\
&\quad
-
\bb{E}\left[\max_{m\leq n} \sum_{\ell=1}^m \one_{\Omega_{\kappa,\ell-1}} \int_{t_{\ell-1}}^{t_\ell} \inpro{\bff{u}_h^{\ell-1}\times \big(\nabla \bff{u}^\varepsilon(s)-\nabla \bff{u}^\varepsilon(t_\ell)+ \nabla \bff{\eta}^\ell\big)}{\nabla \bff{\xi}^\ell} \ds \right]
\nonumber\\
&\quad
-
\bb{E}\left[\max_{m\leq n} \sum_{\ell=1}^m \one_{\Omega_{\kappa,\ell-1}} \int_{t_{\ell-1}}^{t_\ell} \inpro{\big(\bff{u}^\varepsilon(s)+\bff{u}_h^{\ell-1}\big) \cdot \big(\bff{u}^\varepsilon(s)-\bff{u}^\varepsilon(t_{\ell-1})+ \bff{\eta}^{\ell-1} + \bff{\xi}^{\ell-1}\big) \bff{u}^\varepsilon(s)}{\bff{\xi}^\ell} \ds \right]
\nonumber\\
&\quad
-
\bb{E}\left[\max_{m\leq n} \sum_{\ell=1}^m \one_{\Omega_{\kappa,\ell-1}} \int_{t_{\ell-1}}^{t_\ell} \inpro{|\bff{u}_h^{\ell-1}|^2 \bff{\eta}^\ell}{\bff{\xi}^\ell} \ds \right]
\nonumber\\
&\quad
 -
\bb{E}\left[\max_{m\leq n} \sum_{\ell=1}^m \one_{\Omega_{\kappa,\ell-1}} \int_{t_{\ell-1}}^{t_\ell} \inpro{|\bff{u}_h^{\ell-1}|^2 \big(\bff{u}^\varepsilon(s)-\bff{u}^\varepsilon(t_\ell)\big)}{\bff{\xi}^\ell} \ds \right]
\nonumber\\
&\quad
+
\frac12 \bb{E}\left[\max_{m\leq n} \sum_{\ell=1}^m \one_{\Omega_{\kappa,\ell-1}} \int_{t_{\ell-1}}^{t_\ell} \inpro{\big( (\bff{u}^\varepsilon(s)-\bff{u}^\varepsilon(t_{\ell-1})+ \bff{\eta}^{\ell-1}+ \bff{\xi}^{\ell-1})\times \bff{g}\big)\times \bff{g}}{\bff{\xi}^\ell} \ds \right]
\nonumber\\
&\quad
+
\bb{E}\left[\max_{m\leq n} \sum_{\ell=1}^m \one_{\Omega_{\kappa,\ell-1}} \int_{t_{\ell-1}}^{t_\ell} \inpro{(\bff{u}^\varepsilon(s)-\bff{u}^\varepsilon(t_{\ell-1})+ \bff{\eta}^{\ell-1}+ \bff{\xi}^{\ell-1})\times \bff{g}}{\bff{\xi}^\ell} \dW_s \right]
\nonumber\\
&=: \frac12 \bb{E} \left[\norm{\bff{\xi}^0}{\bb{L}^2}^2 \right]+ J_1+J_2+\cdots+ J_8.
\end{align}

We will estimate each term on the last line. We recall that $\bff{u}^\varepsilon$ has regularity given by~\eqref{equ:regularity reg sllb}, which will be used in the estimates without further mention. Let $\delta>0$ be a constant to be determined later. For the term $J_1$, by H\"older continuity in time of $\bff{u}^\varepsilon$, Young's inequality, and \eqref{equ:Ph a approx} we have
\begin{align*}
	\abs{J_1}
	\leq 
	Ch^2+ Ck^{2\alpha} + \delta k \bb{E} \left[\sum_{\ell=1}^n \one_{\Omega_{\kappa,\ell-1}}  \norm{\bff{\xi}^\ell}{\bb{H}^2}^2 \right].
\end{align*}
For the term $J_2$, we use H\"older continuity in time of $\bff{u}^\varepsilon$, the Sobolev embedding $\bb{H}^1\hookrightarrow \bb{L}^4$, and the Gagliardo--Nirenberg inequalities. Employing the same argument as in \eqref{equ:I3 holder}, but noting~\eqref{equ:Ph a approx} and \eqref{equ:Omega H2 k m}, we have
\begin{align*}
	\abs{J_2}
	&\leq
	\bb{E}\left[ \sum_{\ell=1}^n \one_{\Omega_{\kappa,\ell-1}} \int_{t_{\ell-1}}^{t_\ell} \left( \norm{\bff{u}^\varepsilon(s)-\bff{u}^\varepsilon(t_{\ell-1})}{\bb{L}^4} + \norm{\bff{\eta}^{\ell-1}}{\bb{L}^4} + \norm{\bff{\xi}^{\ell-1}}{\bb{L}^4} \right) \norm{\nabla \bff{u}^\varepsilon(s)}{\bb{L}^2} \norm{\nabla \bff{\xi}^\ell}{\bb{L}^4} \ds \right]
	\\
	&\leq
	Ck^{2\alpha} \bb{E}\left[\norm{\bff{u}^\varepsilon}{\mathcal{C}^\alpha_T(\bb{H}^1)}^2 \norm{\bff{u}^\varepsilon}{L^\infty_T(\bb{H}^1)}^2 \right]
	+
	Ch^4 \bb{E}\left[\norm{\bff{u}^\varepsilon}{L^\infty_T(\bb{H}^3)}^4\right]
	+
	\delta k \bb{E}\left[ \sum_{\ell=1}^n \one_{\Omega_{\kappa,\ell-1}} \norm{\Delta \bff{\xi}^\ell}{\bb{L}^2}^2 \right]
	\\
	&\quad
    +
    Ck^{1+\alpha} \bb{E}\left[\sum_{\ell=1}^n \one_{\Omega_{\kappa,\ell-1}} \norm{\bff{\xi}^{\ell-1}}{\bb{L}^4} \norm{\bff{u}^\varepsilon}{\mathcal{C}^\alpha_T(\bb{H}^1)} \norm{\Delta \bff{\xi}^\ell}{\bb{L}^2} \right]
    \\
    &\quad
	+
	Ck \bb{E}\left[ \sum_{\ell=1}^n \one_{\Omega_{\kappa,\ell-1}} \norm{\bff{\xi}^{\ell-1}}{\bb{L}^2}^{\frac12} \norm{ \bff{\xi}^{\ell-1}}{\bb{H}^1}^{\frac12} \norm{\nabla \bff{u}^\varepsilon(t_{\ell-1})}{\bb{L}^2} \norm{\Delta \bff{\xi}^\ell}{\bb{L}^2} \right]
	\\
	&\leq
	Ck^{2\alpha} + Ch^4 
	+
	C\kappa^2 k \bb{E}\left[ \sum_{\ell=1}^n \one_{\Omega_{\kappa,\ell-1}} \norm{\bff{\xi}^{\ell-1}}{\bb{L}^2}^2 \right] 
	\\
	&\quad
	+
	\delta k \bb{E}\left[ \sum_{\ell=1}^n \one_{\Omega_{\kappa,\ell-1}} \norm{\bff{\xi}^{\ell-1}}{\bb{H}^1}^2 \right]
	+
	\delta k \bb{E}\left[ \sum_{\ell=1}^n \one_{\Omega_{\kappa,\ell-1}} \norm{\Delta \bff{\xi}^\ell}{\bb{L}^2}^2 \right].
\end{align*}
For the next term, noting Lemma~\ref{lem:stab H2} and applying Young's inequality, similarly we have
\begin{align*}
	\abs{J_3}
	&\leq 
	\bb{E}\left[ \sum_{\ell=1}^n \one_{\Omega_{\kappa,\ell-1}} \int_{t_{\ell-1}}^{t_\ell} \norm{\bff{u}_h^{\ell-1}}{\bb{L}^\infty} \left( \norm{\nabla \bff{u}^\varepsilon(s)-\nabla \bff{u}^\varepsilon(t_{\ell})}{\bb{L}^2} + \norm{\nabla \bff{\eta}^{\ell}}{\bb{L}^2} \right) \norm{\nabla \bff{\xi}^\ell}{\bb{L}^2} \ds \right]
	\\
	&\leq
    C \bb{E}\left[\sum_{\ell=1}^n \left(k^\alpha \norm{\bff{u}^\varepsilon}{\mathcal{C}^\alpha_T(\bb{H}^1)}\right) \left(k^{\frac12} \norm{\bff{u}_h^{\ell-1}}{\bb{L}^\infty}\right) \left(\one_{\Omega_{\kappa,\ell-1}} k^{\frac12} \norm{\nabla \bff{\xi}^\ell}{\bb{L}^2} \right) \right]
    \\
    &\quad
    +
    C\bb{E}\left[\sum_{\ell=1}^n \left(h^2 \norm{\bff{u}^\varepsilon}{L^\infty_T(\bb{H}^3)} \right) \left(k^{\frac12} \norm{\bff{u}_h^{\ell-1}}{\bb{L}^\infty}\right) \left(\one_{\Omega_{\kappa,\ell-1}} k^{\frac12} \norm{\nabla \bff{\xi}^\ell}{\bb{L}^2} \right) \right]
    \\
    &\leq
    C \bb{E}\left[k^{2\alpha} \norm{\bff{u}^\varepsilon}{\mathcal{C}^\alpha_T(\bb{H}^1)}^2 \sum_{\ell=1}^n k \norm{\bff{u}_h^{\ell-1}}{\bb{L}^\infty}^2 \right] 
    +
    \frac12 \delta k \bb{E}\left[ \sum_{\ell=1}^n \one_{\Omega_{\kappa,\ell-1}} \norm{\nabla \bff{\xi}^\ell}{\bb{L}^2}^2 \right]
    \\
    &\quad
    +
    C\bb{E}\left[h^4 \norm{\bff{u}^\varepsilon}{L^\infty_T(\bb{H}^3)}^2 \sum_{\ell=1}^n k \norm{\bff{u}_h^{\ell-1}}{\bb{L}^\infty}^2 \right]
    +
    \frac12 \delta k \bb{E}\left[ \sum_{\ell=1}^n \one_{\Omega_{\kappa,\ell-1}} \norm{\nabla \bff{\xi}^\ell}{\bb{L}^2}^2 \right]
	\\
	&\leq
	Ck^{2\alpha} \bb{E}\left[ \norm{\bff{u}^\varepsilon}{\mathcal{C}^\alpha_T(\bb{H}^1)}^4 + \left(k \sum_{\ell=1}^n \norm{\bff{u}_h^{\ell-1}}{\bb{H}^2}^2 \right)^2 \right]
    +
    Ch^4 \bb{E}\left[ \norm{\bff{u}^\varepsilon}{L^\infty_T(\bb{H}^3)}^4 + \left(k \sum_{\ell=1}^n \norm{\bff{u}_h^{\ell-1}}{\bb{H}^2}^2 \right)^2 \right]
    \\
    &\quad
	+
	\delta k \bb{E}\left[ \sum_{\ell=1}^n \one_{\Omega_{\kappa,\ell-1}} \norm{\nabla \bff{\xi}^\ell}{\bb{L}^2}^2 \right]
	\\
	&\leq
	Ck^{2\alpha}+ Ch^4+ \delta k \bb{E}\left[ \sum_{\ell=1}^n \one_{\Omega_{\kappa,\ell-1}} \norm{\nabla \bff{\xi}^\ell}{\bb{L}^2}^2 \right].
\end{align*}
We further split the term $J_4$ as follows:
\begin{align*}
	\abs{J_4}
	&\leq
	\bb{E}\left[ \sum_{\ell=1}^n \one_{\Omega_{\kappa,\ell-1}} \int_{t_{\ell-1}}^{t_\ell} \norm{\bff{u}^\varepsilon(s)+\bff{u}_h^{\ell-1}}{\bb{L}^2} \norm{\bff{u}^\varepsilon(s)- \bff{u}^\varepsilon(t_{\ell-1})}{\bb{L}^6} \norm{\bff{u}^\varepsilon(s)}{\bb{L}^6} \norm{\bff{\xi}^\ell}{\bb{L}^6} \,\ds \right]
	\\
	&\quad
	+
	\bb{E}\left[ \sum_{\ell=1}^n \one_{\Omega_{\kappa,\ell-1}} \int_{t_{\ell-1}}^{t_\ell} \norm{\bff{u}^\varepsilon(s)+\bff{u}_h^{\ell-1}}{\bb{L}^2} \norm{\bff{\eta}^{\ell-1}}{\bb{L}^6} \norm{\bff{u}^\varepsilon(s)}{\bb{L}^6} \norm{\bff{\xi}^\ell}{\bb{L}^6} \,\ds \right]
	\\
	&\quad
	+
	\bb{E}\left[ \sum_{\ell=1}^n \one_{\Omega_{\kappa,\ell-1}} \int_{t_{\ell-1}}^{t_\ell} \norm{\bff{u}^\varepsilon(s)}{\bb{L}^6}^2 \norm{\bff{\xi}^{\ell-1}}{\bb{L}^2}  \norm{\bff{\xi}^\ell}{\bb{L}^6} \,\ds \right]
	\\
	&\quad
	+
	\bb{E}\left[ \sum_{\ell=1}^n \one_{\Omega_{\kappa,\ell-1}} \int_{t_{\ell-1}}^{t_\ell} \norm{\abs{\bff{u}_h^{\ell-1}} \abs{\bff{\xi}^\ell}}{\bb{L}^2} \norm{\bff{\xi}^{\ell-1}}{\bb{L}^4} \norm{\bff{u}^\varepsilon(s)}{\bb{L}^4} \,\ds \right]
	\\
	&=:
	J_{4a}+J_{4b}+J_{4c}+J_{4d}.
\end{align*}
To estimate these terms, we use Young's inequality, H\"older continuity in time of $\bff{u}$, the Sobolev embedding $\bb{H}^1\hookrightarrow\bb{L}^6$, Lemma~\ref{lem:stab H2}, and~\eqref{equ:regularity reg sllb}. Firstly, for the term $J_{4a}$, we have
\begin{align*}
	\abs{J_{4a}}
	&\leq
	\bb{E}\left[ \sum_{\ell=1}^n \one_{\Omega_{\kappa,\ell-1}} \int_{t_{\ell-1}}^{t_\ell} \norm{\bff{u}^\varepsilon(s)+\bff{u}_h^{\ell-1}}{\bb{L}^2} \norm{\bff{u}^\varepsilon(s)-\bff{u}^\varepsilon(t_{\ell-1})}{\bb{L}^6} \norm{\bff{u}^\varepsilon(s)}{\bb{L}^6} \norm{\bff{\xi}^\ell}{\bb{L}^6} \ds \right]
	\\
	&\leq
	Ck^{2\alpha} \bb{E}\left[1+\norm{\bff{u}^\varepsilon}{L^\infty_T(\bb{H}^1)}^6 + \left(\max_{j\leq n} \norm{\bff{u}_h^j}{\bb{L}^2}^6\right) + \norm{\bff{u}^\varepsilon}{\mathcal{C}^\alpha_T(\bb{H}^1)}^6 \right]
	+
	\delta k \bb{E}\left[ \sum_{\ell=1}^n \one_{\Omega_{\kappa,\ell-1}} \norm{\bff{\xi}^\ell}{\bb{H}^1}^2 \right]
	\\
	&\leq
	Ck^{2\alpha} +
	\delta k \bb{E}\left[ \sum_{\ell=1}^n \one_{\Omega_{\kappa,\ell-1}} \norm{\bff{\xi}^\ell}{\bb{H}^1}^2 \right].
\end{align*}
For the term $J_{4b}$, noting \eqref{equ:Ph a approx}, similarly we obtain
\begin{align*}
	\abs{J_{4b}}
	&\leq
	\bb{E}\left[ \sum_{\ell=1}^n \one_{\Omega_{\kappa,\ell-1}} \int_{t_{\ell-1}}^{t_\ell} \norm{\bff{u}^\varepsilon(s)+\bff{u}_h^{\ell-1}}{\bb{L}^2} \norm{\bff{\eta}^{\ell-1}}{\bb{L}^6} \norm{\bff{u}^\varepsilon(s)}{\bb{L}^6} \norm{\bff{\xi}^\ell}{\bb{L}^6} \ds \right]
	\\
	&\leq
	Ch^4 \bb{E}\left[1+\norm{\bff{u}^\varepsilon}{L^\infty_T(\bb{H}^3)}^6+ \left(\max_{j\leq n} \norm{\bff{u}_h^j}{\bb{L}^2}^6\right) \right]
	+
	\delta k \bb{E}\left[ \sum_{\ell=1}^n \one_{\Omega_{\kappa,\ell-1}} \norm{\bff{\xi}^\ell}{\bb{H}^1}^2 \right]
	\\
	&\leq
	Ch^4+ \delta k \bb{E}\left[ \sum_{\ell=1}^n \one_{\Omega_{\kappa,\ell-1}} \norm{\bff{\xi}^\ell}{\bb{H}^1}^2 \right].
\end{align*}
Next, for the terms $J_{4c}$ and $J_{4d}$, we use \eqref{equ:Omega H2 k m} and employ the same argument as in \eqref{equ:I5c hold} to obtain
\begin{align*}
	\abs{J_{4c}}
	&\leq
    Ck^{2\alpha}
    +
	C\kappa^2 k \bb{E}\left[ \sum_{\ell=1}^n \one_{\Omega_{\kappa,\ell-1}} \norm{\bff{\xi}^{\ell-1}}{\bb{L}^2}^2 \right] 
	+
	\delta k \bb{E}\left[ \sum_{\ell=1}^n \one_{\Omega_{\kappa,\ell-1}} \norm{\bff{\xi}^\ell}{\bb{H}^1}^2 \right],
\end{align*}
as well as
\begin{align*}
	\abs{J_{4d}}
	&\leq 
	\bb{E}\left[ \sum_{\ell=1}^n \one_{\Omega_{\kappa,\ell-1}} \int_{t_{\ell-1}}^{t_\ell} \norm{\abs{\bff{u}_h^{\ell-1}} \abs{\bff{\xi}^\ell}}{\bb{L}^2} \norm{\bff{\xi}^{\ell-1}}{\bb{L}^4} \left(k^\alpha \norm{\bff{u}^\varepsilon}{\mathcal{C}^\alpha_T(\bb{H}^1)}\right) \,\ds \right]
    \\
    &\quad
    +
    \bb{E}\left[ \sum_{\ell=1}^n \one_{\Omega_{\kappa,\ell-1}} \int_{t_{\ell-1}}^{t_\ell} \norm{\abs{\bff{u}_h^{\ell-1}} \abs{\bff{\xi}^\ell}}{\bb{L}^2} \norm{\bff{\xi}^{\ell-1}}{\bb{L}^2}^{\frac12} \norm{\bff{\xi}^{\ell-1}}{\bb{H}^1}^{\frac12} \norm{\bff{u}^\varepsilon(t_{\ell-1})}{\bb{H}^1} \,\ds \right]
	\\
	&\leq
    Ck^{2\alpha}
    +
	C\kappa k \bb{E}\left[ \sum_{\ell=1}^n \one_{\Omega_{\kappa,\ell-1}} \norm{\bff{\xi}^{\ell-1}}{\bb{L}^2}^2 \right] 
    \\
    &\quad
	+
	\delta k \bb{E}\left[ \sum_{\ell=1}^n \one_{\Omega_{\kappa,\ell-1}} \norm{\abs{\bff{u}_h^{\ell-1}} \abs{\bff{\xi}^\ell}}{\bb{L}^2}^2 \right]
	+
	\delta k \bb{E}\left[ \sum_{\ell=1}^n \one_{\Omega_{\kappa,\ell-1}} \norm{\bff{\xi}^{\ell-1}}{\bb{H}^1}^2 \right].
\end{align*}
For the term $J_5$, by Young's inequality, Lemma~\ref{lem:stab H2}, and \eqref{equ:Ph a approx} we have
\begin{align*}
	\abs{J_5}
	&\leq
	\bb{E}\left[ \sum_{\ell=1}^n \one_{\Omega_{\kappa,\ell-1}} \int_{t_{\ell-1}}^{t_\ell} \norm{\bff{u}_h^{\ell-1}}{\bb{L}^4} \norm{\bff{\eta}^{\ell}}{\bb{L}^4} \norm{\abs{\bff{u}_h^{\ell-1}} \abs{\bff{\xi}^\ell}}{\bb{L}^2} \ds \right]
	\\
	&\leq
	Ch^4 \bb{E}\left[k \sum_{\ell=1}^n \one_{\Omega_{\kappa,\ell-1}} \norm{\bff{u}_h^{\ell-1}}{\bb{L}^4}^2 \norm{\bff{u}^\varepsilon}{L^\infty_T(\bb{H}^3)}^2 \right] 
	+
	\delta k \bb{E}\left[ \sum_{\ell=1}^n \one_{\Omega_{\kappa,\ell-1}} \norm{\abs{\bff{u}_h^{\ell-1}} \abs{\bff{\xi}^\ell}}{\bb{L}^2}^2 \right]
	\\
	&\leq
	Ch^4 \bb{E} \left[ \left(k \sum_{\ell=1}^n \one_{\Omega_{\kappa,\ell-1}} \norm{\bff{u}_h^{\ell-1}}{\bb{H}^1}^2 \right)^2 + \norm{\bff{u}^\varepsilon}{L^\infty_T(\bb{H}^3)}^4 \right] 
	+
	\delta k \bb{E}\left[ \sum_{\ell=1}^n \one_{\Omega_{\kappa,\ell-1}} \norm{\abs{\bff{u}_h^{\ell-1}} \abs{\bff{\xi}^\ell}}{\bb{L}^2}^2 \right]
	\\
	&\leq
	Ch^4 + \delta k \bb{E}\left[ \sum_{\ell=1}^n \one_{\Omega_{\kappa,\ell-1}} \norm{\abs{\bff{u}_h^{\ell-1}} \abs{\bff{\xi}^\ell}}{\bb{L}^2}^2 \right].
\end{align*}
For $J_6$, by a similar argument we have
\begin{align*}
    \abs{J_6}
    &\leq
    \bb{E}\left[ \sum_{\ell=1}^n \one_{\Omega_{\kappa,\ell-1}} \int_{t_{\ell-1}}^{t_\ell} \norm{\bff{u}_h^{\ell-1}}{\bb{L}^2} \norm{\bff{u}^\varepsilon(s)-\bff{u}^\varepsilon(t_\ell)}{\bb{L}^\infty} \norm{\abs{\bff{u}_h^{\ell-1}} \abs{\bff{\xi}^\ell}}{\bb{L}^2} \ds \right]
    \\
    &\leq
    Ck^{2\alpha} \bb{E}\left[\norm{\bff{u}^\varepsilon}{\mathcal{C}^\alpha_T(\bb{H}^2)}^4 + \max_{j\leq n-1} \norm{\bff{u}_h^j}{\bb{L}^2}^4  \right] 
    +
    \delta k \bb{E}\left[ \sum_{\ell=1}^n \one_{\Omega_{\kappa,\ell-1}} \norm{\abs{\bff{u}_h^{\ell-1}} \abs{\bff{\xi}^\ell}}{\bb{L}^2}^2 \right]
    \\
    &\leq
    Ck^{2\alpha} +
    \delta k \bb{E}\left[ \sum_{\ell=1}^n \one_{\Omega_{\kappa,\ell-1}} \norm{\abs{\bff{u}_h^{\ell-1}} \abs{\bff{\xi}^\ell}}{\bb{L}^2}^2 \right].
\end{align*}
Next, for the term $J_7$, noting that $\bff{g}\in \bb{L}^\infty$, we use Young's inequality and \eqref{equ:Ph a approx} to obtain
\begin{align*}
	\abs{J_7}
	&\leq
	Ch^4+ Ck^{2\alpha} + Ck\bb{E}\left[ \sum_{\ell=1}^n \one_{\Omega_{\kappa,\ell-1}} \norm{\bff{\xi}^{\ell-1}}{\bb{L}^2}^2 \right] 
	+
	\delta k \bb{E}\left[ \sum_{\ell=1}^n \one_{\Omega_{\kappa,\ell-1}} \norm{\bff{\xi}^\ell}{\bb{L}^2}^2 \right].
\end{align*}
Finally, the stochastic term $J_8$ can be estimated in the same manner as the term $I_9$ in \eqref{equ:I8 stoch}.

We now substitute all the above estimates into~\eqref{equ:reg 12 theta n}, set $\delta=\min\left\{\frac{1}{16}, \frac{\varepsilon}{4}\right\}$, and rearrange the terms to obtain
\begin{align*}
	&\bb{E}\left[\max_{m\leq n} \left( \one_{\Omega_{\kappa,m-1}} \norm{\bff{\xi}^m}{\bb{L}^2}^2 \right) \right] 
	+
	\sum_{\ell=1}^n \bb{E} \left[\one_{\Omega_{\kappa,\ell-1}} \norm{\bff{\xi}^\ell-\bff{\xi}^{\ell-1}}{\bb{L}^2}^2 \right]
	\nonumber\\
	&\quad
	+
	k
	\bb{E}\left[\sum_{\ell=1}^n \one_{\Omega_{\kappa,\ell-1}} \norm{\bff{\xi}^\ell}{\bb{H}^1}^2 \right]
	+
	\varepsilon k
	\bb{E}\left[\sum_{\ell=1}^n \one_{\Omega_{\kappa,\ell-1}} \norm{\Delta \bff{\xi}^\ell}{\bb{L}^2}^2 \right]
	+
	k
	\bb{E}\left[\sum_{\ell=1}^n \one_{\Omega_{\kappa,\ell-1}} \norm{\abs{\bff{u}_h^{\ell-1}} \abs{\bff{\xi}^\ell}}{\bb{L}^2}^2 \right]
	\nonumber\\
	&\leq
	\bb{E}\left[\norm{\bff{\xi}^0}{\bb{L}^2}^2 \right] 
	+
	Ch^4 + Ck^{2\alpha}
	+
	C\kappa^2 k \sum_{\ell=1}^{n-1} \bb{E} \left[\max_{m\leq \ell} \one_{\Omega_{\kappa,m-1}} \norm{\bff{\xi}^m}{\bb{L}^2}^2 \right].
\end{align*}
By choosing $\bff{u}_h^0$ such that $\bb{E}\left[\norm{\bff{\xi}^0}{\bb{L}^2}^2 \right] \leq Ch^2$, say $\bff{u}_h^0=P_h \bff{u}_0^\varepsilon$, we infer the required result by the discrete Gronwall lemma.
\end{proof}

Consequently, we have the following corollaries on error estimate in a large sample space and convergence in probability of the scheme.

\begin{corollary}\label{cor:error subset 2d}
Under the same hypotheses as Proposition~\ref{pro:2d xi m}, for $n\in \{1,2,\ldots,N\}$ we have
\begin{align}\label{equ:E error 2d}
	\bb{E}\left[\max_{m\leq n} \left( \one_{\Omega_{\kappa,m-1}} \norm{\bff{u}^\varepsilon(t_m)-\bff{u}_h^m}{\bb{L}^2}^2 \right) \right]  
	&+
	k  \sum_{\ell=1}^n \bb{E} \left[ \one_{\Omega_{\kappa,\ell-1}} \norm{\nabla \bff{u}^\varepsilon(t_m)-\nabla \bff{u}_h^m}{\bb{L}^2}^2 \right]
	\nonumber\\
	&\quad \leq
	\widetilde{C} e^{\widetilde{C} \kappa^2} \left(h^2+k^{2\alpha}\right),
\end{align}
for any $\alpha\in (0,\frac12)$. The constant $\widetilde{C}$ depends on $\varepsilon$ and $T$, but is independent of $n$, $h$, $k$, and $\kappa$.

In particular, for the set
\begin{align*}
    \Omega_{h,k}:= \left\{\omega\in \Omega: \max_{t\in [0,T]} \norm{\bff{u}^\varepsilon(t)}{\bb{H}^1}^2 \leq \sqrt{\log\big(\log\left[(h^2+k^{2\alpha})^{-1}\right]\big)} \right\}
\end{align*}
which satisfies $\bb{P}\left[\Omega_{h,k}\right] \geq 1- \left(\log\big(\log\left[(h^2+k^{2\alpha})^{-1}\right]\big)\right)^{-\frac12}$, 
we have for every $\delta>0$,
\begin{align*}
    \bb{E}\left[ \one_{\Omega_{h,k}} \left( \max_{m\leq n}  \norm{\bff{u}^\varepsilon(t_m)-\bff{u}_h^m}{\bb{L}^2}^2 \right) + k \sum_{\ell=1}^n \norm{\partial_x \bff{u}^\varepsilon(t_\ell)-\partial_x \bff{u}_h^\ell}{\bb{L}^2}^2 \right]  
	\leq
	C \left(h^2+k^{2\alpha}\right)^{1-\delta}.
\end{align*}
\end{corollary}

\begin{proof}
This follows immediately from \eqref{equ:Ph a approx}, Proposition~\ref{pro:2d xi m}, and the triangle inequality.
\end{proof}

\begin{corollary}\label{cor:2d prob rate}
	Under the same hypotheses as Proposition~\ref{pro:2d xi m}, for $n\in \{1,2,\ldots,N\}$ and for every $\delta\in (0,1)$, $\alpha\in (0,\frac12)$, and $\gamma>0$, we have
	\begin{align*}
		\lim_{h,k\to 0^+} \bb{P} \left[\max_{m\leq n} \norm{\bff{u}^\varepsilon(t_m)-\bff{u}_h^m}{\bb{L}^2}^2 + k \sum_{\ell=1}^n  \norm{\nabla \bff{u}^\varepsilon(t_\ell)-\nabla \bff{u}_h^\ell}{\bb{L}^2}^2 \geq \gamma(h^2+k^{2\alpha})^{1-\delta} \right] = 0.
	\end{align*}
\end{corollary}

\begin{proof}
The proof is similar to that of Corollary~\ref{cor:1d prob error}.
\end{proof}

We can finally state the main theorem of this section.

\begin{theorem}\label{the:rate 2d}
	Let $d=2$ and $\bff{g}\in \bb{H}^3$. Suppose that $\bff{u}^\varepsilon$ is the pathwise solution to \eqref{equ:reg sllb} with initial data $\bff{u}_0^\varepsilon \in \bb{H}^3$, and let $\{\bff{u}_h^n\}_n$ be a sequence of random variables solving \eqref{equ:2d euler} with $\bff{u}_h^0= Q_h \bff{u}_0^\varepsilon$. Then the semi-implicit scheme \eqref{equ:2d euler} converges in $L^2\big(\Omega; \ell^\infty(0,T;\bb{L}^2(\mathscr{D}))\big) \cap L^2\big(\Omega; \ell^2(0,T;\bb{H}^1(\mathscr{D}))\big)$. More precisely, for sufficiently small $h$ and $k$,
	\begin{align*}
		\bb{E}\left[\max_{m\leq n} \norm{\bff{u}^\varepsilon(t_m)-\bff{u}_h^m}{\bb{L}^2}^2 + k \sum_{\ell=1}^n  \norm{\nabla \bff{u}^\varepsilon(t_m)-\nabla \bff{u}_h^m}{\bb{L}^2}^2 \right]
		&\leq
		C \exp \left(-\gamma_0 \sqrt{\log[(h^2+k^{2\alpha})^{-1}]} \right)
	\end{align*}
	for any $\alpha\in (0,\frac12)$. The constant $C$ depends on $T$ and $\varepsilon$, while the constant $\gamma_0$ depends on $\beta_0$ (defined in Theorem~\ref{the:exp moment}).
\end{theorem}

\begin{proof}
The proof is similar to that of Theorem~\ref{the:rate 1d}.
\end{proof}

\section{An implicit finite element approximation of the sLLB equation for $d=1$}\label{sec:implicit}

We now consider an implicit finite element approximation of the sLLB equation for $d=1$. In this case, the scheme has a better stability property, which allows us to deduce an error estimate in a stronger norm, namely convergence in $L^2\big(\Omega; \ell^\infty(0,T;\bb{H}^1(\mathscr{D}))\big)$. Assuming a more regular initial data, we show that the order of convergence (in probability or strongly over a large sample space) is optimal in the traditional sense of finite element analysis. Throughout this section, we assume that $\bff{u}_0\in \bb{H}^4$ and $\bff{g}\in \bb{H}^4$, thus Theorem~\ref{the:regularity 1d} holds with $s=4$.

Before describing the numerical scheme, we state some relevant preliminary results. We introduce the 1-D discrete Laplacian operator $\partial_{xx}^h: \bb{V}_h \to \bb{V}_h$ given by
\begin{align}\label{equ:disc laplacian}
	\inpro{\partial_{xx}^h \bff{v}_h}{\bff{\chi}}
	=
	- \inpro{\partial_x \bff{v}_h}{\partial_x \bff{\chi}},
	\quad 
	\forall \bff{v}_h, \bff{\chi} \in \bb{V}_h,
\end{align}
as well as the Ritz projection operator $R_h: \bb{H}^1 \to \bb{V}_h$ defined by
\begin{align}\label{equ:Ritz}
	\inpro{\partial_x R_h \bff{v}- \partial_x \bff{v}}{\partial_x \bff{\chi}}=0,
	\quad
	\forall \bff{\chi}\in \bb{V}_h 
        \quad
        \text{such that}
        \quad 
        \inpro{R_h \bff{v}-\bff{v}}{\bff{1}}=0,
\end{align}
where $\bff{1}:= (1,1,1)^\top$.
The operator $R_h$ has the following approximation property: there exists a constant $C>0$ such that
\begin{align}\label{equ:Ritz approx}
	\norm{\bff{v}- R_h\bff{v}}{\bb{L}^2}
	+
	h \norm{\partial_x( \bff{v}-R_h\bff{v})}{\bb{L}^2}
	\leq
	Ch^2 \norm{\bff{v}}{\bb{H}^2}, \quad \forall \bff{v}\in \bb{H}^2.
\end{align}
By taking $\bff{\chi}=\bff{v}_h$ in \eqref{equ:disc laplacian}, for any $\bff{v}_h\in \bb{V}_h$ we have
\begin{equation}\label{equ:disc lapl ineq}
	\norm{\partial_x \bff{v}_h}{\bb{L}^2}^2 \leq \norm{\bff{v}_h}{\bb{L}^2} \norm{\partial_{xx}^h \bff{v}_h}{\bb{L}^2}.
\end{equation}


Our implicit scheme can be described as follows. Let the space $\bb{V}_h$ be as defined in~\eqref{equ:Vh}. Let $\bff{u}_h^n$ be the approximation in $\bb{V}_h$ of $\bff{u}(t)$ at time $t=t_n:=nk\in [0,T]$, where $n=0,1,2,\ldots, N$ and $N=\lfloor T/k \rfloor$.
We start with $\bff{u}_h^0= \Pi_h \bff{u}(0) \in \bb{V}_h$. Given $\bff{u}_h^{n-1} \in \bb{V}_h$ and $\bff{g}\in \bb{H}^2$, we find $\mathcal{F}_{t_n}$-adapted and $\bb{V}_h$-valued random variable $\bff{u}_h^n$ satisfying $\bb{P}$-a.s.,
\begin{align}\label{equ:implicit euler}
	\inpro{\bff{u}_h^n-\bff{u}_h^{n-1}}{\bff{\phi}_h}
	&=
	k \inpro{\partial_{xx}^h \bff{u}_h^n}{\bff{\phi}_h}
	+
	k \inpro{\bff{u}_h^n\times \partial_{xx}^h \bff{u}_h^n}{\bff{\phi}_h}
	-
	k \inpro{\bff{u}_h^n}{\bff{\phi}_h}
	-
	k \inpro{\abs{\bff{u}_h^{n-1}}^2 \bff{u}_h^n}{\bff{\phi}_h}
	\nonumber\\
	&\quad
	+
	\frac{k}{2} \inpro{\left(\bff{u}_h^n\times \bff{g}\right) \times \bff{g}}{\bff{\phi}_h}
	+
	\inpro{\bff{g}+\bff{u}_h^{n-1}\times \bff{g}}{\bff{\phi}_h} \overline{\Delta} W^n,
\end{align}
for all $\bff{\phi}_h \in \bb{V}_h$. Here, $\overline{\Delta}W^n:= W(t_n)-W(t_{n-1}) \sim \mathcal{N}(0,k)$. Note that we treat the nonlinear cross product term implicitly in~\eqref{equ:implicit euler}.

The following proposition shows that the scheme is well-posed for any $k\in (0,1)$.

\begin{proposition}
	Given $k\in (0,1)$ and $\bff{u}_h^{n-1}\in \bb{V}_h$, there exists a $\bb{V}_h$-valued random variable $\bff{u}_h^n$ that solves the fully discrete scheme \eqref{equ:implicit euler}.
\end{proposition}

\begin{proof}
	Fix $\omega\in \Omega$ and suppose that $\bff{u}_h^{n-1}(\omega)$ is given.
	Define a continuous map $\mathcal{G}_n^\omega:\bb{V}_h\to \bb{V}_h$ by
	\begin{align*}
		\mathcal{G}_n^\omega(\bff{v})
		&:=
		\bff{v}
		-
		\bff{u}_h^{n-1}(\omega)
		-
		k \big( \partial_{xx}^h \bff{v}- \bff{v}\big)
		-
		k \Pi_h \big( \bff{v}\times \partial_{xx}^h \bff{v} \big)
		-
		k \Pi_h \big(|\bff{u}_h^{n-1}(\omega)|^2 \bff{v} \big)
		\\
		&\qquad
		-
		\frac{k}{2} \Pi_h \big((\bff{v}\times \bff{g})\times \bff{g}\big)
		-
		\Pi_h (\bff{g}+\bff{u}_h^{n-1}(\omega)\times \bff{g}) \overline{\Delta} W^n(\omega).
	\end{align*}
	The scheme~\eqref{equ:implicit euler} is equivalent to solving~$\mathcal{G}_n^\omega\big(\bff{u}_h^n \big)= \bff{0}$. The existence of solution to this equation will be shown using Brouwer's fixed point theorem. To this end, let~$B_\rho(\omega):= \{\bff{\varphi}\in \bb{V}_h: \norm{\bff{\varphi}}{\bb{L}^2}^2 \leq \rho\}$, where $\rho$ is to be determined later. For any $\bff{v}\in \partial B_\rho(\omega)$, we have
	\begin{align*}
		\inpro{\mathcal{G}_n^\omega(\bff{v})}{\bff{v}}
		&=
		\norm{\bff{v}}{\bb{L}^2}^2
		-
		\inpro{\bff{u}_h^{n-1}(\omega)}{\bff{v}}
		+
		k \norm{\partial_x \bff{v}}{\bb{L}^2}^2
		+
		k \norm{\bff{v}}{\bb{L}^2}^2
		+
		k \norm{\abs{\bff{u}_h^{n-1}(\omega)} \abs{\bff{v}}}{\bb{L}^2}^2
		\\
		&\quad
		+
		\frac{k}{2} \norm{\bff{v}\times \bff{g}}{\bb{L}^2}^2
		-
		\inpro{(\bff{g}+\bff{u}_h^{n-1}(\omega)\times \bff{g}) \overline{\Delta} W^n(\omega)}{\bff{v}}
		\\
		&\geq
		\frac14 \norm{\bff{v}}{\bb{L}^2}^2 
		- \frac12\norm{\bff{u}_h^{n-1}(\omega)}{\bb{L}^2}^2
		+
		k \norm{\bff{v}}{\bb{H}^1}^2
		-
		4\norm{\bff{g}}{\bb{L}^\infty}^2  \left(1+\norm{\bff{u}_h^{n-1}(\omega)}{\bb{L}^2}^2\right) \abs{\overline{\Delta} W^n(\omega)}^2
		\\
		&\geq
		\frac14 \rho 
		- \frac12\norm{\bff{u}_h^{n-1}(\omega)}{\bb{L}^2}^2
		-
		4\norm{\bff{g}}{\bb{L}^\infty}^2  \left(1+\norm{\bff{u}_h^{n-1}(\omega)}{\bb{L}^2}^2\right) \abs{\overline{\Delta} W^n(\omega)}^2
	\end{align*}
	Therefore, for sufficiently large $\rho$, precisely
	\[
	\rho > 2\norm{\bff{u}_h^{n-1}(\omega)}{\bb{L}^2}^2
	+
	16 \norm{\bff{g}}{\bb{L}^\infty}^2  \left(1+\norm{\bff{u}_h^{n-1}(\omega)}{\bb{L}^2}^2\right) \abs{\overline{\Delta} W^n(\omega)}^2
	\]
	we have $\inpro{\mathcal{G}_n^\omega(\bff{v})}{\bff{v}}>0$ for all $\bff{v}\in \partial B_\rho(\omega)$. By a version of Brouwer's fixed point theorem~\citep[Corollary~VI.1.1]{GirRav86}, we infer the existence of $\bff{u}_h^n(\omega)$ such that $\mathcal{G}_n^\omega\big(\bff{u}_h^n \big)= \bff{0}$. The $\mathcal{F}_{t_n}$-measurability of the map $\bff{u}_h^n:\Omega\to \bb{V}_h$ follows by the same argument as in~\citep[Theorem~2.2]{BanBrzNekPro14}. This completes the proof of the proposition.
\end{proof}

Some stability estimates for the scheme are derived in the following lemmas. While Lemma~\ref{lem:implicit stab L2} can be shown by similar argument as before, the moment estimate in Lemma~\ref{lem:implicit stab H1} is new.

\begin{lemma}\label{lem:implicit stab L2}
	There exists a positive constant $C$ independent of $n$, $h$, and $k$, such that for any $p\in [1,\infty)$,
	\begin{align*}
		&\bb{E} \left[ \max_{l\leq n} \norm{\bff{u}_h^l}{\bb{L}^2}^{2p} \right]
		+
		\bb{E}\left[ \left(\sum_{j=1}^n \norm{\bff{u}_h^j-\bff{u}_h^{j-1}}{\bb{L}^2}^2\right)^p \right]
		+
		\bb{E}\left[ \left(k \sum_{j=1}^n \norm{\partial_x \bff{u}_h^j}{\bb{L}^2}^2 \right)^p \right] 
		\leq
		C.
	\end{align*}
\end{lemma}

\begin{proof}
	This is obtained by taking $\bff{\phi}_h= \bff{u}_h^n$ in~\eqref{equ:implicit euler} and following the same argument as in the proof of Lemma~\ref{lem:stab L2}.
\end{proof}

\begin{lemma}\label{lem:implicit stab H1}
There exists a positive constant $C$ independent of $n$, $h$, and $k$, such that for any $p\in [1,\infty)$,
	\begin{align*}
		&\bb{E} \left[ \max_{l\leq n} \norm{\bff{u}_h^l}{\bb{H}^1}^{2p} \right]
		+
		\bb{E}\left[ \left(\sum_{j=1}^n \norm{\bff{u}_h^j-\bff{u}_h^{j-1}}{\bb{H}^1}^2\right)^p \right]
		+
		\bb{E}\left[ \left(k \sum_{j=1}^n \norm{\partial_{xx}^h \bff{u}_h^j}{\bb{L}^2}^2 \right)^p \right] 
		\leq
		C.
	\end{align*}
\end{lemma}

\begin{proof}
Taking $\bff{\phi}_h= -\partial_{xx}^h \bff{u}_h^n \in \bb{V}_h$ in \eqref{equ:implicit euler}, we obtain
\begin{align}\label{equ:partial x I1I3}
	&\frac12 \left(\norm{\partial_x \bff{u}_h^n}{\bb{L}^2}^2 - \norm{\partial_x \bff{u}_h^{n-1}}{\bb{L}^2}^2 \right) 
	+
	\frac12 \norm{\partial_x \bff{u}_h^n- \partial_x \bff{u}_h^{n-1}}{\bb{L}^2}^2
	+
	k \norm{\partial_{xx}^h \bff{u}_h^n}{\bb{L}^2}^2
	+
	k \norm{\partial_x \bff{u}_h^n}{\bb{L}^2}^2
	\nonumber\\
	&=
	k \inpro{\abs{\bff{u}_h^{n-1}}^2 \bff{u}_h^n}{\partial_{xx}^h \bff{u}_h^n}
	-
	\frac{k}{2} \inpro{(\bff{u}_h^n\times \bff{g}) \times \bff{g}}{\partial_{xx}^h \bff{u}_h^n}
	-
	\inpro{\bff{g}+\bff{u}_h^{n-1}\times \bff{g}}{\partial_{xx}^h \bff{u}_h^n} \overline{\Delta} W^n
	\nonumber\\
	&=: I_1+I_2+I_3.
\end{align}
We estimate each term on the last line as follows. Let $\delta>0$. For the term $I_1$, by Agmon's inequality, Young's inequality, and \eqref{equ:disc lapl ineq} we have
\begin{align*}
	\abs{I_1}
	&\leq
	k \norm{\bff{u}_h^{n-1}}{\bb{L}^\infty}^2 \norm{\bff{u}_h^n}{\bb{L}^2} \norm{\partial_{xx}^h \bff{u}_h^n}{\bb{L}^2}
        \\
        &\leq
         Ck \norm{\bff{u}_h^{n-1}}{\bb{L}^2}  \big( \norm{\bff{u}_h^{n-1}}{\bb{L}^2}+ \norm{\partial_x \bff{u}_h^{n-1}}{\bb{L}^2} \big) 
        \norm{\bff{u}_h^n}{\bb{L}^2} \norm{\partial_{xx}^h \bff{u}_h^n}{\bb{L}^2}
	\\
	&\leq
    Ck \norm{\bff{u}_h^{n-1}}{\bb{L}^2}^2  \norm{\bff{u}_h^n}{\bb{L}^2} \norm{\partial_{xx}^h \bff{u}_h^n}{\bb{L}^2} 
        +
	Ck \norm{\bff{u}_h^{n-1}}{\bb{L}^2} \norm{\bff{u}_h^{n-1}}{\bb{L}^2}^{\frac12} \norm{\partial_{xx}^h \bff{u}_h^{n-1}}{\bb{L}^2}^{\frac12} \norm{\bff{u}_h^n}{\bb{L}^2} \norm{\partial_{xx}^h \bff{u}_h^n}{\bb{L}^2}
	\\
	&\leq
	Ck \left(1+ \norm{\bff{u}_h^{n-1}}{\bb{L}^2}^{12} \right)+ Ck \left(1+\norm{\bff{u}_h^n}{\bb{L}^2}^8 \right) + \delta k \norm{\partial_{xx}^h \bff{u}_h^{n-1}}{\bb{L}^2}^2
	+
	\delta k \norm{\partial_{xx}^h \bff{u}_h^n}{\bb{L}^2}^2.
\end{align*}
For the term $I_2$, by Young's inequality,
\begin{align*}
	\abs{I_2}
	&\leq
	Ck \norm{\bff{u}_h^n}{\bb{L}^2}^2
	+
	\delta k \norm{\partial_{xx}^h \bff{u}_h^n}{\bb{L}^2}^2.
\end{align*}
For the last term, note that by \eqref{equ:disc laplacian}, Young's inequality, and the $\bb{H}^1$-stability of $\Pi_h$, we have
\begin{align*}
	I_3 
	&=
	\inpro{\partial_x \Pi_h \big(\bff{g}+ \bff{u}_h^{n-1}\times \bff{g}\big)}{\partial_x \bff{u}_h^n} \overline{\Delta} W^n
	\\
	&=
	\inpro{\partial_x \Pi_h \big(\bff{g}+ \bff{u}_h^{n-1}\times \bff{g}\big)}{\partial_x \bff{u}_h^n-\partial_x \bff{u}_h^{n-1}} \overline{\Delta} W^n
	+
	\inpro{\partial_x \Pi_h \big(\bff{g}+ \bff{u}_h^{n-1}\times \bff{g}\big)}{\partial_x \bff{u}_h^{n-1}} \overline{\Delta} W^n
	\\
	&\leq
	\frac14 \norm{\partial_x \bff{u}_h^n-\partial_x \bff{u}_h^{n-1}}{\bb{L}^2}^2
	+
	C \norm{\partial_x \bff{g}+ \partial_x (\bff{u}_h^{n-1}\times \bff{g})}{\bb{L}^2}^2 \abs{\overline{\Delta} W^n}^2
	\\
	&\quad
	-
	\inpro{\bff{g}+ \bff{u}_h^{n-1}\times \bff{g}}{\partial_{xx}^h \bff{u}_h^{n-1}} \overline{\Delta} W^n.
\end{align*}
We take $\delta=1/8$. Substituting these estimates into \eqref{equ:partial x I1I3}, rearranging the terms, taking the $p$-th power, summing over $j\in \{1,2,\ldots,\ell\}$, and applying the expected value, we obtain for any $p\in [1,\infty)$,
\begin{align}\label{equ:E dx u 2p}
	&\bb{E}\left[\max_{\ell\leq n} \norm{\partial_x \bff{u}_h^\ell}{\bb{L}^2}^{2p} \right] 
	+
	\bb{E} \left[ \left(\sum_{j=1}^n \norm{\partial_x \bff{u}_h^j- \partial_x \bff{u}_h^{j-1}}{\bb{L}^2}^2\right)^p \right] 
	\nonumber\\
	&\quad
	+
	\bb{E} \left[ \left( \sum_{j=1}^n k \left(\norm{\partial_{xx}^h \bff{u}_h^j}{\bb{L}^2}^2 + \norm{\partial_x \bff{u}_h^j}{\bb{L}^2}^2 \right) \right)^p \right]
	\nonumber\\
	&\leq
	\bb{E}\left[\norm{\partial_x \bff{u}_h^0}{\bb{L}^2}^{2p} \right]
	+
	C\bb{E} \left[\max_{\ell\leq n} \norm{\bff{u}_h^{\ell-1}}{\bb{L}^2}^{12p} + \norm{\bff{u}_h^\ell}{\bb{L}^2}^{8p} + \norm{\bff{u}_h^\ell}{\bb{L}^2}^{2p}\right]
	\nonumber\\
	&\quad
	+
	C\bb{E} \left[\max_{\ell\leq n} \left(\sum_{j=1}^\ell \norm{\partial_x \bff{g}+ \partial_x(\bff{u}_h^{j-1}\times \bff{g})}{\bb{L}^2}^2 \abs{\overline{\Delta} W^j}^2 \right)^p \right] 
	\nonumber\\
	&\quad
	+
	C \bb{E} \left[ \max_{\ell\leq n} \left| \sum_{j=1}^\ell \inpro{\partial_x \Pi_h \big(\bff{g}+ \bff{u}_h^{j-1}\times \bff{g}\big)}{\partial_x \bff{u}_h^{j-1}} \overline{\Delta} W^j \right|^p \right].
\end{align}
It remains to bound the last two terms of the above inequality. To this end, note that for any $p\in [1,\infty)$,
\begin{align*}
	&\bb{E} \left[\max_{\ell\leq n} \left(\sum_{j=1}^\ell \norm{\partial_x \bff{g}+ \partial_x(\bff{u}_h^{j-1}\times \bff{g})}{\bb{L}^2}^2 \abs{\overline{\Delta} W^j}^2 \right)^p \right] 
	\\
	&\leq
	n^{p-1} \bb{E} \left[ \sum_{j=1}^\ell \norm{\partial_x \bff{g}+ \partial_x \bff{u}_h^{j-1}\times \bff{g} + \bff{u}_h^{j-1}\times \partial_x \bff{g}}{\bb{L}^2}^{2p} \abs{\overline{\Delta} W^j}^{2p} \right]
	\\
	&\leq
	n^{p-1} k^{p-1} \bb{E}\left[ k \sum_{j=1}^\ell  \left(\norm{\partial_x \bff{g}}{\bb{L}^2}^{2p} + \norm{\partial_x \bff{u}_h^{j-1}}{\bb{L}^2}^{2p} \norm{\bff{g}}{\bb{L}^\infty}^{2p} + \norm{\bff{u}_h^{j-1}}{\bb{L}^4}^{2p} \norm{\partial_x \bff{g}}{\bb{L}^4}^{2p}\right) \right] 
	\\
	&\leq
	CT^{p-1} \bb{E} \left[ k \sum_{j=1}^{\ell} \left(1+ \norm{\partial_x \bff{u}_h^{j-1}}{\bb{L}^2}^{2p}\right) \right],
\end{align*}
where we also used the tower property of conditional expectation and the independence of the Wiener increment in the second inequality. Next, by the Burkholder--Davis--Gundy inequality,
\begin{align*}
	&\bb{E} \left[ \max_{\ell\leq n} \left| \sum_{j=1}^\ell \inpro{\bff{g}+ \bff{u}_h^{j-1}\times \bff{g}}{\partial_{xx}^h \bff{u}_h^{j-1}} \overline{\Delta} W^j \right|^p \right]
	\\
	&\leq
	C \bb{E} \left[ \left( k \sum_{j=1}^n \left(1+ \norm{\bff{u}_h^{j-1}}{\bb{L}^2}^2 \right) \norm{\partial_{xx}^h \bff{u}_h^{j-1}}{\bb{L}^2}^2 \right)^{\frac{p}{2}} \right]
	\\
	&\leq
	C \bb{E} \left[ \max_{j\leq n} \left(1+\norm{\bff{u}_h^{j-1}}{\bb{L}^2}^2\right)^{\frac{p}{2}} \left(\sum_{j=1}^n k \norm{\partial_{xx}^h \bff{u}_h^{j-1}}{\bb{L}^2}^2 \right)^{\frac{p}{2}} \right]
	\\
	&\leq
	C \bb{E} \left[ \max_{j\leq n} \left(1+\norm{\bff{u}_h^{j-1}}{\bb{L}^2}^{2p}\right) \right]
	+
	\frac14 \bb{E} \left[\left(\sum_{j=1}^n k \norm{\partial_{xx}^h \bff{u}_h^{j-1}}{\bb{L}^2}^2 \right)^p \right],
\end{align*}
where in the last step we used Young's inequality. Therefore, continuing from \eqref{equ:E dx u 2p}, rearranging the terms and noting Lemma~\ref{lem:implicit stab L2}, we infer that
\begin{align*}
	&\bb{E}\left[\max_{\ell\leq n} \norm{\partial_x \bff{u}_h^\ell}{\bb{L}^2}^{2p} \right] 
	+
	\bb{E} \left[ \left(\sum_{j=1}^n \norm{\partial_x \bff{u}_h^j- \partial_x \bff{u}_h^{j-1}}{\bb{L}^2}^2\right)^p \right] 
	\nonumber\\
	&\quad
	+
	\bb{E} \left[ \left( \sum_{j=1}^n k \left(\norm{\partial_{xx}^h \bff{u}_h^j}{\bb{L}^2}^2 + \norm{\partial_x \bff{u}_h^j}{\bb{L}^2}^2 \right) \right)^p \right]
    \leq
	C
	+
	C \bb{E} \left[ k \sum_{j=1}^{\ell} \norm{\partial_x \bff{u}_h^{j-1}}{\bb{L}^2}^{2p} \right].
\end{align*}
The required inequality then follows by the discrete Gronwall lemma.
\end{proof}

As before, we decompose the error of the numerical method at time $t_n$:
\begin{align}\label{equ:split u Ritz}
	\bff{u}(t_n)- \bff{u}_h^n
	&= 
	\left(\bff{u}(t_n)-R_h \bff{u}(t_n)\right)
	+
	\left(R_h \bff{u}(t_n)- \bff{u}_h^n\right)
	=:
	\bff{\rho}^n+ \bff{\theta}^n.
\end{align}
By \eqref{equ:Ritz}, we have
\begin{align}
	\label{equ:proj Ritz zero}
	\inpro{\partial_x \bff{\rho}^n}{\partial_x \bff{\phi}_h}&=0, \quad \forall \bff{\phi}_h\in \bb{V}_h,
\end{align}
Furthermore, define a sequence of subsets of $\Omega$ that depend on $\kappa$ and $m$ (and $h$):
\begin{align}\label{equ:implicit Omega k m}
	\Omega_{\kappa,m}:= \left\{\omega\in \Omega: \max_{t\leq t_{m} \wedge T} \norm{\bff{u}(t)}{\bb{H}^3}^2 + \max_{\ell\leq m} \norm{\bff{u}_h^\ell}{\bb{H}^1}^2
	\leq \kappa \right\},
\end{align}
where $\kappa>1$ is to be specified. We prove an auxiliary error estimates next. In the following proposition, since we now employ \eqref{equ:split u Ritz}, a technical restriction $h = O\big(k^{\frac34}\big)$ is required to establish some of the estimates. A stronger restriction is also implicitly assumed in~\citep[Theorem~3.4]{GolJiaLe24}.

\begin{proposition}\label{pro:implicit E theta n H1}
Let $d=1$ and $\bff{g}\in \bb{H}^4$. Suppose that $\bff{u}$ is the pathwise solution to \eqref{equ:sllb} with initial data $\bff{u}_0\in \bb{H}^4$, and let $\{\bff{u}_h^n\}_n$ be a sequence of random variables solving~\eqref{equ:implicit euler} with $\bff{u}_h^0=\Pi_h \bff{u}_0$.
Let $\Omega_{\kappa,m}$ and $\bff{\theta}^n$ be as defined in~\eqref{equ:implicit Omega k m} and~\eqref{equ:split u Ritz}, respectively. Suppose that $h=O(k^{\frac34})$ and $k=O(\kappa^{-2})$. Then for $n\in \{1,2,\ldots, N\}$, we have
\begin{align}\label{equ:implicit C tilde}
    \bb{E}\left[\max_{m\leq n} \left( \one_{\Omega_{\kappa,m-1}} \norm{\bff{\theta}^m}{\bb{H}^1}^2 \right) \right]  
    +
    k  \sum_{\ell=1}^n \bb{E} \left[ \one_{\Omega_{\kappa,\ell-1}} \norm{\partial_{xx}^h \bff{\theta}^\ell}{\bb{L}^2}^2 \right]
    &\leq
    \widetilde{C} e^{\widetilde{C} \kappa^2} \left(h^2+k^{2\alpha}\right),
\end{align}
for any $\alpha\in (0,\frac12)$. The constant $\widetilde{C}$ depends on $T$, but is independent of $n$, $h$, $k$, and $\kappa$.
\end{proposition}

\begin{proof}
Similarly to the proof of Proposition~\ref{pro:E theta n L2}, we have
for any $\bff{\phi}_h\in \bb{V}_h$,
\begin{align}
\label{equ:theta n theta n1 implicit}
    &\inpro{\bff{\theta}^\ell-\bff{\theta}^{\ell-1}}{\bff{\phi}_h}
    +
    \inpro{\bff{\rho}^\ell-\bff{\rho}^{\ell-1}}{\bff{\phi}_h}
    +
    \int_{t_{\ell-1}}^{t_\ell} \inpro{\partial_x \bff{u}(s)-\partial_x \bff{u}_h^\ell}{\partial_x \bff{\phi}_h} \ds
    +
    \int_{t_{\ell-1}}^{t_\ell} \inpro{\bff{u}(s)-\bff{u}_h^\ell}{\bff{\phi}_h} \ds
    \nonumber\\
    &=
    \int_{t_{\ell-1}}^{t_\ell} \inpro{\big(\bff{u}(s)-\bff{u}_h^{\ell} \big) \times \partial_{xx} \bff{u}(s)}{\bff{\phi}_h} \ds
    +
    \int_{t_{\ell-1}}^{t_\ell} \inpro{\bff{u}_h^{\ell} \times \big(\partial_{xx} \bff{u}(s)- \partial_{xx}^h \bff{u}_h^\ell\big)}{\bff{\phi}_h} \ds 
    \nonumber\\
    &\quad
    -
    \int_{t_{\ell-1}}^{t_\ell} \inpro{\big(\abs{\bff{u}(s)}^2 - |\bff{u}_h^{\ell-1}|^2\big) \bff{u}(s)}{\bff{\phi}_h} \ds 
    -
    \int_{t_{\ell-1}}^{t_\ell} \inpro{|\bff{u}_h^{\ell-1}|^2 \big(\bff{u}(s)-\bff{u}_h^\ell\big)}{\bff{\phi}_h} \ds
    \nonumber\\
    &\quad
    +
    \frac12 \int_{t_{\ell-1}}^{t_\ell} \inpro{\big((\bff{u}(s)-\bff{u}_h^{\ell})\times \bff{g}\big)\times \bff{g}}{\bff{\phi}_h} \ds
    +
    \int_{t_{\ell-1}}^{t_\ell} \inpro{\big(\bff{u}(s)-\bff{u}_h^{\ell-1}\big)\times \bff{g}}{\bff{\phi}_h} \dW_s.
\end{align}
First, we set $\bff{\phi}_h=\bff{\theta}^\ell$ in~\eqref{equ:theta n theta n1 implicit} and multiply the resulting equations by $\one_{\Omega_{\kappa,\ell-1}}$, where the set $\Omega_{\kappa,m}$ is defined in~\eqref{equ:implicit Omega k m}. Summing over $\ell\in \{1,2,\ldots,m\}$, taking the maximum over $m\leq n$, and then applying the expectation, we have
\begin{align}\label{equ:12 theta n implicit}
    &\frac12 \bb{E}\left[\max_{m\leq n} \left( \one_{\Omega_{\kappa,m-1}} \norm{\bff{\theta}^m}{\bb{L}^2}^2 \right) \right] 
    +
    \frac12  \sum_{\ell=1}^n \bb{E} \left[\one_{\Omega_{\kappa,\ell-1}} \norm{\bff{\theta}^\ell-\bff{\theta}^{\ell-1}}{\bb{L}^2}^2 \right]
    \nonumber\\
    &\quad
    +
    k
    \bb{E}\left[\sum_{\ell=1}^n \one_{\Omega_{\kappa,\ell-1}} \norm{\partial_x \bff{\theta}^\ell}{\bb{L}^2}^2 \right]
    +
    k
    \bb{E}\left[\sum_{\ell=1}^n \one_{\Omega_{\kappa,\ell-1}} \norm{\bff{\theta}^\ell}{\bb{L}^2}^2 \right]
    +
    k
    \bb{E}\left[\sum_{\ell=1}^n \one_{\Omega_{\kappa,\ell-1}} \norm{\abs{\bff{u}_h^{\ell-1}} \abs{\bff{\theta}^\ell}}{\bb{L}^2}^2 \right]
    \nonumber\\
    &\leq
    \frac12 \bb{E} \left[\norm{\bff{\theta}^0}{\bb{L}^2}^2 \right]
    -
    \bb{E}\left[\max_{m\leq n} \sum_{\ell=1}^m \one_{\Omega_{\kappa,\ell-1}} \inpro{\bff{\rho}^\ell - \bff{\rho}^{\ell-1}}{\bff{\theta}^\ell} \right]
    \nonumber\\
    &\quad
    -
    \bb{E}\left[\max_{m\leq n} \sum_{\ell=1}^m \one_{\Omega_{\kappa,\ell-1}} \int_{t_{\ell-1}}^{t_\ell} \inpro{\partial_x \bff{u}(s)- \partial_x \bff{u}(t_\ell) + \partial_x \bff{\rho}^\ell}{\partial_x \bff{\theta}^\ell} \ds\right]
    \nonumber\\
    &\quad
    -
    \bb{E}\left[\max_{m\leq n} \sum_{\ell=1}^m \one_{\Omega_{\kappa,\ell-1}} \int_{t_{\ell-1}}^{t_\ell} \inpro{\bff{u}(s)- \bff{u}(t_\ell)+\bff{\rho}^\ell}{\bff{\theta}^\ell} \ds\right]
    \nonumber\\
    &\quad
    +
    \bb{E}\left[\max_{m\leq n} \sum_{\ell=1}^m \one_{\Omega_{\kappa,\ell-1}} \int_{t_{\ell-1}}^{t_\ell} \inpro{\big(\bff{u}(s)-{\bff{u}(t_{\ell})} + \bff{\rho}^{\ell}+ \bff{\theta}^{\ell}\big) \times \partial_{xx} \bff{u}(s)}{\bff{\theta}^\ell} \ds\right]
    \nonumber\\
    &\quad
    +
    \bb{E}\left[\max_{m\leq n} \sum_{\ell=1}^m \one_{\Omega_{\kappa,\ell-1}} \int_{t_{\ell-1}}^{t_\ell} \inpro{\bff{u}_h^{\ell}\times \big(\partial_{xx} \bff{u}(s)-\partial_{xx}^h \bff{u}_h^\ell\big)}{\bff{\theta}^\ell} \ds \right]
    \nonumber\\
    &\quad
    -
    \bb{E}\left[\max_{m\leq n} \sum_{\ell=1}^m \one_{\Omega_{\kappa,\ell-1}} \int_{t_{\ell-1}}^{t_\ell} \inpro{\big(\bff{u}(s)+\bff{u}_h^{\ell-1}\big) \cdot \big(\bff{u}(s)-\bff{u}(t_{\ell-1})+ \bff{\rho}^{\ell-1} + \bff{\theta}^{\ell-1}\big) \bff{u}(s)}{\bff{\theta}^\ell} \ds \right]
    \nonumber\\
    &\quad
    -
   	\bb{E}\left[\max_{m\leq n} \sum_{\ell=1}^m \one_{\Omega_{\kappa,\ell-1}} \int_{t_{\ell-1}}^{t_\ell} \inpro{|\bff{u}_h^{\ell-1}|^2 \bff{\rho}^\ell}{\bff{\theta}^\ell} \ds \right]
    \nonumber\\
    &\quad
    -
   \bb{E}\left[\max_{m\leq n} \sum_{\ell=1}^m \one_{\Omega_{\kappa,\ell-1}} \int_{t_{\ell-1}}^{t_\ell} \inpro{|\bff{u}_h^{\ell-1}|^2 \big(\bff{u}(s)-\bff{u}(t_\ell)\big)}{\bff{\theta}^\ell} \ds \right]
    \nonumber\\
    &\quad
    +
    \frac12 \bb{E}\left[\max_{m\leq n} \sum_{\ell=1}^m \one_{\Omega_{\kappa,\ell-1}} \int_{t_{\ell-1}}^{t_\ell} \inpro{\big( (\bff{u}(s)-\bff{u}(t_{\ell})+ \bff{\rho}^{\ell}+ \bff{\theta}^{\ell})\times \bff{g}\big)\times \bff{g}}{\bff{\theta}^\ell} \ds \right]
    \nonumber\\
    &\quad
    +
    \bb{E}\left[\max_{m\leq n} \sum_{\ell=1}^m \one_{\Omega_{\kappa,\ell-1}} \int_{t_{\ell-1}}^{t_\ell} \inpro{(\bff{u}(s)-\bff{u}(t_{\ell-1})+ \bff{\rho}^{\ell-1}+ \bff{\theta}^{\ell-1})\times \bff{g}}{\bff{\theta}^\ell} \dW_s \right]
    \nonumber\\
    &=: \frac12 \bb{E} \left[\norm{\bff{\theta}^0}{\bb{L}^2}^2 \right]+ I_1+I_2+\cdots+ I_{10}.
\end{align}
We can obtain bounds for most of the terms on the last line by the same argument as in the proof of Proposition~\ref{pro:E theta n L2}, except that we now rely on the alternative error decomposition~\eqref{equ:split u Ritz} that leads to \eqref{equ:proj Ritz zero}. We now give further details only on the terms which are estimated differently. Let $\delta>0$. The term $I_1$ can be bounded by Young's inequality, \eqref{equ:Ritz approx}, and the assumption $h=O(k^{\frac34})$ as follows:
\begin{align}\label{equ:E rho theta}
   \abs{I_1}
    &\leq
    Cnh^4 k^{-1}
    \bb{E}\left[\norm{\bff{u}}{L^\infty_T(\bb{H}^2)}^2 \right]
    + 
    \delta k\sum_{\ell=1}^n \bb{E}\left[\one_{\Omega_{\kappa,\ell-1}}\norm{\bff{\theta}^\ell}{\bb{L}^2}^2\right]
    \nonumber\\
    &\leq 
    Ck^{2\alpha}+ \delta k\sum_{\ell=1}^n \bb{E}\left[\one_{\Omega_{\kappa,\ell-1}} \norm{\bff{\theta}^\ell}{\bb{L}^2}^2\right].
\end{align}
For the term $I_2$, we note \eqref{equ:proj Ritz zero} and apply Young's inequality to obtain
\begin{align*}
    \abs{I_2}
    &\leq
    Ck^{2\alpha} \bb{E}\left[\norm{\bff{u}}{\mathcal{C}^\alpha_T(\bb{H}^1)}^2 \right]
    +
    \delta k \bb{E}\left[\sum_{\ell=1}^n \one_{\Omega_{\kappa, \ell-1}} \norm{\partial_x \bff{\theta}^\ell}{\bb{L}^2}^2 \right]
    \\
    &\leq
    Ck^{2\alpha}
    +
    \delta k \bb{E}\left[\sum_{\ell=1}^n \one_{\Omega_{\kappa, \ell-1}} \norm{\partial_x \bff{\theta}^\ell}{\bb{L}^2}^2 \right].
\end{align*}
Similarly, by \eqref{equ:Ritz approx},
\begin{align*}
    \abs{I_3}
    &\leq
    Ck^{2\alpha} \bb{E}\left[\norm{\bff{u}}{\mathcal{C}^\alpha_T(\bb{L}^2)}^2 \right]
    +
    Ch^4 \bb{E}\left[\norm{\bff{u}}{L^\infty_T(\bb{H}^2)}^2 \right]
    +
    \delta k \bb{E}\left[\sum_{\ell=1}^n \one_{\Omega_{\kappa, \ell-1}} \norm{\bff{\theta}^\ell}{\bb{L}^2}^2 \right].
\end{align*}
For $I_4$, by Young's inequality, \eqref{equ:Ritz approx}, and the regularity of $\bff{u}$, we have
\begin{align*}
    \abs{I_4}
    &\leq
    \bb{E}\left[ \sum_{\ell=1}^n \one_{\Omega_{\kappa,\ell-1}} \int_{t_{\ell-1}}^{t_\ell} \left(\norm{\bff{u}(s)-\bff{u}(t_\ell)}{\bb{L}^2} + \norm{\bff{\rho}^\ell}{\bb{L}^2} \right) \norm{\partial_{xx} \bff{u}(s)}{\bb{L}^\infty} \norm{\bff{\theta}^\ell}{\bb{L}^2} \ds \right]
    \\
    &\leq
    Ck^{2\alpha} \bb{E}\left[\norm{\bff{u}}{\mathcal{C}^\alpha_T(\bb{L}^2)}^4 + \norm{\bff{u}}{L^\infty_T(\bb{H}^3)}^4 \right] 
    +
    Ch^4 \bb{E}\left[\norm{\bff{u}}{L^\infty_T(\bb{H}^3)}^4 \right]
    +
    \delta k \bb{E}\left[\sum_{\ell=1}^n \one_{\Omega_{\kappa, \ell-1}} \norm{\bff{\theta}^\ell}{\bb{L}^2}^2 \right]
    \\
    &\leq
    C(h^4+k^{2\alpha})
    +
    \delta k \bb{E}\left[\sum_{\ell=1}^n \one_{\Omega_{\kappa, \ell-1}} \norm{\bff{\theta}^\ell}{\bb{L}^2}^2 \right].
\end{align*}
Next, we estimate $I_5$. Let $\widetilde{\delta}>0$ to be chosen later. Note that since~$\partial_{xx}^h R_h=\Pi_h \partial_{xx}$, it follows that
\begin{align}\label{equ:dxx min dxxh}
\partial_{xx} \bff{u}(t_{\ell})- \partial_{xx}^h \bff{u}_h^\ell= (I-\Pi_h)\partial_{xx} \bff{u}(t_\ell)+ \partial_{xx}^h \bff{\theta}^\ell,
\end{align}
where $I$ is the identity operator and $\Pi_h$ is the $\bb{L}^2$-projection. By H\"older's inequality, and writing $\bff{u}_h^\ell= (\bff{u}_h^\ell- \bff{u}_h^{\ell-1}) + \bff{u}_h^{\ell-1}$, we obtain
\begin{align*}
    \abs{I_5}
    &\leq
    \bb{E}\left[ \sum_{\ell=1}^n \one_{\Omega_{\kappa,\ell-1}} \int_{t_{\ell-1}}^{t_\ell} \norm{\bff{u}_h^\ell}{\bb{L}^\infty} \norm{\partial_{xx} \bff{u}(s)-\partial_{xx} \bff{u}(t_\ell)}{\bb{L}^2} \norm{\bff{\theta}^\ell}{\bb{L}^2} \ds \right]
    \\
    &\quad
    +
    \bb{E}\left[ \sum_{\ell=1}^n \one_{\Omega_{\kappa,\ell-1}} \int_{t_{\ell-1}}^{t_\ell} \norm{\bff{u}_h^\ell}{\bb{L}^\infty} \left( \norm{(I-\Pi_h)\partial_{xx} \bff{u}(t_\ell)}{\bb{L}^2} + \norm{\partial_{xx}^h \bff{\theta}^\ell}{\bb{L}^2} \right) \norm{\bff{\theta}^\ell}{\bb{L}^2} \ds \right]
	\\
	&\leq
    C \bb{E}\left[\sum_{\ell=1}^n \left(k^\alpha \norm{\bff{u}}{\mathcal{C}^\alpha_T(\bb{H}^2)}\right) \left(k^{\frac12} \norm{\bff{u}_h^{\ell}}{\bb{L}^\infty}\right) \left(\one_{\Omega_{\kappa,\ell-1}} k^{\frac12} \norm{\bff{\theta}^\ell}{\bb{L}^2} \right) \right] 
    \\
    &\quad
    +
    C \bb{E}\left[\sum_{\ell=1}^n \left(h^2 \norm{\bff{u}}{L^\infty_T(\bb{H}^4)}\right) \left(k^{\frac12} \norm{\bff{u}_h^{\ell}}{\bb{L}^\infty}\right) \left(\one_{\Omega_{\kappa,\ell-1}} k^{\frac12} \norm{\bff{\theta}^\ell}{\bb{L}^2} \right) \right]
    \\
    &\quad
    +
    C \bb{E}\left[\sum_{\ell=1}^n \norm{\bff{u}_h^{\ell}-\bff{u}_h^{\ell-1}}{\bb{L}^\infty} \left(k^{\frac12} \norm{\partial_{xx}^h \bff{\theta}^\ell}{\bb{L}^2}\right) \left( \one_{\Omega_{\kappa,\ell-1}} k^{\frac12}\norm{\bff{\theta}^\ell}{\bb{L}^2} \right) \right]
     \\
    &\quad
    +
    C \bb{E}\left[\sum_{\ell=1}^n \one_{\Omega_{\kappa,\ell-1}} \norm{\bff{u}_h^{\ell-1}}{\bb{L}^\infty} \left(k^{\frac12} \norm{\partial_{xx}^h \bff{\theta}^\ell}{\bb{L}^2}\right) \left( \one_{\Omega_{\kappa,\ell-1}} k^{\frac12}\norm{\bff{\theta}^\ell}{\bb{L}^2} \right) \right]
    \\
    &=: I_{5a}+I_{5b}+I_{5c}+I_{5d}.
\end{align*}
The terms $I_{5a}$ and $I_{5b}$ can be bounded using Young’s inequality together with {the embedding $\bb{H}^1 \hookrightarrow \bb{L}^\infty$, yielding}
\begin{align*}
    \abs{I_{5a}}
    &\leq
     Ck^{2\alpha} \bb{E}\left[\norm{\bff{u}}{\mathcal{C}^\alpha_T(\bb{H}^2)}^4 + \max_{m\leq n} \norm{\bff{u}_h^m}{\bb{H}^1}^4 \right]
     +
    \delta  k \bb{E} \left[\sum_{\ell=1}^n \one_{\Omega_{\kappa, \ell-1}} \norm{\bff{\theta}^{\ell}}{\bb{L}^2}^2 \right],
    \\
    \abs{I_{5b}}
    &\leq
    Ch^4 \bb{E}\left[\norm{\bff{u}}{L^\infty_T(\bb{H}^4)}^4 + \max_{m\leq n} \norm{\bff{u}_h^m}{\bb{H}^1}^4 \right]
    +
    \delta  k \bb{E} \left[\sum_{\ell=1}^n \one_{\Omega_{\kappa, \ell-1}} \norm{\bff{\theta}^{\ell}}{\bb{L}^2}^2 \right].
\end{align*}
For $I_{5c}$, applying Young's inequality and the embedding $\bb{H}^1\hookrightarrow \bb{L}^\infty$, we obtain, for any $\widetilde{\delta}>0$,
\begin{align*}
    \abs{I_{5c}}
    &\leq
    Ck \bb{E}\left[\max_{m\leq n} \norm{\bff{\theta}^m}{\bb{L}^2}^2 \sum_{\ell=1}^n \norm{\bff{u}_h^\ell-\bff{u}_h^{\ell-1}}{\bb{H}^1}^2 \right]
    +
    \widetilde{\delta} k \bb{E}\left[\sum_{\ell=1}^n \one_{\Omega_{\kappa, \ell-1}} \norm{\partial_{xx}^h \bff{\theta}^\ell}{\bb{L}^2}^2 \right]
    \\
    &\leq
    Ck \bb{E}\left[\max_{m\leq n} \norm{\bff{u}_h^m}{\bb{L}^2}^4 + \norm{\bff{u}}{L^\infty_T(\bb{H}^1)}^4 + \left(\sum_{\ell=1}^n \norm{\bff{u}_h^\ell-\bff{u}_h^{\ell-1}}{\bb{H}^1}^2\right)^2 \right]
    +
    \widetilde{\delta} k \bb{E}\left[\sum_{\ell=1}^n \one_{\Omega_{\kappa, \ell-1}} \norm{\partial_{xx}^h \bff{\theta}^\ell}{\bb{L}^2}^2 \right]
    \\
    &\leq
    Ck +
    \widetilde{\delta} k \bb{E}\left[\sum_{\ell=1}^n \one_{\Omega_{\kappa, \ell-1}} \norm{\partial_{xx}^h \bff{\theta}^\ell}{\bb{L}^2}^2 \right],
\end{align*}
where we also used the definition of $\bff{\theta}^m$ and the boundedness of Ritz projection in $\bb{H}^1$ in the second inequality, and invoked Lemma~\ref{lem:implicit stab H1} in the final step.
For the term $I_{5d}$, we utilise \eqref{equ:implicit Omega k m} and Young's inequality to obtain
\begin{align*}
    \abs{I_{5d}}
    &\leq
    C \kappa k
    \bb{E}\left[ \sum_{\ell=1}^{n} \one_{\Omega_{\kappa, \ell-1}} \norm{\bff{\theta}^\ell}{\bb{L}^2}^2 \right]
    +
    \widetilde{\delta} k \bb{E}\left[\sum_{\ell=1}^n \one_{\Omega_{\kappa, \ell-1}} \norm{\partial_{xx}^h \bff{\theta}^\ell}{\bb{L}^2}^2 \right].
\end{align*}
The remaining terms on the right-hand side of \eqref{equ:12 theta n implicit} can be estimated by the same argument as in the proof of Proposition~\ref{pro:E theta n L2}. Altogether, taking $\delta>0$ sufficiently small, and rearranging the terms, we obtain
\begin{align}\label{equ:E max theta L2}
    &\bb{E}\left[\max_{m\leq n} \left( \one_{\Omega_{\kappa,m-1}} \norm{\bff{\theta}^m}{\bb{L}^2}^2 \right) \right]  
    +
    k  \sum_{\ell=1}^n \bb{E} \left[ \one_{\Omega_{\kappa,\ell-1}} \norm{\bff{\theta}^\ell}{\bb{H}^1}^2 \right]
    +
    k
    \bb{E}\left[\sum_{\ell=1}^n \one_{\Omega_{\kappa,\ell-1}} \norm{\abs{\bff{u}_h^{\ell-1}} \abs{\bff{\theta}^\ell}}{\bb{L}^2}^2 \right]
    \nonumber\\
    &\leq
    C (h^4+k^{2\alpha})
    +
    C\kappa^2 k \sum_{\ell=1}^n \bb{E} \left[\one_{\Omega_{\kappa, \ell-1}} \norm{\bff{\theta}^\ell}{\bb{L}^2}^2 \right]
    +
    \widetilde{\delta} k \bb{E}\left[\sum_{\ell=1}^n \one_{\Omega_{\kappa, \ell-1}} \norm{\partial_{xx}^h \bff{\theta}^\ell}{\bb{L}^2}^2 \right],
\end{align}
where $C$ is independent of $n$, $h$, $k$, and $\kappa$, while $\widetilde{\delta}>0$ is to be chosen later. This estimate will be used again later.

Next, we put $\bff{\phi}=-\partial_{xx}^h \bff{\theta}^\ell$ in~\eqref{equ:theta n theta n1 implicit} and perform the usual procedure. Noting \eqref{equ:dxx min dxxh} and the error decomposition in \eqref{equ:split u Ritz}, we have
\begin{align}\label{equ:implicit 12 theta n}
    &\frac12 \bb{E}\left[\max_{m\leq n} \left( \one_{\Omega_{\kappa,m-1}} \norm{\partial_x \bff{\theta}^m}{\bb{L}^2}^2 \right) \right] 
    +
    \frac12  \sum_{\ell=1}^n \bb{E} \left[\one_{\Omega_{\kappa,\ell-1}} \norm{\partial_x \bff{\theta}^\ell- \partial_x \bff{\theta}^{\ell-1}}{\bb{L}^2}^2 \right]
    \nonumber\\
    &\quad
    +
    k
    \bb{E}\left[\sum_{\ell=1}^n \one_{\Omega_{\kappa,\ell-1}} \norm{\partial_{xx}^h \bff{\theta}^\ell}{\bb{L}^2}^2 \right]
    +
    k
    \bb{E}\left[\sum_{\ell=1}^n \one_{\Omega_{\kappa,\ell-1}} \norm{\partial_x \bff{\theta}^\ell}{\bb{L}^2}^2 \right]
    +
    k
    \bb{E}\left[\sum_{\ell=1}^n \one_{\Omega_{\kappa,\ell-1}} \norm{\abs{\bff{u}_h^{\ell-1}} \abs{\bff{\theta}^\ell}}{\bb{L}^2}^2 \right]
    \nonumber\\
    &\leq
    \frac12 \bb{E} \left[\norm{\bff{\theta}^0}{\bb{H}^1}^2 \right]
    +
    \sum_{\ell=1}^n \bb{E} \left[\one_{\Omega_{\kappa,\ell-1}} \inpro{\bff{\rho}^\ell-\bff{\rho}^{\ell-1}}{\partial_{xx}^h \bff{\theta}^\ell}\right]
    \nonumber\\
    &\quad
    +
    \bb{E}\left[\max_{m\leq n} \sum_{\ell=1}^m \one_{\Omega_{\kappa,\ell-1}} \int_{t_{\ell-1}}^{t_\ell} \inpro{\partial_{x} \bff{u}(s)-\partial_{x} \bff{u}(t_\ell)}{\partial_x \partial_{xx}^h \bff{\theta}^\ell} \ds\right]
    \nonumber\\
    &\quad
    -
    \bb{E}\left[\max_{m\leq n} \sum_{\ell=1}^m \one_{\Omega_{\kappa,\ell-1}} \int_{t_{\ell-1}}^{t_\ell} \inpro{(I-\Pi_h)\partial_{xx} \bff{u}(t_\ell)}{\partial_{xx}^h \bff{\theta}^\ell} \ds\right]
    \nonumber\\
    &\quad
    +
    \bb{E}\left[\max_{m\leq n} \sum_{\ell=1}^m \one_{\Omega_{\kappa,\ell-1}} \int_{t_{\ell-1}}^{t_\ell} \inpro{\bff{u}(s)- \bff{u}(t_\ell) + \bff{\rho}^\ell+ \bff{\theta}^\ell}{\partial_{xx}^h \bff{\theta}^\ell} \ds\right]
    \nonumber\\
    &\quad
    -
    \bb{E}\left[\max_{m\leq n} \sum_{\ell=1}^m \one_{\Omega_{\kappa,\ell-1}} \int_{t_{\ell-1}}^{t_\ell} \inpro{\big(\bff{u}(s)-\bff{u}(t_{\ell})+ \bff{\rho}^{\ell}+ \bff{\theta}^{\ell}\big) \times \partial_{xx} \bff{u}(s)}{\partial_{xx}^h \bff{\theta}^\ell} \ds\right]
    \nonumber\\
    &\quad
    -
    \bb{E}\left[\max_{m\leq n} \sum_{\ell=1}^m \one_{\Omega_{\kappa,\ell-1}} \int_{t_{\ell-1}}^{t_\ell} \inpro{\bff{u}_h^{\ell} \times \big((I-\Pi_h)\partial_{xx} \bff{u}(t_\ell) + \partial_{xx}^h \bff{\theta}^\ell \big)}{\partial_{xx}^h \bff{\theta}^\ell} \ds \right]
    \nonumber\\
    &\quad
    +
    \bb{E}\left[\max_{m\leq n} \sum_{\ell=1}^m \one_{\Omega_{\kappa,\ell-1}} \int_{t_{\ell-1}}^{t_\ell} \inpro{\big(\bff{u}(s)+\bff{u}_h^{\ell-1}\big) \cdot \big(\bff{u}(s)-\bff{u}(t_{\ell-1})+ \bff{\rho}^{\ell-1} + \bff{\theta}^{\ell-1}\big) \bff{u}(s)}{\partial_{xx}^h \bff{\theta}^\ell} \ds \right]
    \nonumber\\
    &\quad
    +
    \bb{E}\left[\max_{m\leq n} \sum_{\ell=1}^m \one_{\Omega_{\kappa,\ell-1}} \int_{t_{\ell-1}}^{t_\ell} \inpro{|\bff{u}_h^{\ell-1}|^2 \big(\bff{\rho}^\ell+ \bff{\theta}^\ell\big)}{\partial_{xx}^h \bff{\theta}^\ell} \ds \right]
    \nonumber\\
    &\quad
    +
   \bb{E}\left[\max_{m\leq n} \sum_{\ell=1}^m \one_{\Omega_{\kappa,\ell-1}} \int_{t_{\ell-1}}^{t_\ell} \inpro{|\bff{u}_h^{\ell-1}|^2 \big(\bff{u}(s)-\bff{u}(t_\ell)\big)}{\partial_{xx}^h \bff{\theta}^\ell} \ds \right]
    \nonumber\\
    &\quad
    -
    \frac12 \bb{E}\left[\max_{m\leq n} \sum_{\ell=1}^m \one_{\Omega_{\kappa,\ell-1}} \int_{t_{\ell-1}}^{t_\ell} \inpro{\big( (\bff{u}(s)-\bff{u}(t_{\ell})+ \bff{\rho}^{\ell}+ \bff{\theta}^{\ell})\times \bff{g}\big)\times \bff{g}}{\partial_{xx}^h \bff{\theta}^\ell} \ds \right]
    \nonumber\\
    &\quad
    -
    \bb{E}\left[\max_{m\leq n} \sum_{\ell=1}^m \one_{\Omega_{\kappa,\ell-1}} \int_{t_{\ell-1}}^{t_\ell} \inpro{(\bff{u}(s)-\bff{u}(t_{\ell-1})+ \bff{\rho}^{\ell-1}+ \bff{\theta}^{\ell-1})\times \bff{g}}{\partial_{xx}^h \bff{\theta}^\ell} \dW_s \right]
    \nonumber\\
    &=: \frac12 \bb{E} \left[\norm{\bff{\theta}^0}{\bb{H}^1}^2 \right]+ J_1+J_2+\cdots+ J_{11}.
\end{align}
Recall that $\widetilde{\delta}>0$ is a constant to be specified later.
The term $J_1$ can be estimated by applying the same argument as in~\eqref{equ:E rho theta}, leading to
\begin{align*}
    \sum_{\ell=1}^n \bb{E} \left[\one_{\Omega_{\kappa,\ell-1}} \inpro{\bff{\rho}^\ell-\bff{\rho}^{\ell-1}}{\partial_{xx}^h \bff{\theta}^\ell}\right]
    &\leq
    Cnh^4 k^{-1}+ \widetilde{\delta} k\sum_{\ell=1}^n \bb{E}\left[\one_{\Omega_{\kappa,\ell-1}}\norm{\partial_{xx}^h \bff{\theta}^\ell}{\bb{L}^2}^2\right]
    \nonumber\\
    &\leq 
    Ck^{2\alpha}+ \widetilde{\delta} k\sum_{\ell=1}^n \bb{E}\left[\one_{\Omega_{\kappa,\ell-1}} \norm{\partial_{xx}^h \bff{\theta}^\ell}{\bb{L}^2}^2\right].
\end{align*}
For the term $J_2$, we integrate by parts and apply Young's inequality to obtain
\begin{align*}
    \abs{J_2}
    \leq
    Ck^{2\alpha} \bb{E}\left[\norm{\bff{u}}{\mathcal{C}^\alpha_T(\bb{H}^2)}^2 \right]
     +
    \widetilde{\delta} k\sum_{\ell=1}^n \bb{E}\left[\one_{\Omega_{\kappa,\ell-1}} \norm{\partial_{xx}^h \bff{\theta}^\ell}{\bb{L}^2}^2\right].
\end{align*}
The terms $J_3$ and $J_4$ can be estimated by Young's inequality and \eqref{equ:proj approx} as
\begin{align*}
    \abs{J_3}+\abs{J_4}
    &\leq
    Ch^4 \bb{E}\left[\norm{\bff{u}}{L^\infty_T(\bb{H}^4)}^2\right] 
    +
    Ck^{2\alpha} \bb{E}\left[\norm{\bff{u}}{\mathcal{C}^\alpha_T(\bb{H}^1)}^2\right]
    \\
    &\quad
    +
    Ck \sum_{\ell=1}^n \bb{E}\left[\one_{\Omega_{\kappa,\ell-1}} \norm{\bff{\theta}^\ell}{\bb{L}^2}^2\right]
    +
    \widetilde{\delta} k\sum_{\ell=1}^n \bb{E}\left[\one_{\Omega_{\kappa,\ell-1}} \norm{\partial_{xx}^h \bff{\theta}^\ell}{\bb{L}^2}^2\right].
\end{align*}
For the term $J_5$, noting the regularity of $\bff{u}$, by Young's inequality, the embedding $\bb{H}^1\hookrightarrow \bb{L}^\infty$, and \eqref{equ:implicit Omega k m}, we have
\begin{align*}
    \abs{J_5}
    &\leq
    \bb{E}\left[\sum_{\ell=1}^n \one_{\Omega_{\kappa,\ell-1}} \int_{t_{\ell-1}}^{t_\ell} \norm{\bff{u}(s)-\bff{u}(t_\ell)}{\bb{L}^\infty} \norm{\partial_{xx}\bff{u}(s)}{\bb{L}^2} \norm{\partial_{xx}^h \bff{\theta}^\ell}{\bb{L}^2} \ds \right]
    \\
    &\quad
    +
    \bb{E}\left[\sum_{\ell=1}^n \one_{\Omega_{\kappa,\ell-1}} \int_{t_{\ell-1}}^{t_\ell}  \norm{\bff{\rho}^\ell}{\bb{L}^2}  \norm{\partial_{xx}\bff{u}(s)}{\bb{L}^\infty} \norm{\partial_{xx}^h \bff{\theta}^\ell}{\bb{L}^2} \ds \right]
    \\
    &\quad
    +
    \bb{E}\left[\sum_{\ell=1}^n \one_{\Omega_{\kappa,\ell-1}} \int_{t_{\ell-1}}^{t_\ell}  \norm{\bff{\theta}^\ell}{\bb{L}^\infty} \norm{\partial_{xx}\bff{u}(s)}{\bb{L}^2} \norm{\partial_{xx}^h \bff{\theta}^\ell}{\bb{L}^2} \ds \right]
    \\
    &=:
    J_{5a}+ J_{5b} + J_{5c}.
\end{align*}
The terms $J_{5a}$ and $J_{5b}$ can be estimated in the usual way by by Young's inequality, H\"older continuity in time of $\bff{u}$, and the embedding $\bb{H}^1\hookrightarrow \bb{L}^\infty$, leading to
\begin{align*}
    \abs{J_{5a}}
    &\leq
    Ck^{2\alpha} 
    \bb{E}\left[\norm{\bff{u}}{\mathcal{C}^\alpha_T(\bb{H}^1)}^4 + \norm{\bff{u}}{L^\infty_T(\bb{H}^2)}^4 \right]
    +
    \widetilde{\delta} k\sum_{\ell=1}^n \bb{E}\left[\one_{\Omega_{\kappa,\ell-1}} \norm{\partial_{xx}^h \bff{\theta}^\ell}{\bb{L}^2}^2\right],
    \\
    \abs{J_{5b}}
    &\leq
    Ch^4 \bb{E}\left[\norm{\bff{u}}{L^\infty_T(\bb{H}^3)}^4 \right]
    +
    \widetilde{\delta} k\sum_{\ell=1}^n \bb{E}\left[\one_{\Omega_{\kappa,\ell-1}} \norm{\partial_{xx}^h \bff{\theta}^\ell}{\bb{L}^2}^2\right].
\end{align*}
For the term $J_{5c}$, we employ the same argument as in the term $I_{3c}$ in \eqref{equ:I3 holder} to obtain
\begin{align*}
    \abs{J_{5c}}
    &\leq
    Ck^{2\alpha}
    +
    C\kappa^2 k \bb{E} \left[\sum_{\ell=1}^n \one_{\Omega_{\kappa,\ell-1}} \norm{\bff{\theta}^\ell}{\bb{L}^2}^2 \right]
    +
    \widetilde{\delta} k\sum_{\ell=1}^n \bb{E}\left[\one_{\Omega_{\kappa,\ell-1}} \norm{\partial_{xx}^h \bff{\theta}^\ell}{\bb{L}^2}^2\right].
\end{align*}
Next, noting \eqref{equ:proj approx} and the identity $(\bff{a}\times \bff{b})\cdot \bff{b}=0$ for any $\bff{a},\bff{b}\in \bb{R}^3$, the term $J_6$ is bounded by
\begin{align*}
    \abs{J_6}
    &\leq
    \bb{E}\left[\sum_{\ell=1}^n \one_{\Omega_{\kappa,\ell-1}} \int_{t_{\ell-1}}^{t_\ell} \norm{\bff{u}_h^{\ell}}{\bb{L}^\infty} \norm{\big((I-\Pi_h)\partial_{xx} \bff{u}(t_\ell)}{\bb{L}^2} \norm{\partial_{xx}^h \bff{\theta}^\ell}{\bb{L}^2}\ds \right]
    \\
    &\leq
    Ch^4 \bb{E}\left[\max_{m\leq n} \norm{\bff{u}_h^{m}}{\bb{L}^\infty}^4 + \norm{\bff{u}}{L^\infty_T(\bb{H}^4)}^4 \right] 
    +
    \delta k\sum_{\ell=1}^n \bb{E}\left[\one_{\Omega_{\kappa,\ell-1}} \norm{\partial_{xx}^h \bff{\theta}^\ell}{\bb{L}^2}^2\right]
    \\
    &\leq
    Ch^4 
    +
    \widetilde{\delta}  k\sum_{\ell=1}^n \bb{E}\left[\one_{\Omega_{\kappa,\ell-1}} \norm{\partial_{xx}^h \bff{\theta}^\ell}{\bb{L}^2}^2\right],
\end{align*}
where in the last step we used the embedding $\bb{H}^1\hookrightarrow \bb{L}^\infty$, Lemma~\ref{lem:implicit stab H1}, and the regularity of $\bff{u}$. For the term $J_7$, we apply similar argument as in~\eqref{equ:term I5} to obtain
\begin{align*}
    \abs{J_7}
    &\leq
    C(h^4+k^{2\alpha}) +
    C\kappa^2 k \sum_{\ell=1}^n \bb{E}\left[\one_{\Omega_{\kappa,\ell-1}} \norm{\bff{\theta}^\ell}{\bb{L}^2}^2\right]
    +
    \widetilde{\delta} k\sum_{\ell=1}^n \bb{E}\left[\one_{\Omega_{\kappa,\ell-1}} \norm{\partial_{xx}^h \bff{\theta}^\ell}{\bb{L}^2}^2\right].
\end{align*}
For the term $J_8$, a straightforward application of Young's inequality, \eqref{equ:implicit Omega k m}, the embedding~$\bb{H}^1\hookrightarrow \bb{L}^\infty$, and Lemma~\ref{lem:implicit stab H1} yields
\begin{align*}
    \abs{J_8}
    &\leq
    \bb{E}\left[ \sum_{\ell=1}^n \one_{\Omega_{\kappa,\ell-1}} \int_{t_{\ell-1}}^{t_\ell} \norm{\bff{u}_h^{\ell-1}}{\bb{L}^\infty}^2 \left( \norm{\bff{\rho}^{\ell}}{\bb{L}^2} + \norm{\bff{\theta}^\ell}{\bb{L}^2} \right) \norm{\partial_{xx}^h \bff{\theta}^\ell}{\bb{L}^2} \ds \right]
	\\
	&\leq
	Ch^4 + C\kappa^2 k \sum_{\ell=1}^n \bb{E}\left[\one_{\Omega_{\kappa,\ell-1}} \norm{\bff{\theta}^\ell}{\bb{L}^2}^2\right]
    +
    \widetilde{\delta}  k\sum_{\ell=1}^n \bb{E}\left[\one_{\Omega_{\kappa,\ell-1}} \norm{\partial_{xx}^h \bff{\theta}^\ell}{\bb{L}^2}^2\right].
\end{align*}
For the term $J_9$, similarly we have
\begin{align*}
    \abs{J_9}
    &\leq
    \bb{E}\left[ \sum_{\ell=1}^n \one_{\Omega_{\kappa,\ell-1}} \int_{t_{\ell-1}}^{t_\ell} \norm{\bff{u}_h^{\ell-1}}{\bb{L}^\infty}^2 \norm{\bff{u}(s)-\bff{u}(t_\ell)}{\bb{L}^2} \norm{\partial_{xx}^h \bff{\theta}^\ell}{\bb{L}^2} \ds \right]
    \\
    &\leq
    Ck^{2\alpha} \bb{E}\left[\max_{m\leq n} \norm{\bff{u}_h^m}{\bb{H}^1}^8 + \norm{\bff{u}}{\mathcal{C}^\alpha_T(\bb{L}^2)}^4\right] 
    +
    \widetilde{\delta}  k\sum_{\ell=1}^n \bb{E}\left[\one_{\Omega_{\kappa,\ell-1}} \norm{\partial_{xx}^h \bff{\theta}^\ell}{\bb{L}^2}^2\right]
    \\
    &\leq
    Ck^{2\alpha}
    +
    \widetilde{\delta}  k\sum_{\ell=1}^n \bb{E}\left[\one_{\Omega_{\kappa,\ell-1}} \norm{\partial_{xx}^h \bff{\theta}^\ell}{\bb{L}^2}^2\right].
\end{align*}
The term $J_{10}$ is estimated in a similar manner as the term $J_4$ to obtain
\begin{align*}
    \abs{J_{10}}
    \leq
    C(h^4+k^{2\alpha})
    +
    Ck \sum_{\ell=1}^n \bb{E}\left[\one_{\Omega_{\kappa,\ell-1}} \norm{\bff{\theta}^\ell}{\bb{L}^2}^2\right]
    +
    \widetilde{\delta} k\sum_{\ell=1}^n \bb{E}\left[\one_{\Omega_{\kappa,\ell-1}} \norm{\partial_{xx}^h \bff{\theta}^\ell}{\bb{L}^2}^2\right].
\end{align*}
The stochastic integral term $J_{11}$ is split further as: \begin{align}\label{equ:I8 stoch implicit}
	J_{11}
	&=
	\bb{E}\left[\max_{m\leq n} \sum_{\ell=1}^m \one_{\Omega_{\kappa,\ell-1}} \int_{t_{\ell-1}}^{t_\ell} \inpro{\partial_x \Pi_h \big((\bff{u}(s)-\bff{u}(t_{\ell-1})+ \bff{\rho}^{\ell-1}+ \bff{\theta}^{\ell-1})\times \bff{g}\big)}{\partial_x \bff{\theta}^{\ell-1}} \dW_s \right]
	\nonumber\\
	&\quad
	+
	\bb{E}\left[\max_{m\leq n} \sum_{\ell=1}^m \one_{\Omega_{\kappa,\ell-1}} \int_{t_{\ell-1}}^{t_\ell} \inpro{\partial_x \Pi_h \big((\bff{u}(s)-\bff{u}(t_{\ell-1})+ \bff{\rho}^{\ell-1}+ \bff{\theta}^{\ell-1})\times \bff{g}\big)}{\partial_x \bff{\theta}^\ell- \partial_x \bff{\theta}^{\ell-1}} \dW_s \right]
	\nonumber\\
	&=: J_{11a}+ J_{11b}.
\end{align}
For the term $J_{11a}$, by the $\bb{H}^1$-stability of $\Pi_h$, the H\"older continuity of $\bff{u}$, the Burkholder--Davis--Gundy and the Young inequalities, we have
\begin{align*}
	\abs{J_{11a}}
	&\leq
	C \norm{\bff{g}}{\bb{H}^2} \bb{E}\left[\left( \sum_{\ell=1}^n \one_{\Omega_{\kappa,\ell-1}} \int_{t_{\ell-1}}^{t_\ell} \left(\norm{\bff{u}(s)-\bff{u}(t_{\ell-1})}{\bb{H}^1}^2 + \norm{\bff{\rho}^{\ell-1}}{\bb{H}^1}^2 + \norm{\bff{\theta}^{\ell-1}}{\bb{H}^1}^2 \right) \norm{\partial_x \bff{\theta}^{\ell-1}}{\bb{L}^2}^2 \ds \right)^{\frac12} \right]
	\\
	&\leq
	\widetilde{\delta} \bb{E} \left[\max_{m\leq n} \one_{\Omega_{\kappa,m-1}} \norm{\partial_x \bff{\theta}^m}{\bb{L}^2}^2 \right]
	+
	C(h^2 + k^{2\alpha})
        +
        Ck \sum_{\ell=1}^n \bb{E}\left[\one_{\Omega_{\kappa,\ell-1}} \norm{\bff{\theta}^{\ell-1}}{\bb{H}^1}^2\right].
\end{align*}
For the term $J_{11b}$, we use Young's inequality, It\^o's isometry, and the H\"older continuity of $\bff{u}$ to obtain
\begin{align*}
	\abs{J_{11b}}
	&\leq
	C\bb{E} \left[\sum_{\ell=1}^n \norm{\one_{\Omega_{\kappa,\ell-1}} \int_{t_{\ell-1}}^{t_\ell} \partial_x \Pi_h \left((\bff{u}(s)-\bff{u}(t_{\ell-1})+ \bff{\rho}^{\ell-1}+ \bff{\theta}^{\ell-1})\times \bff{g}\right) \dW_s}{\bb{L}^2}^2 \right] 
	\\
	&\quad
	+
	\widetilde{\delta} \bb{E}\left[\sum_{\ell=1}^n \one_{\Omega_{\kappa,\ell-1}} \norm{\partial_x \bff{\theta}^\ell -\partial_x \bff{\theta}^{\ell-1}}{\bb{L}^2}^2 \right]
	\\
	&\leq
	C\bb{E} \left[\sum_{\ell=1}^n \one_{\Omega_{\kappa,\ell-1}} \int_{t_{\ell-1}}^{t_\ell} \norm{(\bff{u}(s)-\bff{u}(t_{\ell-1})+ \bff{\rho}^{\ell-1}+ \bff{\theta}^{\ell-1})\times \bff{g}}{\bb{H}^1}^2 \ds \right] 
	\\
	&\quad
	+
	\widetilde{\delta} \bb{E}\left[\sum_{\ell=1}^n \one_{\Omega_{\kappa,\ell-1}} \norm{\partial_x \bff{\theta}^\ell -\partial_x \bff{\theta}^{\ell-1}}{\bb{L}^2}^2 \right]
	\\
	&\leq
	C (h^2 + k^{2\alpha})
        +
        Ck \sum_{\ell=1}^n \bb{E}\left[\one_{\Omega_{\kappa,\ell-1}} \norm{\bff{\theta}^{\ell-1}}{\bb{H}^1}^2\right]
	+
	\widetilde{\delta} \bb{E}\left[\sum_{\ell=1}^n \one_{\Omega_{\kappa,\ell-1}} \norm{\partial_x \bff{\theta}^\ell -\partial_x \bff{\theta}^{\ell-1}}{\bb{L}^2}^2 \right].
\end{align*}
We gather all the above estimates for $J_1$ up to $J_{11}$, set $\widetilde{\delta} =1/50$, and substitute them back into~\eqref{equ:implicit 12 theta n}, thus obtaining
\begin{align}\label{equ:E max theta dx L2}
    &\bb{E}\left[\max_{m\leq n} \left( \one_{\Omega_{\kappa,m-1}} \norm{\partial_x \bff{\theta}^m}{\bb{L}^2}^2 \right) \right]  
    +
    \bb{E}\left[\sum_{\ell=1}^n \one_{\Omega_{\kappa,\ell-1}} \norm{\partial_x \bff{\theta}^\ell -\partial_x \bff{\theta}^{\ell-1}}{\bb{L}^2}^2 \right]
    \nonumber\\
    &\quad
    +
    k  \sum_{\ell=1}^n \bb{E} \left[ \one_{\Omega_{\kappa,\ell-1}} \norm{\partial_{xx}^h \bff{\theta}^\ell}{\bb{L}^2}^2 \right]
    +
    k  \sum_{\ell=1}^n \bb{E} \left[ \one_{\Omega_{\kappa,\ell-1}} \norm{\partial_x \bff{\theta}^\ell}{\bb{L}^2}^2 \right]
    \nonumber\\
    &\leq
    \bb{E}\left[\norm{\partial_x \bff{\theta}^0}{\bb{L}^2}^2 \right] 
    +
    C (h^2+k^{2\alpha})
    +
    C\kappa^2 k \sum_{\ell=1}^n \bb{E} \left[\one_{\Omega_{\kappa, \ell-1}} \norm{\bff{\theta}^\ell}{\bb{H}^1}^2 \right].
\end{align}
Adding \eqref{equ:E max theta dx L2} with \eqref{equ:E max theta L2} and rearranging the terms, we infer that
\begin{align*}
     &\bb{E}\left[\max_{m\leq n} \left( \one_{\Omega_{\kappa,m-1}} \norm{\bff{\theta}^m}{\bb{H}^1}^2 \right) \right] 
     +
    \bb{E}\left[\sum_{\ell=1}^n \one_{\Omega_{\kappa,\ell-1}} \norm{\bff{\theta}^\ell -\bff{\theta}^{\ell-1}}{\bb{H}^1}^2 \right]
    \\
    &\quad
    +
    k
    \bb{E}\left[\sum_{\ell=1}^n \one_{\Omega_{\kappa,\ell-1}} \norm{\partial_{xx}^h \bff{\theta}^\ell}{\bb{L}^2}^2 \right]
    +
    k  \sum_{\ell=1}^n \bb{E} \left[ \one_{\Omega_{\kappa,\ell-1}} \norm{\bff{\theta}^\ell}{\bb{H}^1}^2 \right]
    \\
    &\leq
    \bb{E} \left[\norm{\bff{\theta}^0}{\bb{H}^1}^2 \right]
    +
    C(h^2 + k^{2\alpha})
    +
    C_1 \kappa^2 k \sum_{\ell=1}^n \bb{E} \left[\one_{\Omega_{\kappa, \ell-1}} \norm{\bff{\theta}^\ell- \bff{\theta}^{\ell-1}}{\bb{H}^1}^2 + \one_{\Omega_{\kappa, \ell-1}} \norm{\bff{\theta}^{\ell-1}}{\bb{H}^1}^2 \right].
\end{align*}
Therefore, when $k<1/(C_1\kappa^2)$, we deduce that
\begin{align}\label{equ:const C1}
    \bb{E}\left[\max_{m\leq n} \left( \one_{\Omega_{\kappa,m-1}} \norm{\bff{\theta}^m}{\bb{H}^1}^2 \right) \right] 
    \leq
    \bb{E} \left[\norm{\bff{\theta}^0}{\bb{H}^1}^2 \right]
    +
    C(h^2 + k^{2\alpha})
    +
    C_1 \kappa^2 k \sum_{\ell=1}^{n-1} \bb{E} \left[\max_{m\leq \ell} \one_{\Omega_{\kappa, m-1}} \norm{\bff{\theta}^m}{\bb{H}^1}^2 \right].
\end{align}
The required result then follows by the discrete Gronwall lemma.
\end{proof}

Consequently, we have the following results on strong convergence over a large sample space as well as convergence in probability, with rates which appear to be optimal in the traditional sense of finite element analysis.

\begin{corollary}\label{cor:error subset implicit}
Under the same hypotheses as Proposition~\ref{pro:implicit E theta n H1}, for $n\in \{1,2,\ldots,N\}$ we have
\begin{align}\label{equ:E error 1 implicit}
	\bb{E}\left[\max_{m\leq n} \left( \one_{\Omega_{\kappa,m-1}} \norm{\bff{u}(t_m)-\bff{u}_h^m}{\bb{H}^1}^2 \right) \right]  
    \leq
	\widetilde{C} e^{\widetilde{C} \kappa^2} \left(h^2+k^{2\alpha}\right),
\end{align}
for any $\alpha\in (0,\frac12)$. The constant $\widetilde{C}$ depends on $T$, but is independent of $n$, $h$, $k$, and $\kappa$.

In particular, when $h=O(k^{\frac34})$, for the set
\begin{align*}
    \Omega_{h,k}:= \left\{\omega\in \Omega: \max_{t\in [0,T]} \norm{\bff{u}(t)}{\bb{H}^1}^2 \leq \sqrt{\log\big(\log\left[(h^2+k^{2\alpha})^{-1}\right]\big)} \right\}
\end{align*}
which satisfies $\bb{P}\left[\Omega_{h,k}\right] \geq 1- \left(\log\big(\log\left[(h^2+k^{2\alpha})^{-1}\right]\big)\right)^{-\frac12}$, 
we have for every $\delta>0$,
\begin{align}\label{equ:E h kappa}
    \bb{E}\left[ \one_{\Omega_{h,k}} \left( \max_{m\leq n}  \norm{\bff{u}(t_m)-\bff{u}_h^m}{\bb{L}^2}^2 \right) + k \sum_{\ell=1}^n \norm{\partial_x \bff{u}(t_\ell)-\partial_x \bff{u}_h^\ell}{\bb{H}^1}^2 \right]  
	\leq
    C \left(h^2+k^{2\alpha}\right)^{1-\delta}.
\end{align}
\end{corollary}

\begin{proof}
Inequality~\eqref{equ:E error 1 implicit} follows directly from \eqref{equ:proj approx}, Proposition~\ref{pro:implicit E theta n H1}, and the triangle inequality. To prove \eqref{equ:E h kappa}, we observe that if $\kappa=\sqrt{\log\big(\log\left[(h^2+k^{2\alpha})^{-1}\right]\big)}$, then $k=O(\kappa^{-2})$. Consequently, the assumptions underlying the estimate~\eqref{equ:E error 1 implicit} remain valid, which allows us to deduce~\eqref{equ:E h kappa}.
\end{proof}

\begin{corollary}\label{cor:1d prob error implicit}
Suppose that $h=O(k^{\frac34})$ and let $n\in \{1,2,\ldots,N\}$. For every $\delta\in (0,1)$, $\alpha\in (0,\frac12)$, and $\gamma>0$, we have
\begin{align}\label{equ:error prob impli}
	\lim_{h,k\to 0^+} \bb{P} \left[\max_{m\leq n} \norm{\bff{u}(t_m)-\bff{u}_h^m}{\bb{H}^1}^2 \geq \gamma(h^2+k^{2\alpha})^{1-\delta} \right] = 0.
\end{align}
\end{corollary}

\begin{proof}
This follows by the same argument as in Corollary~\ref{cor:1d prob error}, noting the result in Corollary~\ref{cor:error subset implicit}.
\end{proof}

Finally, we deduce the following theorem on the strong convergence of the implicit scheme~\eqref{equ:implicit euler} over the whole sample space. In this case, since \eqref{equ:implicit Omega k m} involves the norm of the numerical solution, the argument based on \eqref{equ:E exp} and the exponential Markov inequality is no longer applicable, and the resulting convergence rate is weaker than that in~\eqref{equ:1d rate} (see Remark~\ref{rem:fast rate}).

\begin{theorem}\label{the:rate 1d implicit}
Let $d=1$ and $\bff{g}\in \bb{H}^4$. Suppose that $\bff{u}$ is the pathwise solution to \eqref{equ:sllb} with initial data $\bff{u}_0\in \bb{H}^4$, and let $\{\bff{u}_h^n\}_n$ be a sequence of random variables solving~\eqref{equ:implicit euler} with $\bff{u}_h^0=\Pi_h \bff{u}_0$. Assume that $h=O(k^\frac34)$. Then the implicit scheme \eqref{equ:implicit euler} converges in $L^2\big(\Omega; \ell^\infty(0,T;\bb{H}^1)\big)$. More precisely, for sufficiently small $h$ and $k$, and for all $\alpha\in (0,\frac12)$, we have
\begin{align}\label{equ:1d rate implicit}
	 \bb{E}\left[\max_{m\leq n} \norm{\bff{u}(t_m)-\bff{u}_h^m}{\bb{H}^1}^2 \right]
	&\leq
	C_r \abs{\log \left(h^2+k^{2\alpha}\right)}^{-r}
\end{align}
for any $r\geq 1$. The constant $C_r$ depends on $T$ and $r$, but is independent of $n$, $h$, and $k$. In particular, the right-hand side of \eqref{equ:1d rate implicit} tends to zero as $h,k\to 0^+$.
\end{theorem}

\begin{proof}
This follows by a similar argument as in Theorem~\ref{the:rate 1d}. By H\"older's inequality with exponents $2^{q-1}$ and $p=2^{q-1}/(2^{q-1}-1)$, for any $q>1$, we obtain
\begin{align}\label{equ:max compl u}
    \bb{E} \left[\max_{m\leq n} \one_{\Omega_{\kappa,m-1}^\complement} \norm{\bff{u}(t_m)- \bff{u}_h^m}{\bb{H}^1}^2 \right]
    \leq
    C\left[\bb{P} \left(\Omega_{\kappa,n-1}^\complement\right)\right]^{\frac{1}{p}} 
    \left[\bb{E}\left(\max_{t\in [0,T]} \norm{\bff{u}(t)}{\bb{H}^1}^{2^q} + \max_{m\leq n} \norm{\bff{u}_h^m}{\bb{H}^1}^{2^q}\right) \right]^{\frac{1}{2^{q-1}}}.
\end{align}
The last terms on the right-hand side of \eqref{equ:max compl u} are bounded due to the assumed regularity of $\bff{u}$ and the stability estimate in Lemma~\ref{lem:implicit stab H1}. It remains to establish a bound for the probability of the `bad' set $\Omega_{\kappa,n-1}^\complement$. To this end, by Chebyshev's inequality and the definition of the set $\Omega_{\kappa,n-1}$ in~\eqref{equ:implicit Omega k m}, we have
\begin{align*}
    \bb{P} \left(\Omega_{\kappa,n-1}^\complement\right)
    \leq
    \kappa^{-2^{q-1}} \left[ \bb{E}\left( \max_{t\in [0,T]} \norm{\bff{u}(t)}{\bb{H}^3}^{2^q}  + \max_{m\leq n} \norm{\bff{u}_h^m}{\bb{H}^1}^{2^q} \right) \right].
\end{align*}
This implies
\begin{align}\label{equ:E Bn}
    \bb{E} \left[\max_{m\leq n} \one_{\Omega_{\kappa,m-1}^\complement} \norm{\bff{u}(t_m)- \bff{u}_h^m}{\bb{H}^1}^2 \right] \leq C_q \kappa^{-2^{q-1}}.
\end{align}
We now choose
\begin{equation*}
    \kappa^2 = \frac{1}{\widetilde{C}}\left( \abs{\log \left(h^2+k^{2\alpha}\right)} - (2^{q-1}-1) \log \abs{\log \left(h^2+k^{2\alpha}\right)} \right),
\end{equation*}
where $\widetilde{C}$ is the constant from~\eqref{equ:E error 1 implicit}. For sufficiently small $h$ and $k$, the right-hand side is positive. With this choice of $\kappa$, and using \eqref{equ:E Bn}, we infer from~\eqref{equ:E error 1 implicit} and \eqref{equ:E Bn} that
\begin{align*}
    \bb{E}\left[\max_{m\leq n} \norm{\bff{u}(t_m)-\bff{u}_h^m}{\bb{H}^1}^2 \right]
    &\leq
    \widetilde{C} e^{\widetilde{C}\kappa^2} \left(h^2+k^{2\alpha}\right)
    +
    C_q \kappa^{-2^{q-1}}
    \leq
    C_r \abs{\log \left(h^2+k^{2\alpha}\right)}^{-r},
\end{align*}
for any $r\geq 1$. This completes the proof of the theorem.
\end{proof}

\section{Numerical experiments}\label{sec:num exp}

We present a set of numerical experiments to validate the theoretical convergence properties of the proposed finite element schemes for the stochastic Landau--Lifshitz--Bloch (LLB) equation above the Curie temperature. All computations are carried out in the~\textsc{FEniCS} environment~\citep{AlnaesEtal15}. The computational domain is taken to be either the unit square or the unit interval. The magnetisation vector field is discretised in space using continuous piecewise linear finite elements on a family of quasi-uniform meshes. Time integration is performed using either the semi-implicit scheme \eqref{equ:euler} or the implicit scheme \eqref{equ:implicit euler}.

We remark that a $C^1$-conforming scheme \eqref{equ:2d euler} is included in our analysis in Section~\ref{sec:fem 2d} for completeness, since we improve upon the theoretical results of~\citep{GolJiaLe24}. However, due to its computational cost, especially in the stochastic setting, numerical experiments for this scheme will not be performed. Instead, we have developed a mixed scheme for a regularised fourth-order formulation of the LLB system for $d=2$ in a separate work~\citep{GolSoeTra24c}, which is more efficient, and numerical validation for that approach will be reported there. In the following, we will focus on verifying the order of convergence for schemes \eqref{equ:euler} and \eqref{equ:implicit euler}.

To assess convergence, we compute a reference solution on a very fine mesh and with a small time step. This reference solution serves as an approximation of the exact stochastic solution for each realisation of the Wiener process. For coarser discretisations, we use the same Brownian path so that the difference between the numerical solution and the reference is meaningful pathwise. The error $\mathcal{E}_s(h,k)$ at the final time $T$ is then measured in $L^2(\Omega; \bb{H}^s)$-norm for $s=0$ or $1$, defined by
\[
    \mathcal{E}_{s}(h,k):= \left(\bb{E}\norm{\bff{u}_h^N- \bff{u}_{\text{ref}}(T)}{\bb{H}^s}^2\right)^{\frac12}.
\]
Here, $\bff{u}_h^N$ denotes the numerical solution with mesh size $h$ at time $T=Nk$, and $\bff{u}_{\text{ref}}(T)$ is the reference solution at the same final time along the same realisation of the Brownian motion. In practice, the expectation is approximated by a Monte Carlo average over $M$ independent sample paths.

We vary separately the mesh size $h$ and the time step $k$ to test spatial and temporal convergence. To verify spatial convergence, we fix a sufficiently small time step and compare errors across a sequence of meshes ($h=2^{-j}$, for $j=2,3,4,5$, with $j=6$ as the reference solution). For temporal convergence, we fix a fine mesh and vary the time step size ($k=(25\times 2^j)^{-1}$, for $j=0,1,2, 3$, with $j=4$ as the reference solution). In both cases, errors are averaged over $M=50$ realisations. The experimental rate of convergence is obtained by fitting the errors against mesh size or time step in a log-log plot.

\subsection{Simulation 1}

Let the domain $\mathscr{D}=[0,1]$. We take the parameters to be $\kappa_1=0.2$, $\gamma=4.0$, $\kappa_2=0.5$, and $\mu=1.0$. The initial data is specified to be
\[
    \bff{u}_0(x)= \big(0, \cos(2\pi x), \sin(2\pi x) \big),
\]
and the vector field $\bff{g}$ is taken to be
\[
    \bff{g}(x)= \big(0.2, 0.1(1+x), 0\big).
\]
We solve the stochastic LLB equation by employing the implicit scheme~\eqref{equ:implicit euler}.

Snapshots of the magnetisation vector field $\bff{u}$ with mesh-size $h=1/16$ at selected times are shown in Figure~\ref{fig:snapshots u 1d 1}. The colour indicates the relative value of the magnitude. Figures~\ref{fig:order u spatial 1} and \ref{fig:order u time 1} show plots of $\mathcal{E}_s$ {at the final time $T=0.2$} versus $1/h$ and $1/k$, respectively. These results are consistent with Corollary~\ref{cor:error subset implicit} in the $\bb{H}^1$ norm, while the numerical simulation indicates an even higher convergence rate in the $\bb{L}^2$ norm. Such behavior is in line with what is traditionally expected in finite element analysis, though a rigorous proof in our setting remains an interesting open question.

Figure~\ref{fig:mass energy exp1} shows the energy of the system over 30 independent sample paths for $h=1/32$, $k=0.01$, and for $h=1/64$, $k=0.005$. Here, the energy is defined as
\[
    \text{Energy}(\bff{u}):= \frac{\kappa_1}{2} \norm{\nabla \bff{u}}{\bb{L}^2}^2 + \frac{\kappa_2}{2} \norm{\bff{u}}{\bb{L}^2}^2 + \frac{\kappa_2 \mu}{4} \norm{\bff{u}}{\bb{L}^4}^4.
\]
For relatively small noise intensity, the energy shows only minor pathwise fluctuations and, on average, decays over time (albeit not necessarily to zero).

\begin{figure}[!htb]
	\centering
	\begin{subfigure}[b]{0.41\textwidth}
		\centering
		\includegraphics[width=\textwidth]{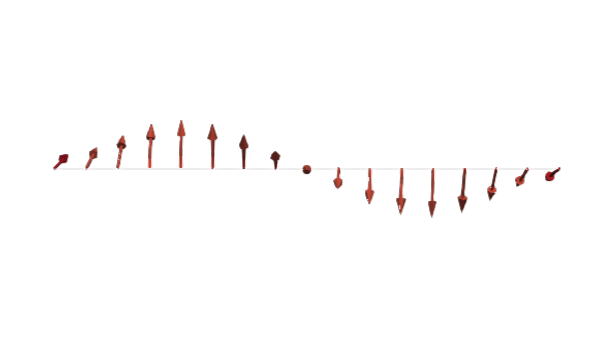}
		\caption{$t=0$}
	\end{subfigure}
	\begin{subfigure}[b]{0.41\textwidth}
		\centering
		\includegraphics[width=\textwidth]{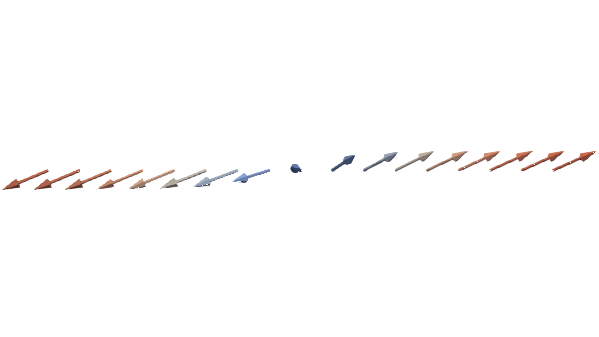}
		\caption{$t=0.15$}
	\end{subfigure}
    \begin{subfigure}[b]{0.09\textwidth}
		\centering
		\includegraphics[width=\textwidth]{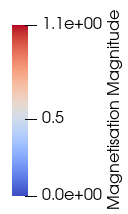}
	\end{subfigure}
	\begin{subfigure}[b]{0.41\textwidth}
		\centering
	\includegraphics[width=\textwidth]{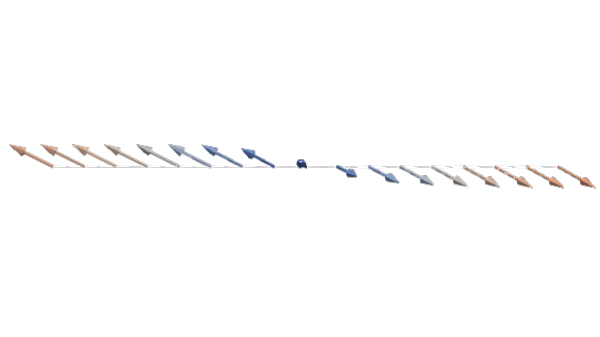}
		\caption{$t=0.25$}
	\end{subfigure}
    \begin{subfigure}[b]{0.41\textwidth}
		\centering
	\includegraphics[width=\textwidth]{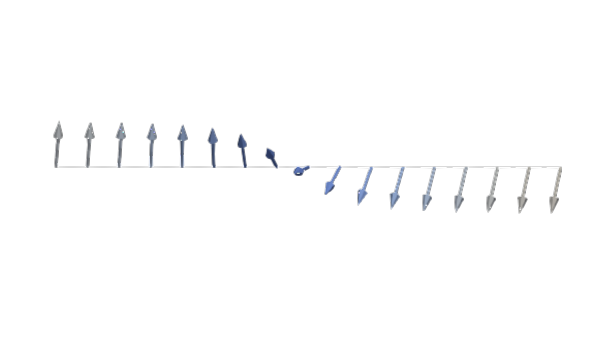}
		\caption{$t=0.4$}
	\end{subfigure}
	\begin{subfigure}[b]{0.09\textwidth}
		\centering
		\includegraphics[width=\textwidth]{exp1_legend.png}
	\end{subfigure}
	\caption{Snapshots of a sample path of the magnetisation $\bff{u}$ in simulation 1 at selected times.}
	\label{fig:snapshots u 1d 1}
\end{figure}

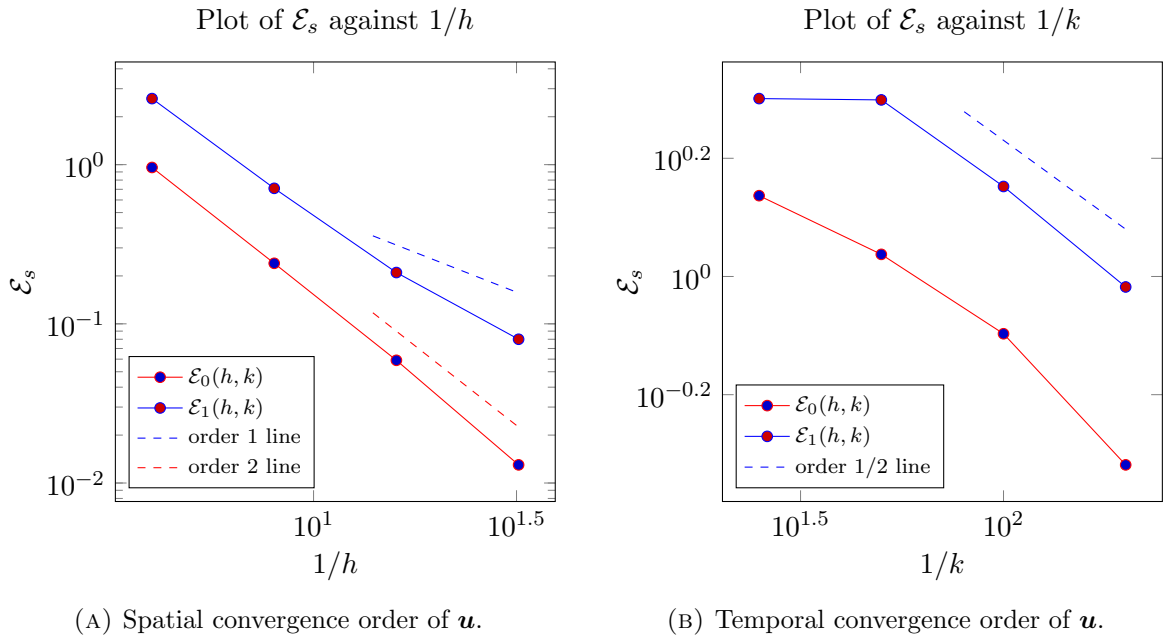
\begin{figure}[!htb]
	\begin{subfigure}[b]{0.45\textwidth}
		\centering
		\begin{tikzpicture}
			\begin{axis}[
				title=Plot of $\mathcal{E}_s$ against $1/h$,
				height=1\textwidth,
				width=1\textwidth,
				xlabel= $1/h$,
				ylabel= $\mathcal{E}_s$,
				xmode=log,
				ymode=log,
				legend pos=south west,
				legend cell align=left,
				]
				\addplot+[mark=*,red] coordinates {(4,0.96)(8,0.24)(16,0.059)(32,0.013)};
                \addplot+[mark=*,blue] coordinates {(4,2.6)(8,0.71)(16,0.21)(32,0.08)};
				\addplot+[dashed,no marks,blue,domain=14:32]{5/x};
				\addplot+[dashed,no marks,red,domain=14:32]{23/x^2};
			     \legend{\scriptsize{$\mathcal{E}_0(h,k)$}, \scriptsize{$\mathcal{E}_1(h,k)$}, \scriptsize{order 1 line}, \scriptsize{order 2 line}}
			\end{axis}
		\end{tikzpicture}
		\caption{Spatial convergence order of $\bff{u}$.}
		\label{fig:order u spatial 1}
	\end{subfigure}
    \hspace{1em}
	\begin{subfigure}[b]{0.45\textwidth}
		\centering
		\begin{tikzpicture}
			\begin{axis}[
				title=Plot of $\mathcal{E}_s$ against $1/k$,
				height=1\textwidth,
				width=1\textwidth,
				xlabel= $1/k$,
				ylabel= $\mathcal{E}_s$,
				xmode=log,
				ymode=log,
				legend pos=south west,
				legend cell align=left,
				]
				\addplot+[mark=*,red] coordinates {(25,1.37)(50,1.09)(100,0.8)(200,0.48)};
                \addplot+[mark=*,blue] coordinates {(25,2.0)(50,1.99)(100,1.42)(200,0.96)};
				\addplot+[dashed,no marks,blue,domain=80:200]{17/sqrt(x)};
			\legend{\scriptsize{$\mathcal{E}_0(h,k)$}, \scriptsize{$\mathcal{E}_1(h,k)$}, \scriptsize{order 1/2 line}}
			\end{axis}
		\end{tikzpicture}
		\caption{Temporal convergence order of $\bff{u}$.}
		\label{fig:order u time 1}
	\end{subfigure}
    \caption{Convergence orders of the magnetisation vector field in simulation 1.}
\end{figure}

\begin{figure}[!htb]
	\centering
	\begin{subfigure}[b]{0.47\textwidth}
		\centering
		\includegraphics[width=\textwidth]{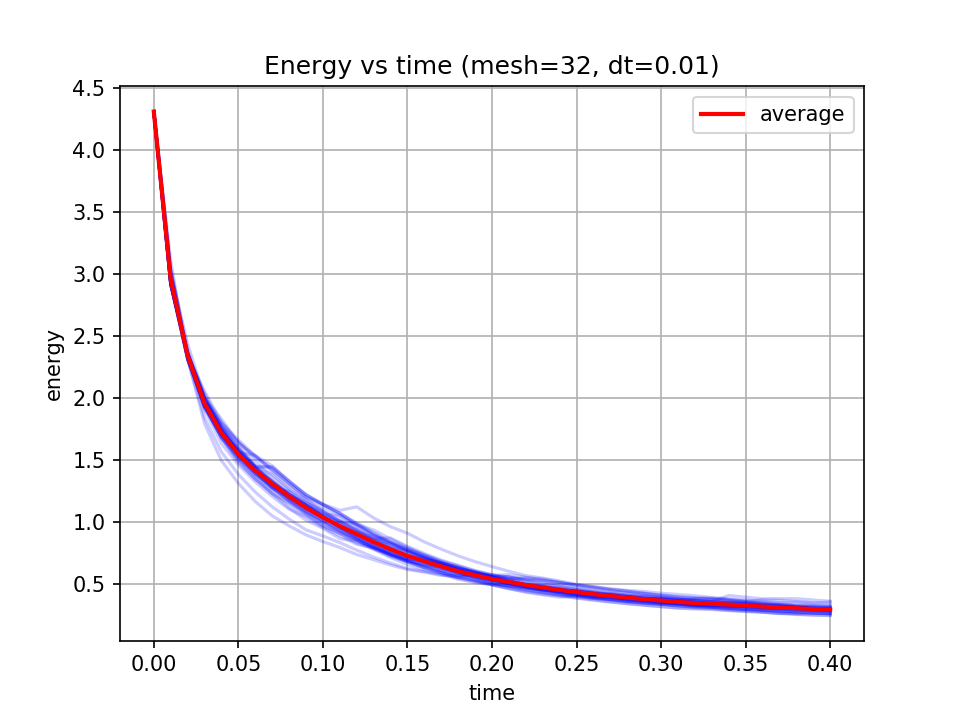}
		\caption{Graph of energy vs time with $h=1/32$ and $k=1/100$ for 30 sample paths.}
	\end{subfigure}
    \hspace{1em}
	\begin{subfigure}[b]{0.47\textwidth}
		\centering
		\includegraphics[width=\textwidth]{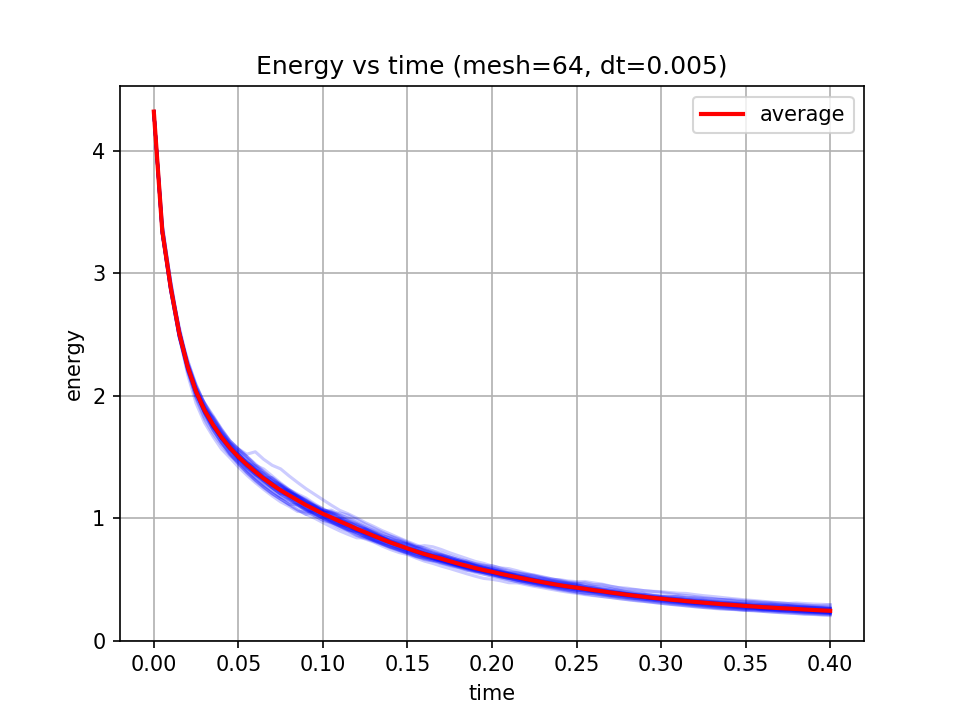}
		\caption{Graph of energy vs time with $h=1/64$ and $k=1/200$ for 30 sample paths.}
	\end{subfigure}
        \caption{Energy evolution in simulation 1}
	\label{fig:mass energy exp1}
\end{figure}

\subsection{Simulation 2}

Let the domain $\mathscr{D}=[0,1]\times [0,1]$. We take the parameters to be $\kappa_1=0.5$, $\gamma=2.0$, $\kappa_2=0.2$, and $\mu=1.0$. The initial data is specified to be
\[
    \bff{u}_0(x,y)= \big(y, -x, 0 \big),
\]
and the vector field $\bff{g}$ is taken to be
\[
    \bff{g}(x,y)= \big(0.5(1+x), 0.25(1+y), 0\big).
\]
We solve the stochastic LLB equation by employing the semi-implicit scheme~\eqref{equ:euler}. Although the case $d=2$ is not covered by the error analysis of this scheme, we include a numerical experiment to indicate that the method appears to exhibit the expected order of convergence.

Snapshots of the magnetisation vector field $\bff{u}$ with mesh-size $h=1/16$ at selected times are shown in Figure~\ref{fig:snapshots u 2}. The colour indicates the relative value of the magnitude. Plots of $\mathcal{E}_s$ {at the final time $T=0.2$} as functions of $1/h$ and $1/k$ are presented in Figures~\ref{fig:order u spatial 2} and
\ref{fig:order u time 2}, respectively.
 Figure~\ref{fig:mass energy exp2} shows the energy of the system over 30 independent sample paths for $h=1/32$, $k=0.01$, and for $h=1/64$, $k=0.005$.  Since the noise intensity is moderately large, the energy trajectory exhibits pronounced pathwise variability around the mean profile.

\begin{figure}[!htb]
	\centering
	\begin{subfigure}[b]{0.21\textwidth}
		\centering
		\includegraphics[width=\textwidth]{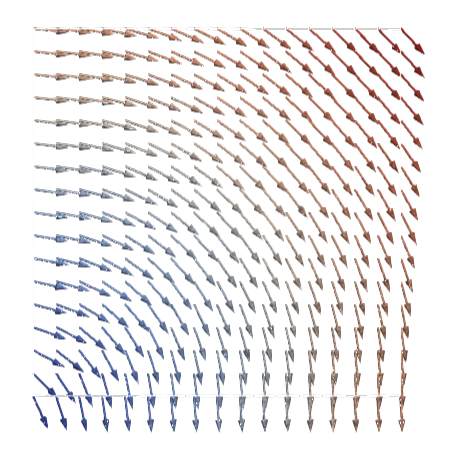}
		\caption{$t=0$}
	\end{subfigure}
	\begin{subfigure}[b]{0.21\textwidth}
		\centering
		\includegraphics[width=\textwidth]{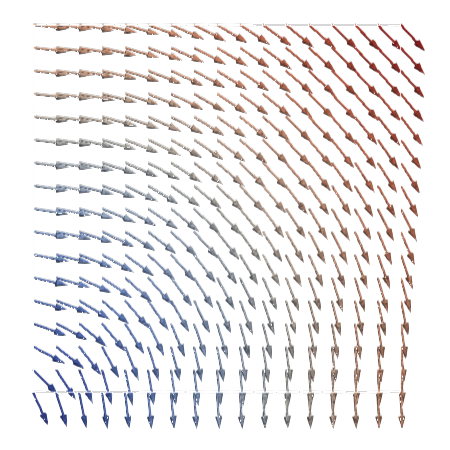}
		\caption{$t=0.1$}
	\end{subfigure}
	\begin{subfigure}[b]{0.21\textwidth}
		\centering
	\includegraphics[width=\textwidth]{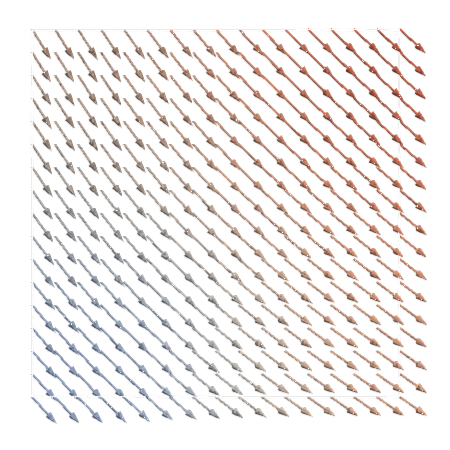}
		\caption{$t=0.15$}
	\end{subfigure}
    \begin{subfigure}[b]{0.21\textwidth}
		\centering
	\includegraphics[width=\textwidth]{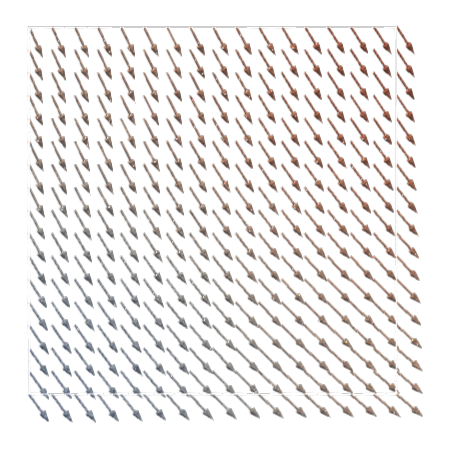}
		\caption{$t=0.2$}
	\end{subfigure}
	\begin{subfigure}[b]{0.08\textwidth}
		\centering
		\includegraphics[width=\textwidth]{exp1_legend.png}
	\end{subfigure}
	\caption{Snapshots of a sample path of the magnetisation $\bff{u}$ in simulation 2 at selected times.}
	\label{fig:snapshots u 2}
\end{figure}

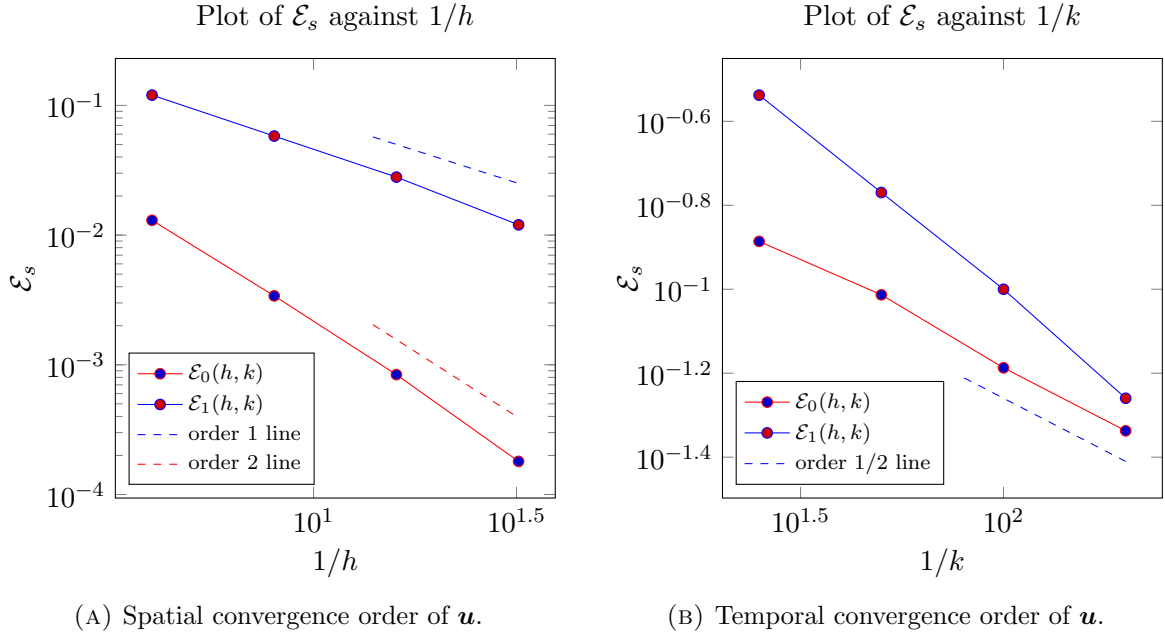
\begin{figure}[!htb]
	\begin{subfigure}[b]{0.45\textwidth}
		\centering
		\begin{tikzpicture}
			\begin{axis}[
				title=Plot of $\mathcal{E}_s$ against $1/h$,
				height=1\textwidth,
				width=1\textwidth,
				xlabel= $1/h$,
				ylabel= $\mathcal{E}_s$,
				xmode=log,
				ymode=log,
				legend pos=south west,
				legend cell align=left,
				]
				\addplot+[mark=*,red] coordinates {(4,0.013)(8,0.0034)(16,0.00084)(32,0.00018)};
                \addplot+[mark=*,blue] coordinates {(4,0.12)(8,0.058)(16,0.028)(32,0.012)};
				\addplot+[dashed,no marks,blue,domain=14:32]{0.8/x};
				\addplot+[dashed,no marks,red,domain=14:32]{0.4/x^2};
			     \legend{\scriptsize{$\mathcal{E}_0(h,k)$}, \scriptsize{$\mathcal{E}_1(h,k)$}, \scriptsize{order 1 line}, \scriptsize{order 2 line}}
			\end{axis}
		\end{tikzpicture}
		\caption{Spatial convergence order of $\bff{u}$.}
		\label{fig:order u spatial 2}
	\end{subfigure}
    \hspace{1em}
	\begin{subfigure}[b]{0.45\textwidth}
		\centering
		\begin{tikzpicture}
			\begin{axis}[
				title=Plot of $\mathcal{E}_s$ against $1/k$,
				height=1\textwidth,
				width=1\textwidth,
				xlabel= $1/k$,
				ylabel= $\mathcal{E}_s$,
				xmode=log,
				ymode=log,
				legend pos=south west,
				legend cell align=left,
				]
				\addplot+[mark=*,red] coordinates {(25,0.13)(50,0.097)(100,0.065)(200,0.046)};
                \addplot+[mark=*,blue] coordinates {(25,0.29)(50,0.17)(100,0.1)(200,0.055)};
				\addplot+[dashed,no marks,blue,domain=80:200]{0.55/sqrt(x)};
			\legend{\scriptsize{$\mathcal{E}_0(h,k)$}, \scriptsize{$\mathcal{E}_1(h,k)$}, \scriptsize{order 1/2 line}}
			\end{axis}
		\end{tikzpicture}
		\caption{Temporal convergence order of $\bff{u}$.}
		\label{fig:order u time 2}
	\end{subfigure}
    \caption{Convergence orders of the magnetisation vector field in simulation 2.}
\end{figure}

\begin{figure}[!htb]
	\centering
	\begin{subfigure}[b]{0.47\textwidth}
		\centering
		\includegraphics[width=\textwidth]{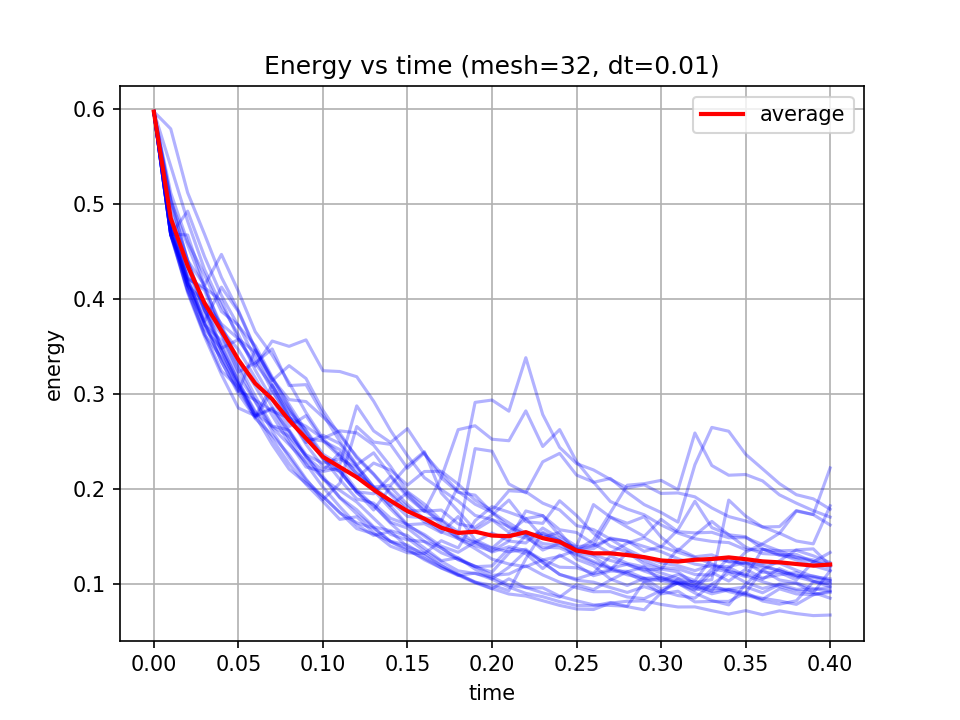}
		\caption{Graph of energy vs time with $h=1/32$ and $k=1/100$ for 30 sample paths.}
	\end{subfigure}
    \hspace{1em}
	\begin{subfigure}[b]{0.47\textwidth}
		\centering
		\includegraphics[width=\textwidth]{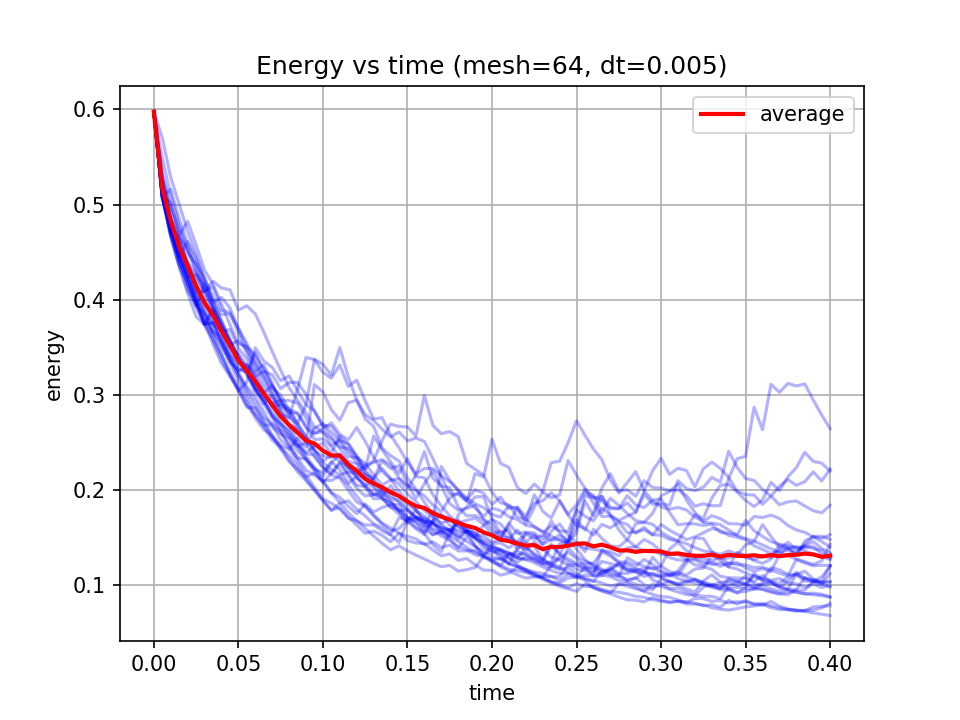}
		\caption{Graph of energy vs time with $h=1/64$ and $k=1/200$ for 30 sample paths.}
	\end{subfigure}
        \caption{Energy evolution in simulation 2}
	\label{fig:mass energy exp2}
\end{figure}

{
\subsection{Simulation 3}

This simulation is designed to verify the exponential stability established in Proposition~\ref{pro:eps zero}.
Let the spatial domain be $\mathscr{D}=[0,1]$, and let all physical and numerical parameters be chosen as in Simulation~1. We consider two numerical solutions $\bff{v}$ and $\bff{w}$ of the stochastic LLB equation, driven by the same Wiener process but with different initial data $\bff{v}_0,\bff{w}_0\in\bb{H}^1$. Specifically, we take
\begin{align*}
    \bff{v}_0(x) &= \big(0, \cos(2\pi x), \sin(2\pi x) \big),
    \\
    \bff{w}_0(x) &= \big(\epsilon, \cos(2\pi x), \sin(2\pi x) \big),
\end{align*}
where the parameter $\epsilon>0$ measures the size of the initial perturbation. Throughout this experiment, we fix the spatial mesh size $h=1/32$ and the time step $k=1/100$, and employ the scheme~\eqref{equ:euler}.

We first set $\epsilon=10^{-2}$ and choose the noise coefficient
\[
    \bff{g}(x)= \big(0.2, 0.1(1+x), 0\big),
\]
corresponding to a small-noise regime.
Figure~\ref{fig:exp small noise} shows $30$ sample paths of $\norm{\bff{v}(t)-\bff{w}(t)}{\bb{L}^2}^2$ as functions of time $t\in [0,2]$, together with their Monte Carlo average. After a short transient, the averaged difference exhibits a clear exponential decay.
Furthermore, Figure~\ref{fig:exp small vary eps} displays the Monte Carlo mean $\bb{E}\left[\norm{\bff{v}(t)-\bff{w}(t)}{\bb{L}^2}^2\right]$ on a semi-logarithmic scale for $\epsilon=10^{-1},10^{-2},10^{-3}$. The curves differ only by a vertical shift while sharing the same asymptotic decay rate, thus providing a numerical confirmation of the exponential stability estimate in Proposition~\ref{pro:eps zero}.

Next, we repeat the same experiment with increased noise amplitudes in order to explore the behaviour outside the small-noise regime. Figure~\ref{fig:exp mod vary eps} corresponds to the moderately large noise coefficient
\[
    \bff{g}(x)= \big(0.5, 0.3(1+x), 0\big),
\]
while Figure~\ref{fig:exp large vary eps} corresponds to the significantly larger choice
\[
    \bff{g}(x)= \big(2.0, 1.5(1+x), 0\big).
\]
In these cases, the exponential decay observed in the small-noise regime is no longer present. For moderately large noise, the Monte Carlo mean exhibits a mild growth, while for large noise the curves become highly irregular and do not display any clear contraction. This may indicate a loss of the contractive mechanism underlying Proposition~\ref{pro:eps zero}.

\begin{figure}[!htb]
	\centering
	\begin{subfigure}[b]{0.47\textwidth}
		\centering
		\includegraphics[width=\textwidth]{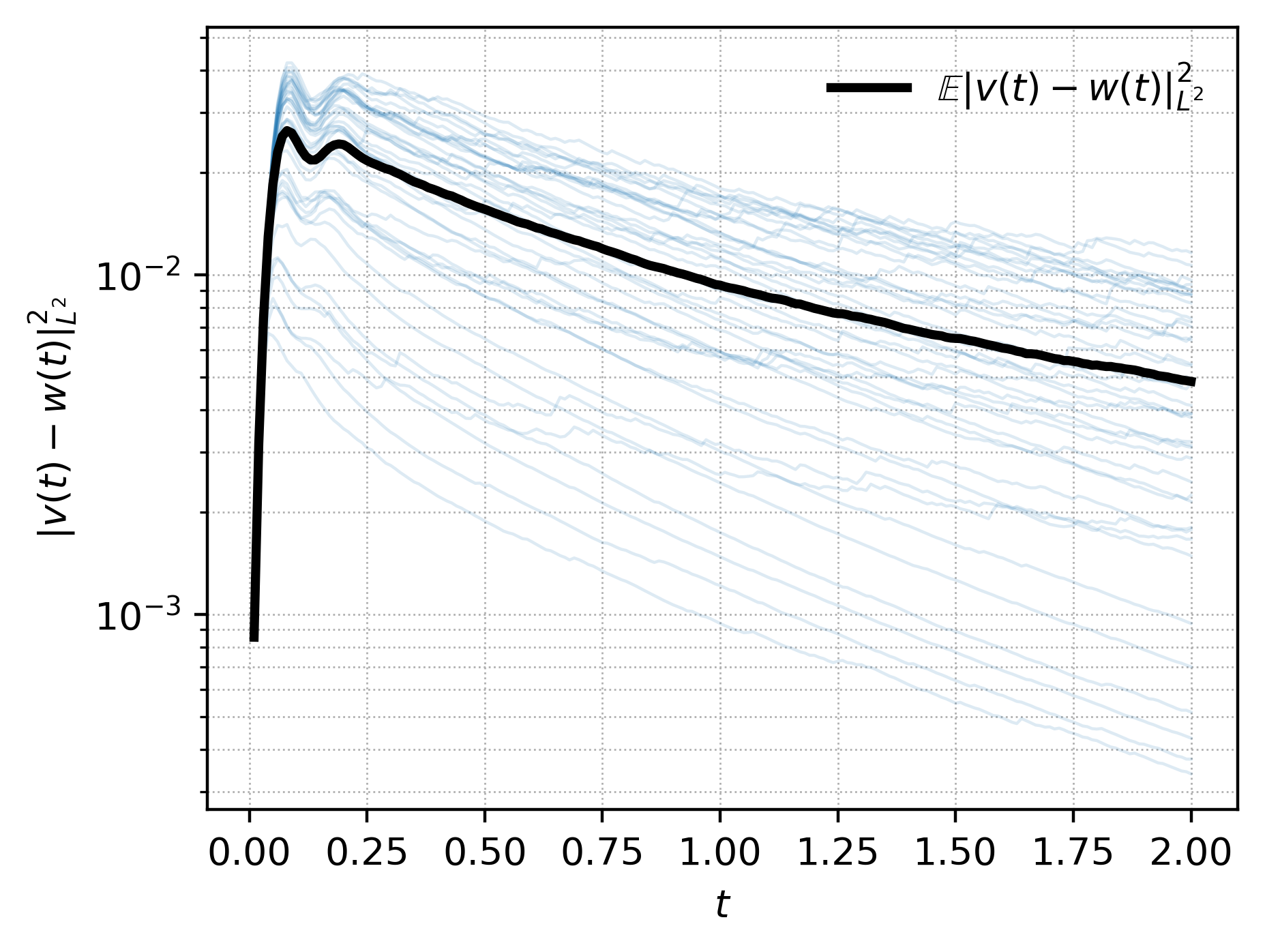}
		\caption{Graph of $\norm{\bff{v}(t)-\bff{w}(t)}{\bb{L}^2}^2$ against $t$ for 30 sample paths and its average.}
        \label{fig:exp small noise}
	\end{subfigure}
    \hspace{1em}
	\begin{subfigure}[b]{0.47\textwidth}
		\centering
		\includegraphics[width=\textwidth]{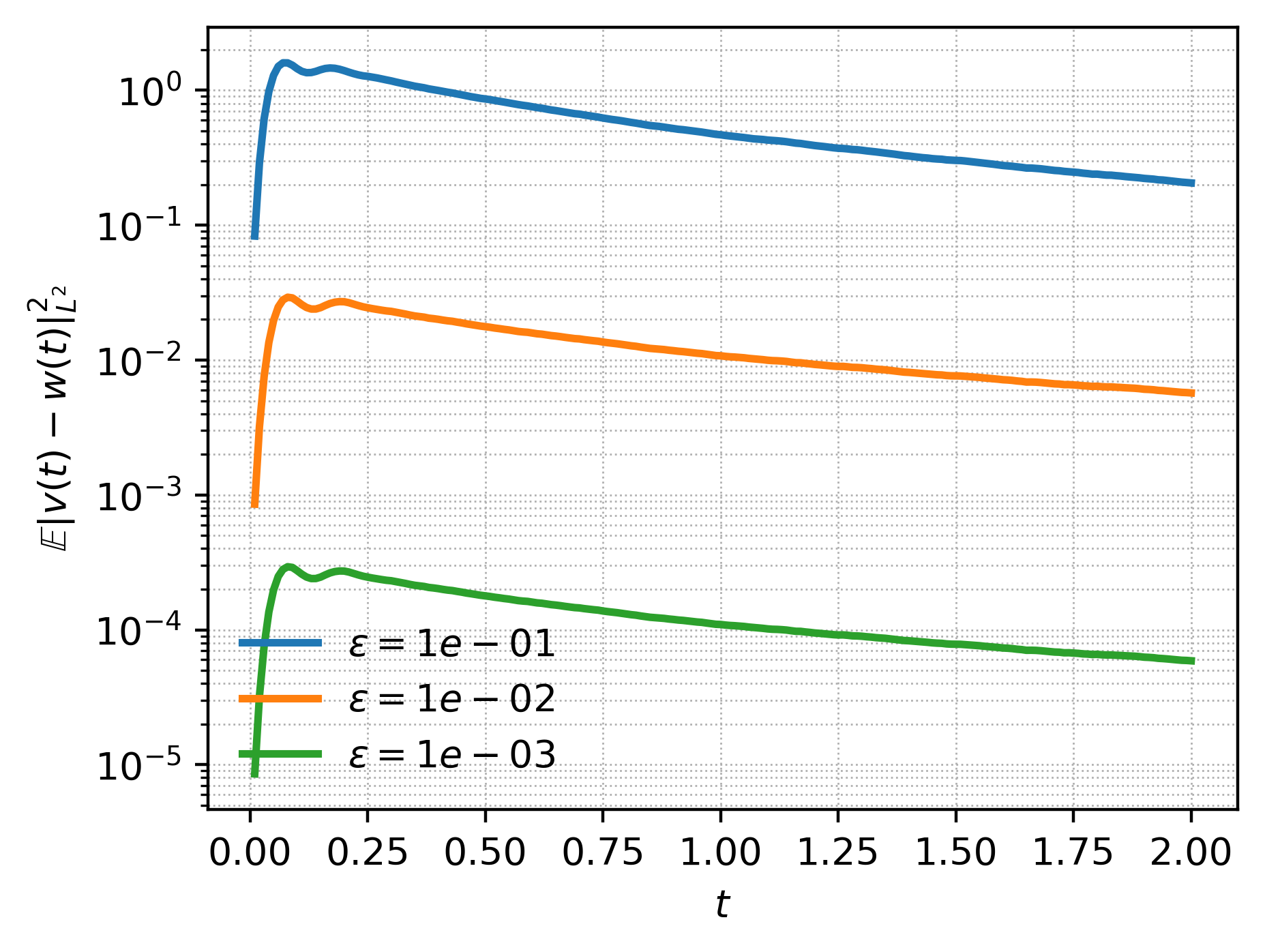}
		\caption{Graph of $\bb{E}\big[\norm{\bff{v}(t)-\bff{w}(t)}{\bb{L}^2}^2\big]$ against $t$ for various values of $\epsilon$ with a small noise.}
        \label{fig:exp small vary eps}
	\end{subfigure}
        \caption{Evolution of $\bb{E}\big[\norm{\bff{v}(t)-\bff{w}(t)}{\bb{L}^2}^2\big]$ in simulation 3 for a small noise.}
\end{figure}

\begin{figure}[!htb]
	\centering
	\begin{subfigure}[b]{0.47\textwidth}
		\centering
		\includegraphics[width=\textwidth]{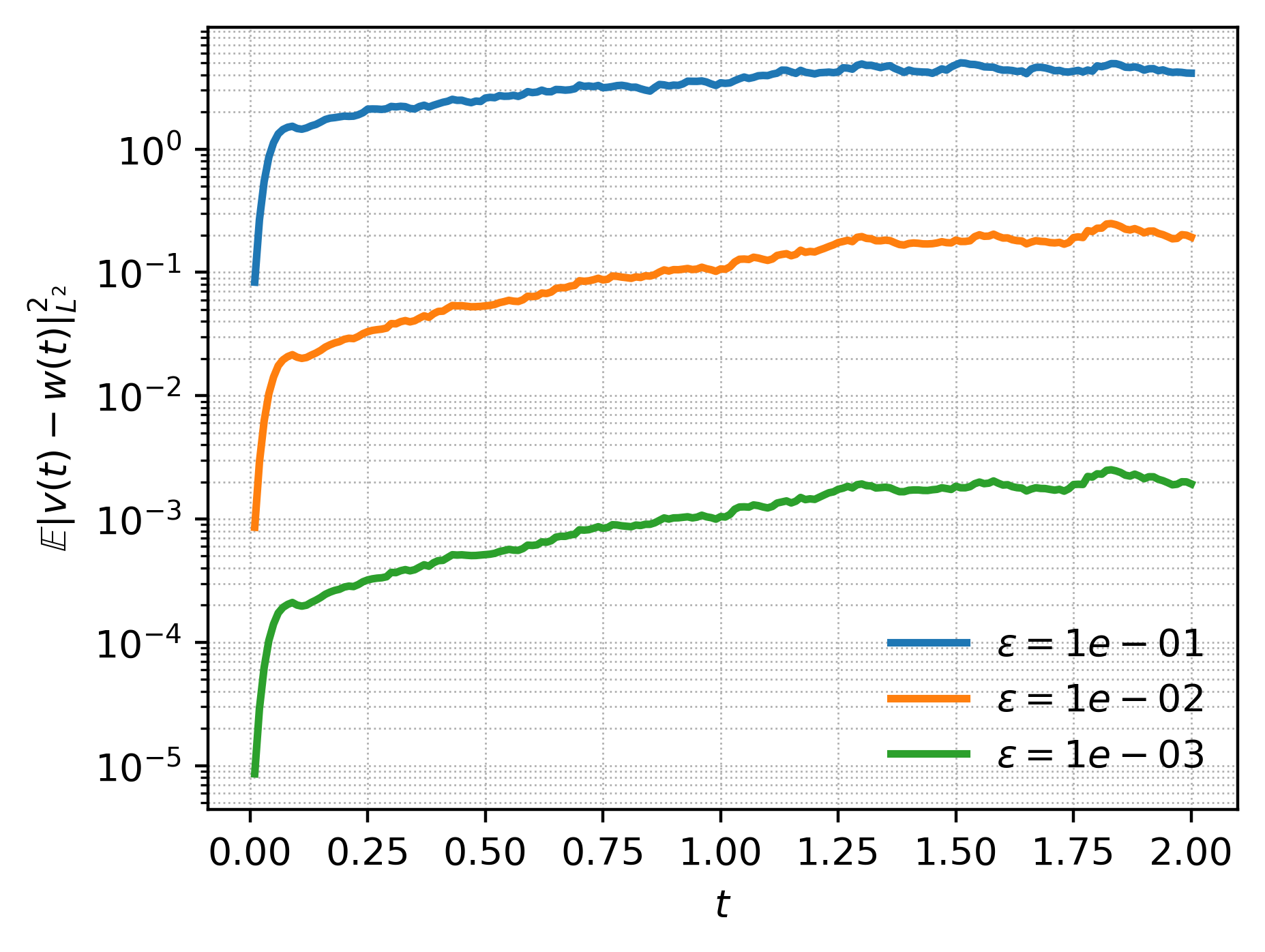}
		\caption{Graph of $\bb{E}\big[\norm{\bff{v}(t)-\bff{w}(t)}{\bb{L}^2}^2\big]$ against $t$ for various values of $\epsilon$ with a moderately large noise.}
        \label{fig:exp mod vary eps}
	\end{subfigure}
    \hspace{1em}
	\begin{subfigure}[b]{0.47\textwidth}
		\centering
		\includegraphics[width=\textwidth]{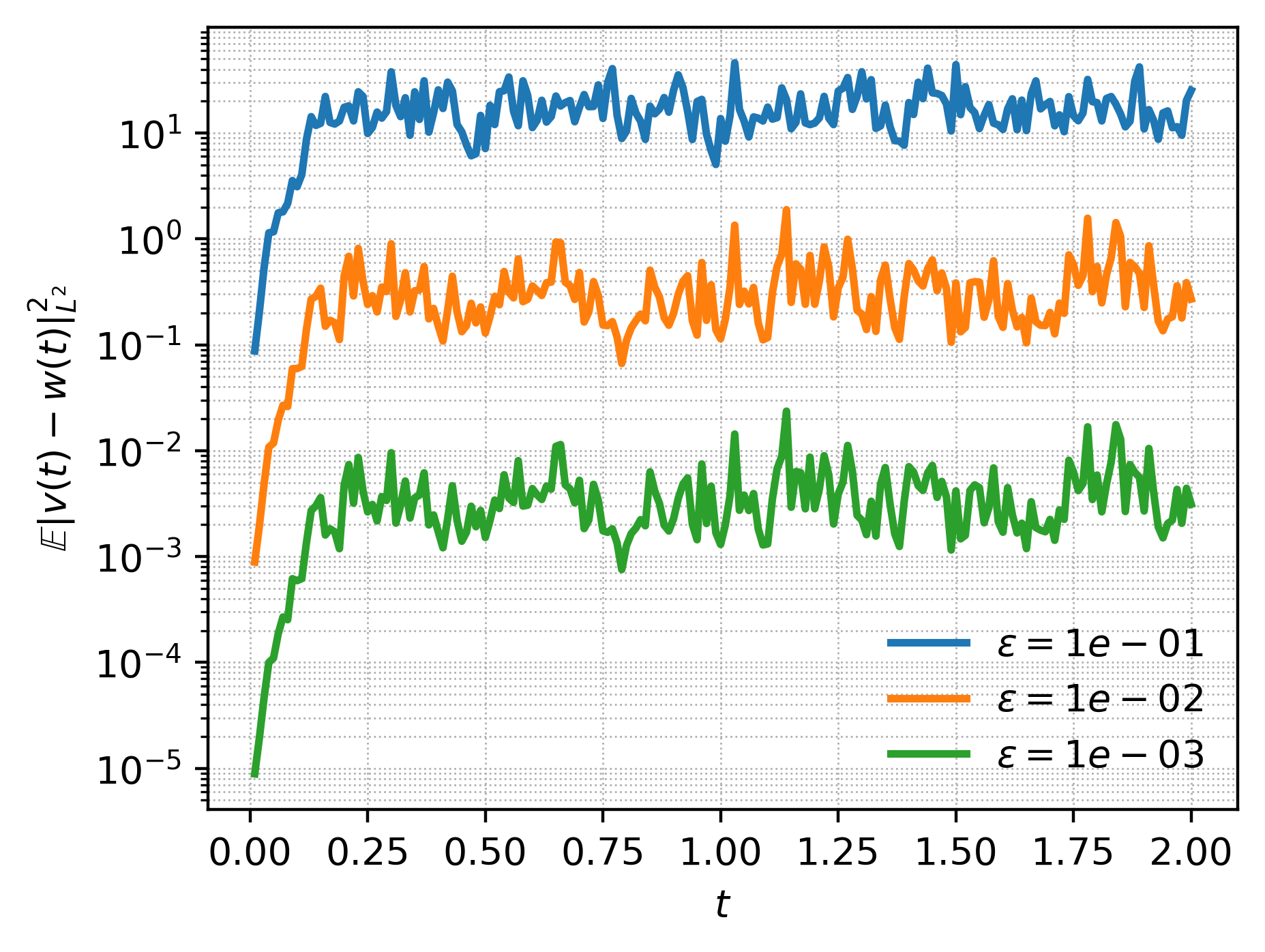}
		\caption{Graph of $\bb{E}\big[\norm{\bff{v}(t)-\bff{w}(t)}{\bb{L}^2}^2\big]$ against $t$ for various values of $\epsilon$ with a significantly large noise.}
        \label{fig:exp large vary eps}
	\end{subfigure}
        \caption{Evolution of $\bb{E}\big[\norm{\bff{v}(t)-\bff{w}(t)}{\bb{L}^2}^2\big]$ in simulation 3 for a moderate and a large noise.}
\end{figure}
}

\section*{Acknowledgements}
The author is supported by the Australian Government Research Training Program (RTP) Scholarship awarded at the University of New South Wales, Sydney. 
Part of this work was carried out during a research visit to the Institute for Analysis and Scientific Computing at TU Wien, whose hospitality the author gratefully acknowledges.
{The author also wishes to thank the anonymous referees for their careful reading of the manuscript and for their valuable suggestions.}



\end{document}